%% file: ManyAnosovFlows.tex
\documentclass[11pt,a4paper]{amsart}

\usepackage{amssymb,amsfonts,amsthm,array,epsfig,graphics,graphicx,tabularx,mathabx,mathbbol}
\usepackage[usenames,dvipsnames]{color}
\usepackage[margin=3cm]{geometry}
\usepackage{enumitem}
\setlist{leftmargin=25pt}

 \def\NN{{\mathbb N}}  
 \def\RR{{\mathbb R}} \def\SS{{\mathbb S}} \def\TT{{\mathbb T}}
   
 \def\ZZ{{\mathbb Z}}
\newcommand{\ff}{{\mathbb f}}

\def\cC{{\mathcal C}}   \def\cO{{\mathcal O}}

\def\cF{{\mathcal F}}  \def\cL{{\mathcal L}}

\def\Wi{\widetilde}

\newtheorem{theorem}{{Theorem}}[section]
\newtheorem{proposition}[theorem]{{Proposition}}

\newtheorem{lemma}[theorem]{{Lemma}}

\newtheorem{corollary}[theorem]{{Corollary}}

\newtheorem{fact}[theorem]{{Fact}}

\newtheorem{claim}[theorem]{{Claim}}

\theoremstyle{definition}
\newtheorem{definition}[theorem]{{Definition}}

\theoremstyle{remark}
\newtheorem{remark}[theorem]{{Remark}}

\newtheorem{example}[theorem]{{Example}}

\title[Non-$\RR$-covered Anosov flows on hyperbolic $3$-manifolds]{Existence of arbitrary large numbers of non-$\RR$-covered Anosov flows on hyperbolic $3$-manifolds}
\author{Fran\c cois B\'eguin and Bin Yu}
\date{\today}

\begin{document}

\begin{abstract}
The purpose of this paper is to prove that, for every $n\in \NN$, there exists a closed hyperbolic $3$-manifold $M$ which carries at least  $n$ non-$\RR$-covered Anosov flows, that are pairwise orbitally inequivalent. Due to a recent result by Fenley, such Anosov flows are quasi-geodesic. Hence, we get the existence of hyperbolic $3$-manifolds carrying many pairwise orbitally inequivalent quasi-geodesic Anosov flows. One of the main ingredients of our proof is a description of the clusters of lozenges that appear in the orbit spaces of the Anosov flows that we construct. The number of lozenges involved in each cluster, as well as the \emph{orientation-type} of this cluster, provide powerful dynamical invariants allowing to prove that the flows are orbitally inequivalent. We believe that these dynamical invariants could be helpful in a much wider context.
\end{abstract}

\maketitle

\section{Introduction}
\label{s.int}
\input{Introduction.tex}

\section{Construction of a hyperbolic plug}
\label{s.hplug}
\input{Hyperbolic_plug.tex}

\section{Construction of Anosov flows on a toroidal manifold}
\label{s.con}
\input{Gluing-diffeomorphisms.tex}

\section{Constructing non-$\RR$-covered Anosov flows on a hyperbolic manifold}
\label{s.conh}
\input{Hyperbolic-manifold.tex}

\section{Lozenges of the flow $Z_t^m$}
\label{s.EAP}
\input{Lozenges.tex}

\section{Separatrix-adjacent Birkhoff annuli}
\label{s.SAa}
\input{Separatrix-adjacent-Birkhoff-annuli.tex}

\section{The proof of the main theorem: Theorem \ref{t.main}}
\label{s.dif}
\input{Proof-main-theorem.tex}

\vskip 1cm

\noindent Fran\c cois B\'eguin\\
{\small LAGA\\UMR 7539 du CNRS}\\
{\small Universit\'e Paris 13, 93430 Villetaneuse, FRANCE}\\
{\footnotesize{E-mail: beguin@math.univ-paris13.fr}}

\vskip 2mm

\noindent Bin Yu\\
{\small School of Mathematical Sciences,}\\
{\small Key Laboratory of Intelligent Computing and Applications(Ministry of Education), }\\
{\small Tongji University, Shanghai 200092, CHINA}\\
\noindent{\footnotesize{E-mail: binyu1980@gmail.com }}

\end{document}

%% file: Introduction.tex
The flow $X_t$ of a nonsingular $C^1$ vector field $X$ on a closed Riemannian manifold $M$ is called an \emph{Anosov flow} if the manifold $M$ is a hyperbolic set for the flow $X_t$, that is if there exists an $X$-invariant splitting $TM =E^s \oplus \RR X\oplus E^u$ and some constant $C>0$, $\lambda >0$
such that:
 \begin{eqnarray*}
 \|D X_t (v)\| & \leq & C e^{-\lambda t} \|v\|, \mbox{ for any } v\in E^s \mbox{, } t\geq 0;\\
 \|D X_{-t} (v)\| & \leq & C e^{-\lambda t} \|v\|, \mbox{ for any } v\in E^u \mbox{, } t\geq 0.
 \end{eqnarray*}
Due to the definition, $3$ is the lowest possible dimension for manifolds to carry Anosov flows. The dynamics of an Anosov flow is uniformly chaotic, yet structurally stable : if a vector field $X$ on a closed manifold $M$ generates an Anosov flow $X_t$, then there exists a $C^1$-neighbourhood $\mathcal{U}$ of $X$, so that every vector field $Y$ in $\mathcal{U}$ generates an Anosov flow which is orbitally equivalent to $X$. In our context, two flows $X_t$ and $Y_t$ on a manifold $M$ are said to be \emph{orbitally equivalent} if there exists a self-homeomorphism of $M$ mapping the oriented orbits of $X_t$ to the oriented orbits of $Y_t$. 

Understanding the qualitative dynamics (that is the dynamics up to orbital equivalence) of $3$-dimensional Anosov flows is an important mathematical challenge, which builds bridges between different fields: dynamical system, geometry, topology, \emph{etc.}  See for instance \cite{Fen1}, \cite{Fen2}, \cite{Fen3}, \cite{Ba}, \cite{BaFe}, \cite{BarFe2}, \cite{BI}, \cite{BM}. One of the main questions in the area is to count the number of Anosov flows (up to orbit-equivalence) on a given $3$-manifold. In the 90s, Barbot found a first example of a manifold carrying two non-orbitally equivalent Anosov flows (\cite{Ba1}). In \cite{BBY}, Bonatti and the authors of the present paper showed that, for every $n\in \NN$, there exists a closed $3$-manifold $M$ which carries at least $n$ pairwise non-orbitally equivalent Anosov flows. Due to the construction procedure, all the manifolds exhibited in \cite{BBY} are toroidal. So, it was natural to ask if a similar results also held for hyperbolic $3$-manifolds. In fact, this is Problem 3.53 (C) in Kirby’s problem list \cite{Kir}, which was attributed to Christy. One should note that the problem is harder in the realm of hyperbolic $3$-manifolds, mostly because that, in toroidal case, the incompressible tori greatly help to distinguish Anosov flows from one another. In \cite{BM}, Bowden and Mann constructed hyperbolic 3-manifolds that admit arbitrarily many orbitally inequivalent Anosov flows, solving Kirby's Problem 3.53 (C). 

\subsection{Main result} \label{ss.mainresult}

The Anosov flows constructed by Bowden and Mann are $\RR$-covered. An Anosov flow on a closed $3$-manifold $M$ is said to be $\RR$-covered if the leaf space of weak stable foliation of the flow lifted to the universal cover is a Hausdorff space (if it is the case, a similar property holds for the weak unstable foliation, and both orbit spaces are homeomorphic to the real line). In the realm of Anosov flows on $3$-manifolds, $\RR$-covered case and non-$\RR$-covered case are quite different. See for instance \cite{Fen1}, \cite{Fen2}, \cite{Fen3}, \cite{Ba}. This raised the question of the existence  of hyperbolic 3-manifolds carrying arbitrarily many orbitally inequivalent non-$\RR$-covered Anosov flows. The purpose of this paper is devoted to construct such manifold and flows. More precisely, the main result of this paper is the following.

\begin{theorem}\label{t.main}
For every $n\in \NN$, there exists a closed hyperbolic $3$-manifold $M$ which carries at least  $n$ non-$\RR$-covered transitive Anosov flows which  are pairwise orbitally inequivalent.
\end{theorem}

Recently, Fenley (\cite{Fen3}) proved a very beautiful theorem that states that every non-$\RR$-covered  Anosov flow on a closed hyperbolic $3$-manifold must be a quasigeodesic flow. This result is quite amazing since no example of quasigeodesic Anosov flows on closed hyperbolic $3$-manifolds was known before. Combining Fenley's result and Theorem \ref{t.main}, we immediately get the following consequence.

\begin{corollary}\label{c.c}
For every $n\in \NN$, there exists a closed hyperbolic $3$-manifold $M$ which carries at least  $n$ quasigeodesic Anosov flows which are pairwise orbitally inequivalent.
\end{corollary}

\subsection{Outline of the proof of Theorem \ref{t.main}} \label{ss.sproof}
The proof of Theorem \ref{t.main} naturally splits in two parts. In the first part, given a positive integer $n$, we construct  a closed hyperbolic $3$-manifold and $2n+1$ transitive non-$\RR$-covered Anosov flows on this manifold. In the second part, we prove that $2n-1$ out of these $2n+1$ flows are pairwise orbitally inequivalent. In order to distinguish orbital equivalence classes of Anosov flows, we introduce a new dynamical invariant using chains of fundamental Birkhoff annuli. This dynamical invariant is one of the main novelties of our paper, and could probably be useful in many other situations.

\medskip

Let $n$ be a positive integer. The construction part roughly goes as follows:

We first consider a particular pseudo-Anosov diffeomorphism on a closed surface. A key property is that this pseudo-Anosov homeomorhism has many (namely $4n$) singular points, all with a different number of prongs (the $i^{th}$ singular point has $2i+2$ prongs). By doing a DpA-type (Derived from pseudo Anosov type) bifurcation and a suspension, we get an Axiom A flow $X_t^+$ on a fibered hyperbolic $3$-manifold $V^+$, so that the non-wandering set of the flow $X_t^+$ is the union of a non-trivial saddle basic set $\Lambda^+$, $2n$ periodic attracting orbits $\gamma_1^+, \gamma_3^+, \dots, \gamma_{4n-1}^+$ and $2n$ periodic repelling orbits $\gamma_2^+, \gamma_4^+, \dots, \gamma_{4n}^+$. 

We choose some small open tubular neighborhoods of the $4n$ periodic orbits $\gamma_1^+,\dots,\gamma_{4n}^+$, the vector field $X^+$ with respect to the flow $X_t^+$ being transverse to the boundary of these tubular neighborhoods. By cutting out these tubular neighborhoods from $V^+$, we get a compact $U^+$ with $4n$ boundary components, each of which is a $2$-dimensional torus. We pick a copy $U^-$ of $U^+$, and equip $U^-$ with the vector field $X^-:=-X^+$. Then we consider the disjoint union $U:=U^+\sqcup U^-$ endowed with the vector field $X$ which is equal to $X^+$ on $U^+$ and equal to $X^-$ on $U^-$. Note that  $U$ is a compact $3$-dimensional manifold with $8n$ boundary components, each of which is a $2$-dimensional torus. The vector field $X$ is transverse to $\partial U$, pointing inwards $U$ on $4n$ boundary components (the union of which is called the entrance boundary of $U$ and denoted by $\partial^{in} U$), and outwards $U$ on $4n$ boundary components (the union of which is called the exit boundary of $U$ and denoted by $\partial^{out} U$). 

Then we glue the exit boundary of $U$ on the entrance boundary of $U$ in order to get a closed $3$-manifold. More precisely, we carefully choose $2n+1$ gluing homeomorphisms  $\varphi^0, \varphi^1, \dots, \varphi^{2n}: \partial^{out} U \to \partial^{in} U$. For each $m\in\{0,\dots,2n\}$, we consider the closed $3$-dimensional manifold $W^m:=U/\varphi^m$.  Since $X$ the vector field tranverse to $\partial U$, the flow $X_t$ induces a flow $Y_t^m$ on $W^m$ for $m\in\{0,\dots,2n\}$. Moreover, the main theorem of \cite{BBY} ensures that the gluing homeomorphisms $\varphi^0, \varphi^1, \dots, \varphi^{2n}$ can be chosen so that $Y_t^m$ is a transitive Anosov flow for every $m\in\{0,\dots,2n\}$. The gluing homeomorphisms $\varphi^0, \varphi^1, \dots, \varphi^{2n}$ being chosen in the same isotopy class, the manifolds $W^1,\dots,W^{2n}$ are pairwise homeomorphic. Note that these manifolds are toroidal. More precisely, $W^m$ contains exactly $4n$ pairwise non-isotopic incompressible tori $T_1^m,\dots,T_{4n}^m$ corresponding to the $8n$ boundary components of $U$ (which have be glued together by pairs). 

Then, we prove that, for each $m\in\{0,\dots,2n\}$, we can choose $2n$ periodic orbits $\alpha_1^m, \dots, \alpha_{2n}^m$ of the Anosov flow $Y^m_t$ such that:
\begin{enumerate} 
\item[(i)] $\alpha_j^m$ intersects the torus $T_{2j-1}^m$ and the torus $T_{2j}^m$ at exactly one point, and is disjoint from the torus $T_k^m$ for $k\neq 2j-1, 2j$ 
\item[(ii)] there exists a homeomorphism $H^m: W^m \to W^0$ so that $H(\alpha_j^m)=\alpha_j^0$ and so that \hbox{$H(W^s_{loc}(\alpha_j^m))= W^s_{loc}(\alpha_j^0)$} for every $j\in \{1,2,\dots, 2n\}$.
\end{enumerate}

Then, starting from the Anosov flow $Y_t^m$ on $W^m$, we perform a \emph{Dehn-Fried surgery} of index $k$ on each of the orbits $\alpha_1^m,\dots,\alpha_{2n}^m$. This yields  a new Anosov flow $Z_t^m$ on a new manifold $M^m$. Due to Property~(ii) above, the closed $3$-manifolds $M^0,\dots,M^{2n}$ are pairwise homeomorphic (so, we denote all these manifolds by $M$). We also prove that, provided that the index $k$ of the Dehn-Fried surgeries is chosen large enough, the manifold $M$ is a hyperbolic. We also prove that the Anosov flow $Z_t^m$ is a non-$\RR$-covered Anosov flow for every $m$. This concludes the construction part: we have constructed $2n+1$ transitive non-$\RR$-covered Anosov flows $Z_t^0,\dots,Z_t^{2n}$ on a closed hyperbolic $3$-manifold $M$.

\medskip

Then we are left to prove that the  Anosov flows $Z_t^1,\dots,Z_t^{2n-1}$ ($2n-1$ of the $2n+1$ Anosov flows $Z_t^0,\dots,Z_t^{2n}$) are pairwise orbitally inequivalent. For this purpose, we will use the orbit space theory, which was developed mostly by Barbot and Fenley, see for instance \cite{Fen1}, \cite{Fen2} and \cite{Ba}. 

Recall that the orbit space $\cO_m$ of the Anosov flow $Z_t^m$ is by definition the quotient of the universal cover $\widetilde M$ of the manifold $M$ by the action of the lift of the flow $Z^m_t$. One can prove that $\cO_m$ is a topological plane. Moreover $\cO_m$ is equipped with  a pair of transverse one-dimensional foliations obtained by projecting the lifts of the weak stable and unstable foliations of $Z_t^m$. The geometry of the bi-foliation of $\cO_m$ encodes many important properties of the Anosov flow $Z_t^m$. For example, Barbot and Fenley have discovered that so-called \emph{lozenges} in $\cO_m$ correspond to fundamental Birkhoff annuli for the flow $Z^m_t$, and maximal chain of lozenges is closely related to a free homotopy classes of periodic orbits. 
Roughly speaking, a \emph{lozenge} in $\cO_m$ a quadrilateral with two opposite ideal vertices, bounded by two half stable leaves and two half unstable leaves, and so that the bifoliation is trivial in the interior of the quadrilateral.
Here a  \emph{Birkhoff annulus} $\Sigma$ for the flow $Z^m_t$ is a compact annulus immersed in $M$, bounded by two orbits of $Z^m_t$, whose interior is transverse to the vector field $Z^m$, and  a \emph{fundmental  Birkhoff annulus} $\Sigma$ is a Birkhoff annulus with a further condition: the  $1$-foliation induced by   the intersection of the interior of $\Sigma$ and the weak stable/unstable manifolds of $Z^m_t$ does not contain any compact leaf. 

We will study some particular sets of lozenges in $\cO_m$. We say that two lozenges in $\cO_m$ are \emph{edge-adjacent} if they share a common edge. A \emph{cluster} is a set of lozenges, in which any two lozenges can be connected by a finite sequence of lozenges, each lozenge of the sequence being edge-adjacent to the next one. See Subsection~\ref{ss.feap} for precise definitions. With the help of a series of observations and homology computations (Lemma \ref{l.Whomotopy}, \ref{l.noadjloz}, \ref{l.noonenew}), we can almost completely describe the \emph{maximal clusters} in $\cO_m$ (Proposition \ref{p.EAP}). 

Recall that the Anosov flow $Z^m_t$ was obtained from an Anosov flow $Y^m_t$ by performing Dehn-Fried surgeries on periodic orbits $\alpha_1^m,\dots,\alpha_{2n}^m$. The Anosov flow  $Y^m_t$ has many fundamental Brikhoff annuli. Indeed, we know that $Y^m_t$ has $4n$ transverse tori $T_1^m,\dots,T_{4n}^m$, and that the torus $T_{i}^m$ is isotopic to a cyclic sequence of $4i+4$ fundamental Birkhoff annuli of $Y^m_t$. A key point is that the periodic orbits  $\alpha_1^m,\dots,\alpha_{2n}^m$ are disjoint from most of these fundamental Birkhoff annnuli. As a consequence, although the transverse tori $T_1^m,\dots,T_{4n}^m$ are destroyed by the Dehn Fried surgeries, most of the fundamental Bikrhoff annuli of $Y^m_t$ are not destroyed by these surgeries. This yields fundamental Birkhoff annnuli for the flow $Z^m_t$. More precisely, the torus $T_i^m$ gives rise to a chain of $4i+3$ fundamental Birkhoff annuli for $Y^m_t$, and therefore a cluster of $4i+3$ lozenges in the orbit space $\cO_m$, which we denote by $C_i^m$. Of course, new fundamental Birkhoff annuli, and therefore new lozenges, can have been created by the Dehn surgeries. Nevertheless, we manage to prove that such new lozenges cannot aggregate and therefore cannot yield new clusters. This allows to use the clusters $C_1^m,\dots,C_{4n}^m$ as a kind of signature of the flow $Z^m_t$. 

Another important ingredient is the possibility to distinguish \emph{left-handed} and \emph{right-handed} fundamental Birkhoff annuli, which lead to an analogous distinction for (certain specific kinds of) clusters of lozenges. We observe that, for any given $i$, the handedness of the above mentioned cluster $C_i^m$ depends on $m$. This allows us to rule out the possibility of the existence of a orbit equivalence between the Anosov flows $Z^{m_1}_t$ and $Z^{m_2}_t$ for $m_1,m_2\in\{1\dots,2n-1\}$, $m_1\neq m_2$. In other words, this proves that the $2n-1$ Anosov flows $Z^{1}_t,\dots, Z^{2n-1}_t$ are pairwise orbitally inequivalent.

\section*{Acknowledgments}
The second author is supported by Shanghai Pilot Program for Basic Research, National Program for Support of Top-notch Young Professionals and the Fundamental Research Funds for the Central Universities.

%% file: Hyperbolic_plug.tex
A \emph{hyperbolic plug} is a pair $(U,X)$ where $U$ is a compact orientable $3$-dimensional  manifold with boundary, and $X$ a $C^1$ vector field on $U$, which is transverse to the boundary $\partial U$, and such that the maximal invariant set $\Lambda:=\bigcap_{t\in\RR} X_t(U)$ is a saddle hyperbolic set for the flow $X_t$.  The purpose of the present section is to build a hyperbolic plug with some specific topological and dynamical features. Before proceeding to the construction, we will recall some basic definitions and properties concerning hyperbolic plugs. A much more detailed discussion, including proofs of the properties stated below, can be founded in \cite[section 3]{BBY}.

\subsection{General facts about hyperbolic plugs}\label{ss.Gfhplug}

Let $(U,X)$ be such a hyperbolic plug, and $\Lambda:=\bigcap_{t\in\RR} X_t(U)$ be its maximal invariant set. For sake of simplicity, we assume that $\Lambda$ is a transitive set, not reduced to a periodic orbit. A periodic orbit $\omega\in\Lambda$ is called a \emph{u-boundary periodic orbit} if one stable separatrix of $\omega$ (that is one connected component of $W^s(\omega)\setminus\omega$) is disjoint form $\Lambda$. Such a stable separatrix disjoint from $\Lambda$ is called a \emph{free stable separatrix}. Roughly speaking, this means that $\omega$ is a u-boundary periodic orbit if $W^u(\omega)$ is not accumulated by $W^u(\Lambda)$ on both sides. An easy argument, based on the local product structure of locally maximal hyperbolic sets, implies that $\Lambda$ contains only finitely many $u$-boundary periodic orbits. Moreover, our hypothesis implies that the lamination $W^u(\Lambda)$ has no isolated leaf; it follows that every $u$-boundary periodic orbit $\omega$ has two stable separatrices ($W^s(\omega)$ is a cylinder, not a Möbius band), only one of which is a free separatrix. Of course, \emph{$s$-boundary periodic orbits} are defined similarly, and bear similar properties. 

Since the vector field $X$ is transverse to $\partial U$, one may decompose $\partial U$ as 
$$\partial U=\partial^{in} U\sqcup\partial^{out} U$$ 
where $X$ is pointing inwards $U$ along $\partial^{in} U$, and outwards $U$ along $\partial^{out} U$. The sets $\partial^{in} U$ and $\partial^{out} U$ are called the \emph{entrance boundary} and the \emph{exit boundary} of $U$. Both are unions of connected components of $\partial U$. The set 
$$\cL^s:=W^s(\Lambda)\cap\partial^{in} U$$
is called the \emph{entrance lamination} of the hyperbolic plug $(U,X)$. It can be shown that $\cL^s$ is a lamination, with one-dimensional leaves, embedded in the surface $\partial^{in}U$, with quite specific properties: 
\begin{enumerate}
\label{properties-laminations}
\item[(i)] Every compact leaf of $\cL^s$ is the intersection of $\partial^{in} U$  with the free stable separatrix of a $u$-boundary periodic orbit. Conversely, the intersection of $\partial^{in} U$ with the free stable separatrix of a $u$-boundary periodic orbit is a compact leaf of $\cL^s$, but we will not use this fact. In particular, $\cL^s$ has only finitely many compact leaves. 
\item[(ii)] Every non-compact leaf of $\cL^s$  accumulates on a compact leaf at both ends\footnote{Every non-compact leaf is the image of a injective immersion of the real line, therefore has two ends.}.
\end{enumerate}
Let $c$ be a closed leaf of $\cL^s$. According to item~(i) above, $c$ is included in the stable manifold $W^s(\omega)$ of a $u$-boundary periodic orbit $\omega$. Since $c$ is included in $\partial^{in} U$, it is transverse to the vector field $X$. It follows that $c$ is freely homotopic to $\omega$ in the cylinder $W^s(\omega)$ (indeed every unoriented simple closed curve drawn on $W^s(\omega)$ and transverse to $X$ is freely homotopic to $\omega$ in $W^s(\omega)$). Hence, the orientation of the orbit $\omega$ (given by the vector field $X$) induces an orientation of the closed leaf $c$, called the \emph{dynamical orientation} of $c$. The hyperbolicity of the flow $X_t$ easily implies:
\begin{enumerate}
\item[(iii)] The holonomy of $\cL^s$ along a dynamically oriented closed leaf $c$ is a dilation.
\end{enumerate}
The set $\cL^u:=W^u(\Lambda)\cap\partial^{out} U$ bears analogous properties. Item (iii) must be replaced by: 
\begin{enumerate}
\item[(iii')] The holonomy of $\cL^u$ along a dynamically oriented closed leaf $c$ is a contraction.
\end{enumerate}
Items (iii) and (iii') can be rephrased as follows: the compact leaves of $\cL^s$ and $\cL^u$ can be oriented so that the holonomy along each compact leaf is a contraction. This orientation coincides with the dynamical orientation in the case of $\cL^s$, and is opposite to the dynamical orientation in the case of $\cL^u$. 

\subsection{Construction of a hyperbolic plug $(U,X)$}

We are now in a position to state the main result of the section.

\begin{theorem}\label{t.fhplug}
For every $n\in\ZZ_{>0}$, there exists a hyperbolic plug $(U, X)$ satisfying the \hbox{properties (1)...(8)} below. 
\begin{enumerate}
\item[(1)] $U$ has two connected components $U^-$ and $U^+$, and there exists an involutive diffeomorphism $\sigma:U\righttoleftarrow$ such that $\sigma(U^\pm)=U^\mp$. 
\item[(2)] $U^\pm$ is an orientable fibered hyperbolic $3$-manifold with $4n$ tori boundary components $T_1^\pm,\dots, T_{4n}^\pm$. We choose the indexing  so that $\sigma(T_i^\pm)=T_i^\mp$.
\item[(3)] $\sigma_*X=-X$ (hence $\sigma$ conjugates the flow $X_t$ on $U^\pm$ to the reverse flow $X_{-t}$ on $U^\mp$), 
\item[(4)] $X$ is transverse to the fibration of $U\pm$.
\item[(5)] $X$ is pointing inwards $U$ along $T_1^-, T_3^-,\dots, T_{4n-1}^-$ and $T_2^+, T_4^+, \dots, T_{4n}^+$, and outwards $U$ along $T_1^+, T_3^+,\dots, T_{4n-1}^+$ and $T_2^-, T_4^-,\dots, T_{4n}^-$. 
\end{enumerate}
In accordance with item (5), for $i=1,\dots,4n$, we denote:
$$T_i^{in}=\left\{\begin{array}{ll} T_i^+ & \mbox{if $i$ is even}\\T_i^- & \mbox{if $i$ is odd}\end{array}\right.\mbox{ and }\quad  T_i^{out}=\left\{\begin{array}{ll} T_i^- & \mbox{if $i$ is even}\\T_i^+ & \mbox{if $i$ is odd}\end{array}\right.,$$
so that the entrance boundary and the exit boundary of the plug $(U,X)$ are
$$\partial^{in} U=T_1^{in}\sqcup\dots\sqcup T_{4n}^{in}\quad\mbox{ and }\partial^{out} U=T_1^{out}\sqcup\dots\sqcup T_{4n}^{out}$$
See Figure \ref{f.UXt}. Note that $\sigma(T_i^{in})=T_i^{out}$ for every $i$. We denote by $\Lambda^-$, $•\Lambda^-$ and $\Lambda$ the maximal invariant sets of $U^-$, $U^+$ and $U$:
$$\Lambda^-:=\bigcap_{t\in\RR}X_t(U^-)\quad\quad\Lambda^+:=\bigcap_{t\in\RR}X_t(U^+)\quad\quad\Lambda:=\Lambda^-\sqcup\Lambda^+=\bigcap_{t\in\RR}X_{t}(U)$$
Recall that $(U,X)$ being a hyperbolic plug entails that $\Lambda^-$, $\Lambda^+$ and $\Lambda$ are saddle hyperbolic sets. 
For $i=1,\dots,4n$, we consider the laminations 
$$\cL_i^s:=W^s(\Lambda) \cap T_i^{in}\quad\mbox{and}\quad\cL_i^u:=W^u(\Lambda) \cap T_i^{out},$$
so that the entrance and exit laminations of the plug $(U,X)$ are 
$$\cL^s=W^s(\Lambda)\cap\partial^{in} U=\cL_1^s\sqcup\dots\sqcup\cL_{4n}^s \quad\mbox{and}\quad \cL^u=W^u(\Lambda)\cap\partial^{out} U=\cL_1^u\sqcup\dots\sqcup\cL_{4n}^u.$$
Note that $\sigma(\cL_i^s)=\cL^u_i$ for every $i$.
\begin{enumerate}
\item[(6)] The saddle hyperbolic sets $\Lambda^-$ and $\Lambda^+$ are transitive. 
\item[(7)] $\cL_i^s$ consists in $2i+2$ Reeb lamination\footnote{By a \emph{Reeb lamination annulus}, we mean a compact annulus $A\simeq (\RR/\ZZ)\times [0,1]$ endowed with a one-dimensional lamination $\cL$ such that the leaves of $\cL$ are leaves of a Reeb foliation on $A$, such that $\partial A\subset \cL$ and $\cL\cap\mathrm{int}(A)\neq\emptyset$.} annuli $\{A_i^{j,s}\}_{j\in\ZZ/(2i+2)\ZZ}$, where $A_i^{j-1,s}$ and $A_i^{j,s}$ share a boundary component $c_i^{j,s}$ for every $j$. See Figure \ref{f.T1in} for the case $i=1$. 
\item[(7')] Similarly, $\cL_i^u$ consists in $2i+2$ Reeb  lamination annuli $\{A_i^{j,u}\}_{j\in\ZZ/(2i+2)\ZZ}$, where $A_i^{j-1,u}$ and $A_i^{j,u}$ share a boundary component $c_i^{j,u}$ for every $j$. 
\end{enumerate}
For each $i$, we choose the indexing of the Reeb annuli of the laminations $\cL^s_i$ and $\cL^u_i=\sigma(\cL^s_i)$ so that $\sigma(A_i^{j,s})=A_i^{j,u}$ and  $\sigma(c_i^{j,s})=c_i^{j,u}$ for every $j$. According to item~(i) page~\pageref{properties-laminations}, if $i$ is even (odd), 
the compact leaf $c_i^{j,s}$ (resp. $c_i^{j,u}$) is the intersection of the torus  $T_i^{in}$ (resp.  $T_i^{out}$) with the free stable (resp. unstable) separatrix of a $u$-boundary (resp. $s$-boundary) periodic orbit $\gamma_i^{j,+}$. We define $\gamma_i^{j,-}=\sigma(\gamma_i^{j,+})$ that is an s/u-boundary periodic orbit in $\Lambda^-$.
The converse assertion in item~(i) page~\pageref{properties-laminations} implies that 
$$\Gamma:=\{\gamma_i^{j,\epsilon}\}_ {i=\{1,\dots,4n\}, j\in\ZZ/(2i+2)\ZZ, \epsilon\in\{+,-\}}$$
is the collection of \emph{all} the s/u-boundary periodic orbits in $\Lambda$. 
\begin{enumerate}
\item[(8)] If $\omega$ is a periodic orbit which is not in $\Gamma$, then $\omega$ is not freely homotopic (as an unoriented curve) to any other periodic orbit in $\Lambda$, nor to any closed curve in $\partial U$.
\end{enumerate}
\end{theorem}

 \begin{figure}[htp]
\begin{center}
  \includegraphics[totalheight=6.8cm]{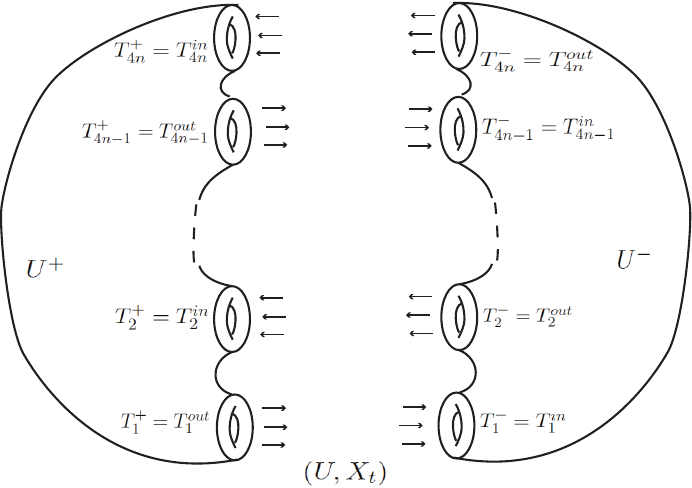}\\
  \caption{The notations for $\partial U$}\label{f.UXt}
\end{center}
\end{figure}

\begin{figure}[htp]
\begin{center}
\includegraphics[totalheight=5.2cm]{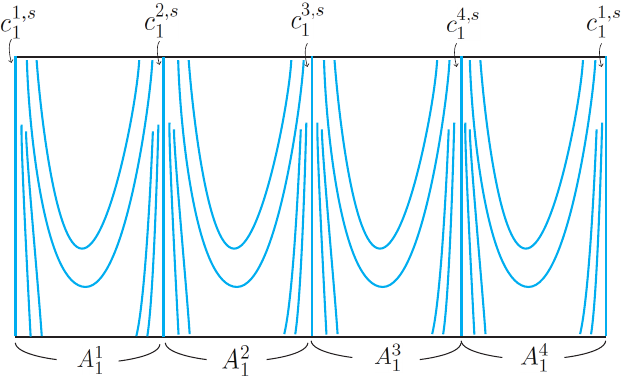}
\caption{The lamination $\cL_1^s$ on the torus $T_1^{in}$}\label{f.T1in}
\end{center}
\end{figure}

Theorem~\ref{t.fhplug} relies on the following lemma, which is a particular case of a result of Masur and Smillie \cite[Theorem 2]{MS}:

\begin{lemma}\label{l.psin}
For every $n\in \ZZ^+$, there exists an orientation preserving  pseudo-Anosov homeomorphism $\bar f^+:S^+\to S^+$, where $S^+$ is a closed orientable surface of genus $4n^2+n+1$, such that $\bar f^+$ has exactly $4n$ singular points $\bar x_1^+,\dots,\bar x_{4n}^+$, where $\bar x_i^+$ is a $(2i+2)$-prongs singularity. 
\end{lemma}

\begin{proof}[The proof of Theorem \ref{t.fhplug}]
We begin our construction by considering the pseudo-Anosov homeomorphism $\bar f^+:S^+\to S^+$ provided by lemma~\ref{l.psin}. Replacing $\bar f^+$ by a power if necessary, we may assume that $\bar f^+$ fixes each prong of each of its singularities. We denote by $V^+$ the mapping torus of $\bar f^+$. The suspension flow $\bar X_t^+$ of $\bar f^+$ is a pseudo-Anosov flow on $V^+$, with $4n$ singular orbits $\bar \gamma_1^+,\dots,\bar \gamma_{4n}^+$, where $\bar \gamma_i^+$ has a $(2i+2)$ stable and (2i+2) unstable separatrices, each of which wraps only once around $\gamma_i^+$. For definitions and details on pseudo-Anosov flows, see e.g.~\cite{Mo}.

Then we perform DpA bifurcations. More precisely, we consider the flow $X_t^+$ on $V^+$ obtained from $\bar X_t^+$ by performing a attracting DpA-bifurcation along the singular orbits $\bar \gamma_1^+, \bar\gamma_3^+, \dots,\bar\gamma_{4n -1}^+$ and a repelling DpA-bifurcation along the singular orbits  $\bar\gamma_2^+, \bar\gamma_4^+, \dots,\bar\gamma_{4n}^+$. Then $X_t^+$ is an Axiom A flow whose non-wandering set consists in a non-trivial saddle basic piece $\Lambda^+$, of $2n$ attracting periodic orbits $\gamma_1^+, \gamma_3^+, \dots, \gamma_{4n-1}^+$ and of $2n$ repelling periodic orbits $\gamma_2^+, \gamma_4^+, \dots, \gamma_{4n}^+$. Moreover, for $i$ odd (resp. even), each stable (resp. unstable) separatrix of the singular orbit $\bar \gamma_i^+$ of $\bar X_t^+$ gives rise to $2i+2$ s-boundary (resp. u-boundary) periodic orbits $\gamma_i^{1,+},\dots,\gamma_i^{2i+2,+}$ for the flow $X_t^+$. See Figure~\ref{f.DpA} for the case $i=1$. Details on DA and DpA-bifurcations for flows\footnote{Note that instead of considering the suspension of the pseudo-Anosov homoeomorphism $\bar f^+$ and then performing DpA bifurcations on the suspended flow, one could have first perform DpA bifurcations on the pseudo-Anosov homoeomorphism $\bar f^+$ and then consider the suspension of the resulting axiom A diffeomorphism. This leads to the same flow.} can be found in \cite{Mo} or \cite[Section 8.1]{BBY} for example. 

\begin{figure}[htp]
\begin{center}
\includegraphics[totalheight=6.3cm]{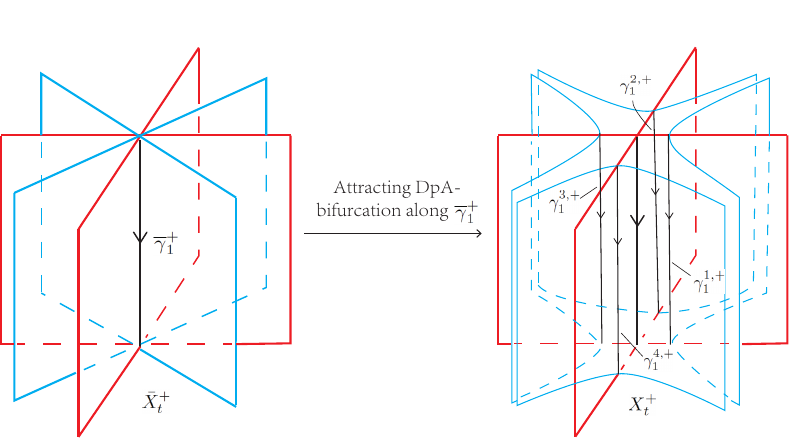}\\
\caption{Left: the union of blue (red) surface is the singular stable (resp. unstable) manifold of $\overline{\gamma}_1^+$ of $\overline{X}_t^+$. Right: every blue surface is the stable manifold of some $\gamma_1^{j,+}$ ($j=1,2,3,4$) of $X_t^+$. }\label{f.DpA}
\end{center}
\end{figure}

Now we cut out tubular neighborhood of the attracting and repelling periodic orbits. More precisely, we consider $U^+ := V^+ \setminus (\Omega_1^+\sqcup \dots \sqcup \Omega_{4n}^+)$ where $\Omega_i^+$ is a small open tubular neighborhood of the periodic orbit $\gamma_i^+$, such that $X^+$ is transverse to the torus $T_i^+:=\partial \Omega_i^+$. By construction, $U^+$ is a compact manifold with $4n$  tori boundary components $T_1^+,\dots,T_{4n}^+$, the vector field $X^+$ is transverse to $\partial U^+$, pointing inwards $U^+$ along $T_2^+,T_4^+,\dots,T_{4n}^+$ and outwards $U^+$ along $T_1^+,T_3^+,\dots,T_{4n-1}^+$. Also by construction, $U^+$ is a fibered manifold with pseudo-Anosov type monodromy (the interior of $U^+$ is the mapping torus the restriction of $\bar f$ to $S^+ - \{x_1^+,\dots,x_{4n}^+\}$). Hence, by Thurston's hyperbolization theorem for fibered $3$-manifolds, $U^+$ is a fibered hyperbolic manifold. The maximal invariant set of $U^+$ for the flow $X_t^+$ is nothing but the saddle hyperbolic set $\Lambda^+$. 

To obtain the plug $(U,X)$, we consider a copy $\{U^-, X^-,\Lambda^-,T_1^-,\dots,T_{4n}^-\}$ of $\{U^+, X^+,$ $\Lambda^+,T_1^+,\dots,T_{4n}^+\}$. We set $U =U^+ \sqcup U^-$. We define a vector field $X$ on $U$ by setting $X = X^+$ on $U^+$ and $X =-X^-$ on $U^-$.  We denote by $\sigma:U\righttoleftarrow$ the involutive diffeomorphism defined by the identification between $U^+$ and $U^-$.

Items (1),\dots,(5) are immediate consequences of the construction. 

Items (6) and (6') follow from \cite[item 8 of Proposition 8.1]{BBY} which describes the entrance/exit laminations of hyperbolic plugs that are obtained by performing DpA bifurcations on (pseudo-)Anosov flow, and cutting out tubular neighborhoods of the attracting/repelling periodic orbits. 

We are left to prove item (7). Although we are interested in homotopies in the manifold with boundary $U^+$, the proof essentially takes place in closed manifold $V^+$. Since $\bar X_t^+$ is the suspension of a pseudo-Anosov homeomorphism, Nielsen theory (see \cite{JG}), implies that~:
\begin{enumerate}
\item[(a)] a periodic orbit of $\bar X_t^+$ is never homotopically trivial in $V^+$, and two different periodic orbits of $\bar X_t^+$ are never homotopic in $V^+$ (even when considered as unoriented curves).
\end{enumerate}
Now, recall that $X_t^+$ was obtained from  $\bar X_t^+$ by performing DpA bifurcations on the singular orbits $\bar\gamma_1^+,\dots,\bar\gamma_{4n}^+$. It follows (see \emph{e.g.} \cite[Section 8.1]{BBY}) that there is a continuous onto map $\pi:V^+\to V^+$, homotopic to identity, which sends each orbit of $X_t^+$ to an orbit of $\bar X_t^+$. Moreover, this map $\pi$ is ``almost one-to-one". More precisely:
\begin{itemize}
\item[(b)] The inverse image under $\pi$ of a non-singular periodic orbit of $\bar X_t^+$ is a single periodic orbit of $X^t$, which is not a boundary periodic orbit. 
\item[(c)] For $i$ odd (resp. even), the inverse image under $\pi$ of the singular periodic orbit $\bar\gamma_i^+$ is the union of the bassin of the attracting (resp. repelling) periodic orbit $\gamma_i^+$ and of the stable (resp. unstable) manifolds of the boundary periodic orbits  $\gamma_i^{1,+},\dots,\gamma_i^{2i+2,s}$. 
\end{itemize}
Consider a periodic orbit $\omega\subset\Lambda^+$ of $X_t^+$, which is not a boundary periodic orbit. If $\omega'\subset\Lambda^+$ is another periodic orbit of $X_t^+$, then items (b) and (c) imply that $\omega$ and $\omega'$ are mapped by $\pi$ on two different periodic orbits $\bar\omega$ and $\bar\omega'$ of $\bar X_t^+$. According to item (a), the orbits  $\bar\omega$ and $\bar\omega'$ are not freely homotopic (as unoriented curves) in $V^+$. Since $\pi$ is homotopic to identity, it follows that $\omega$ and $\omega'$ are not  homotopic in $V^+$, and therefore not  homotopic in $U^+$. Now let $c$ be a closed curve drawn on a connected component $T_i^+$ of $\partial U^+$. Recall that $T_i^+$ bounds a tubular neighborhood $\Omega_i^+$ of the orbit $\gamma_i^+$ in $V^+$. It follows that, in $V^+$, the curve $c$ is either homotopically trivial, or homotopic to a power of the attracting/repelling periodic orbit $\gamma_i^+$. Hence, the curve $c$ is  either homotopically trivial in $V^+$, or homotopic to a power of the singular orbit $\bar\gamma_i^+$ of $\bar X_t^+$. On the other hand, items (b) and (c) and  the fact that $\pi$ is homotopic to identity imply that $\omega$ is homotopic to a non-singular periodic orbit of $\bar X_t^+$. Using item (a), we conclude that $\omega$ and $c$ are not freely homotopic in $V^+$, therefore not freely homotopic in $U^+$. Item (8) is proved.
\end{proof}

\subsection{Orientation of the hyperbolic plug $(U,X)$}\label{ss.orientation}
We fix an orientation on $U$ as follows: we choose an orientation on the connected component $U^-$, and we endow $U^+$ with the opposite of the push forward of this orientation under $\sigma$. With this connvention,
\begin{enumerate}
\item[(O1)] $\sigma$ is orientation reversing. 
\end{enumerate}
Now recall that $U^+$ and $U^-$ are fibered manifolds, and that the vector field $X$ is transverse to the fibrations. So we may endow the fibers in $U^-$ and $U^+$ with orientation such that:
\begin{enumerate}
\item[(O2)] The vector field $X$ intersects positively the fibers of the fibration of $U^-$ and $U^+$.
\end{enumerate}
With this convention, we have the following important property:
\begin{enumerate}
\item[(O3)] Endow the compact leaves of the lamination $\cL^s$ and $\cL^u$ so that the holonomy of $\cL^{s/u}$ along each oriented compact leaf is a dilation. Then the compact leaves of $\cL^u$ (resp. $\cL^s$) intersect positively (resp. negatively) the fibers of the fibration of $U^\pm$.
\end{enumerate}
\begin{proof}[Proof of Property (O3)]
Each compact leaf $c$ of the lamination $\cL^{s/u}$ is freely homotopic to a $u/s$-boundary periodic orbit $\omega$ (see item (i) page~\pageref{properties-laminations}). By definition, the dynamical orientation of the compact leaf $c$  is induced by the flow-orientation of the periodic orbit $\omega$ (see again page~\pageref{properties-laminations}). In view of Property~(O2), this implies that, for their dynamical orientation, the compact leaves of the lamination $\cL^{s/u}$ intersect positively the fibers of $U^\pm$. Yet, we wish to consider another orientation of the compact leaves. We endow the compact leaves of $\cL^{s/u}_i$ with the orientation so that the holonomy of the lamination $\cL_i^{s/u}$ is a contraction along every oriented compact leaf. According to properties (iii) and (iii') on page~\pageref{properties-laminations}, this orientation agrees (resp. disagrees) with the dynamical orientation for the leaves of $\cL^{s/u}$ (resp. $\cL^{u/s}$). Hence, for the ``contracting holonomy orientation", the compact leaves of $\cL^u$ (resp. $\cL^s$) intersect positively (resp. negatively)  the fibers of $U^\pm$.
\end{proof}

\noindent For $i\in\{1,\dots,4n\}$ and $\epsilon=\pm$, the boundary torus $T_i^\epsilon$ is trivially foliated by boundary curves of the fibers of the fibration of $U^\epsilon$. The fibers being oriented by Property (O2), their boundary curves inherit of orientations. 
\begin{enumerate}
\item[(O4)] We choose the cyclic ordering of the annuli $\{A_i^{j,s/u}\}_{j\in\ZZ/(2i+2)\ZZ}$ in $T_i^{\epsilon}$ so that it agrees with the orientation of the boundary curves of the fibers of $U^\epsilon$ when $i$ is even, and disagrees when $i$ is odd.
\end{enumerate}

\begin{remark}
Property (O4) is meaningful since the cores of the Reeb annuli are not in the same homotopy class as the boundary curves of the fibers in $T_i^\epsilon$: this follows from Property~(O3). Moreover Property (O4) is compatible with the requirement that $\sigma(A_i^{j,s})=A_i^{j,u}$ for every $i$ and $j$: this can be checked using Properties~(O1),~(O2) and(O3).
\end{remark}

The following proposition provides a convenient way to track the orientation of $U$. 

\begin{proposition}\label{p.orient}
Let $x$ be a point on the compact leaf $c_i^{j,\nu}$ of the lamination $\cL_i^\nu$ for some $i\in\{1,\dots,4n\}$, some $j\in \ZZ/(2i+2)\ZZ$ and some $\nu\in\{s,u\}$ (hence $x\in T_i^\epsilon$ where $\epsilon=+$ if $i$ is even and $\nu=s$ or $i$ is odd and $\nu=u$, and $\epsilon=-$ otherwise, a).  Let $(\vec{e_1}(x),\vec{e_2}(x),\vec{e_3}(x))$ be a basis of the tangent space $T_{x}U^\epsilon$ with the following characteristics (see Figure~\ref{f.orientation}):
\begin{itemize}
\item[---] $\vec{e_1}(x)$ is tangent to the torus $T_i^\epsilon$, transverse to the compact leaf $c_i^{j,\nu}$, pointing towards the Reeb annulus $A_i^{j+1,\nu}$;
\item[---] $\vec{e_2}(x)$ is tangent to the compact leaf $c_i^{j,\nu}$ pointing in the direction where the holonomy of the lamination $\cL^\nu$ is contracting if $i$ is odd, and expanding if $i$ is even;
\item[---] $\vec{e_3}(x)=X(x)$. 
\end{itemize}
Then $(\vec{e_1}(x),\vec{e_2}(x),\vec{e_3}(x))$ is positively oriented.
\end{proposition}

\begin{proof}
Define another basis $(\vec{f_1}(x),\vec{f_2}(x),\vec{f_3}(x))$ of  $T_{x}U^\epsilon$ as follows:
\begin{itemize}
\item[---] $\vec{f_1}(x)=\vec{e_1}(x)$ if $i$ is even and $\vec{f_1}(x)=-\vec{e_1}(x)$ if $i$ is odd;
\item[---] $\vec{f_2}(x)=\vec{e_2}(x)$ if ($\nu=u$ and $i$ even) or ($\nu=s$ and $i$ odd), and $\vec{f_2}(x)=\vec{e_2}(x)$ if ($\nu=u$ and $i$ odd) or ($\nu=s$ and $i$ even)
\item[---] $\vec{f_3}(x)=-X(x)=-\vec{e_3}(x)$ if $\nu=s$ (equivalently $T_i^\epsilon=T_i^{in}$) and $\vec{f_3}(x)=X(x)=\vec{e_3}(x)$ if $\nu=u$ (equivalently $T_i^\epsilon=T_i^{out}$). 
\end{itemize}
By Property (O3), $\vec{f_2}(x)$ intersects positively the fibers of $U^\epsilon$. The vector $\vec{f_3}(x)$ is pointing outwards $U^\epsilon$, hence it can be used as an outward-pointing normal vector to defined the orientation of the boundary of the fibers in $U^\epsilon$. Finally, according to Property (O4), $\vec{f_1}(x)$ agrees with the orientation of the fibers of $U^\epsilon$. It follows that the basis $(\vec{f_2}(x),\vec{f_3}(x),\vec{f_1}(x))$ is positively oriented. Hence the basis $(\vec{f_1}(x),\vec{f_2}(x),\vec{f_3}(x))$ is also positively oriented. Now, using the every definition of $(\vec{f_1}(x),\vec{f_2}(x),\vec{f_3}(x))$ and considering four different cases (depending on whether $i$ is odd or even, and $\nu=s$ or $u$), one easily checks that that $(\vec{f_1}(x),\vec{f_2}(x),\vec{f_3}(x))$ always defines the same orientation as $(\vec{e_1}(x),\vec{e_2}(x),\vec{e_3}(x))$. We conclude that $(\vec{e_1}(x),\vec{e_2}(x),\vec{e_3}(x))$ is positively oriented.
\end{proof}

\begin{figure}[htp]
\begin{center}
 \includegraphics[width=8cm]{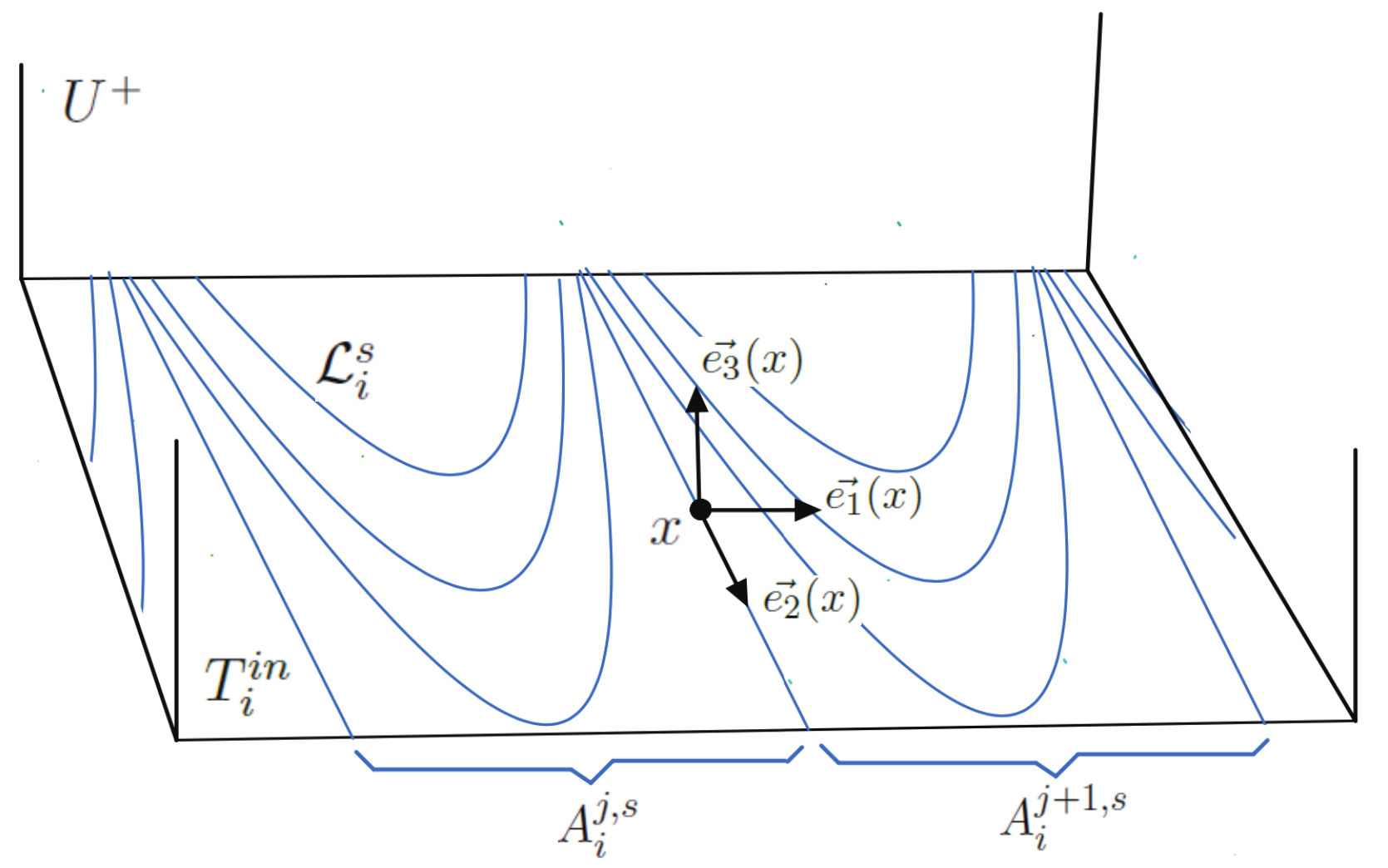}
 \caption{The frame $(\vec{e}_1 (x), \vec{e}_2 (x), \vec{e}_3 (x))$ for $i$ even and $\nu=s$}
 \label{f.orientation}
\end{center}
\end{figure}

%% file: Gluing-diffeomorphisms.tex
In this section, given a positive integer $n$, we will build $2n+1$ transitive Anosov flows carried by $2n+1$ homeomorphic toroidal manifolds. The construction involves the hyperbolic plug $(U,X)$ provided by Theorem~\ref{t.fhplug} and some gluing diffeomorphisms  $\varphi^0,\dots,\varphi^{2n}: \partial^{out}U\to\partial^{in} U$. More precisely, we will prove the following result.

\begin{theorem}\label{t.con}
Let $n\in\ZZ_{>0}$ and $(U,X)$ be the hyperbolic plug provided by Theorem~\ref{t.fhplug}. 
\begin{enumerate}
\item[(1)] For $m=0, \dots, 2n$, there exists a diffeomorphism $\varphi^m: \partial^{out}U \to \partial^{in} U$ such that 
\begin{enumerate}
\item the flow $Y_t^m$ induced by $X_t$ on the closed manifold $W^m :=U/\varphi^m$ is a transitive Anosov flow,
\item $\varphi^m$ is isotopic to the restriction of the involution $\sigma$ to $\partial^{out}U$ (in particular, $\varphi^m$ maps $T_i^{out}$ to $T_i^{in}$ for every $i=1,\dots,4n$),
\item for $i\leq 2m$, for $j\in\ZZ/(2n+2)\ZZ$, the annulus $\varphi^m (A_i^{j,u})$ intersects the annuli $A_i^{j-1,s}$ and $A_i^{j,s}$ (and does not intersect $A_i^{\ell,s}$ for $\ell\neq j-1,j$),
\item for $i>2m$, for $j\in\ZZ/(2n+2)\ZZ$, the annulus $\varphi^m (A_i^{j,u})$ intersects the annuli $A_i^{j,s}$ and $A_i^{j+1,s}$ (and does not intersect $A_i^{\ell,s}$ for $\ell\neq j,j+1$).
\end{enumerate}
\end{enumerate}
Let $\pi^m: U \to W^m=U/\varphi^m$ be the natural projection, and \hbox{$T_i^m:= \pi^m (T_i^{out})=\pi^m (T_i^{in})$.}
\begin{enumerate}
\item[(2)] For $m=0, \dots, 2n$ and $j=1,\dots,2n$, there exists a periodic orbit $\alpha_j^m$ of the Anosov flow $Y_t^m$ with the following properties:
\begin{enumerate}
\item $\alpha_j^m$ intersects the torus $T_{2j-1}^m$ at one point $q_{2j-1}^m\in\mathrm{int}\left(\pi^m (A_{2j-1}^{1,s}) \cap \pi^m (A_{2j-1}^{1,u})\right)$,
\item $\alpha_j^m$ intersects the torus $T_{2j}^m$ at one point $q_{2j}^m\in\mathrm{int}\left(\pi^m (A_{2j}^{1,s}) \cap \pi^m (A_{2j}^{1,u})\right)$,
\item  $\alpha_j^m$ does not intersect the torus $T_k^m$ for $k\neq 2j-1,2j$.
\end{enumerate}
\item[(3)] For $m=0,\dots, 2n$, for $j=1,\dots, 2n$, the local stable manifold $W^s_{loc}(\alpha_j^m)$ of the orbit $\alpha_j^m$ is an annulus (not a Möbius band).
\item[(4)] For $m=1,\dots, 2n$, there exists a homeomorphism $H^m: W^m \to W^0$ which maps the oriented periodic orbit $\alpha_j^m$ to the oriented periodic orbit $\alpha_j^0$, and the local stable manifold $W^s_{loc}(\alpha_j^m)$ to the local stable manifold $W^s_{loc}(\alpha_j^0)$ for $j=1,\dots, 2n$.
\end{enumerate}
\end{theorem}

The remainder of the section is devoted to the proof of Theorem~\ref{t.con}. Item (1), (2), (3) and (4) will be proved respectively in Subsection~\ref{ss.gluing-maps}.~\ref{ss.periodic-orbits},~\ref{ss.local-stable-annulus} and~\ref{ss.homeomorphic}. 

\subsection{Construction of the gluing maps}
\label{ss.gluing-maps}

This subsection is devoted to the proof of item~(1) of Theorem~\ref{t.con}. A key tool is the main theorem of~\cite{BBY}. Another tool is the existence of standard models for (pairs of) foliations with $(2i+2)$ Reeb annuli in a two-torus. Having presented these tools, we construct a bifoliation $(\cF^s,\cF^u)$ on the hyperbolic plug $(U,X_t)$. This bifoliation provides a kind of coordinate system on $\partial U$. Using this coordinates system, we are able to define some gluing maps  ${\widebar\varphi}^0,\dots,{\widebar\varphi}^{2n}$ enjoying all the required properties, except that the flow induced by $X_t$ on the closed manifold $U/{\widebar\varphi}^m$ need not be an Anosov flow. Then, using the main result of~\cite{BBY}, we will be able to deform ${\widebar\varphi}^m$ into a gluing map $\varphi^m$ so that the flow induced by $X_t$ on $U/\varphi^m$ is an Anosov flow. 

\subsubsection{Hyperbolic plugs and Anosov flows}
\label{sss.BBY}
Let us recall some definitions and the main result of~\cite{BBY}. A lamination $\cL$ in a surface $S$ is said to be \emph{filling} if any connected component of $S\setminus\cL$ is a topological disc, whose accessible boundary is made of two non-compact leaves of $\cL$ that are asymptotic to each other at both ends. Two laminations $\cL,\cL'$ embedded in a surface $S$ are said to be \emph{strongly transverse} if they are transverse, and if the closure of every connected component $D$ of $S\setminus (\cL\cup\cL')$ is a closed disc, so that $\partial D$ is a $4$-gon with two opposite sides included in leaves of $\cL$ and two opposite sides included in some leaves of $\cL'$. Let $(U, X)$ be a hyperbolic plug with maximal invariant set $\Lambda=\bigcap_{t} X_t(U)$, and $\varphi: \partial^{out} U \to \partial^{in} U$ be a diffeomorphism. Denote by $\cL^s:=W^s(\Lambda)\cap\partial^{in} U$ and $\cL^u=W^u(\Lambda)\cap\partial^{out} U$ the entrance lamination and the exit lamination of the plug $(U,X)$. The laminations $\cL^s$ is said to be \emph{filling} if any connected component of $\partial^{in} U\setminus\cL^s$ is a topological disc, whose accessible boundary is made of two non-compact leaves of $\cL^s$ that are asymptotic to each other at both ends. A similar definition holds for $\cL^u$. The gluing map $\varphi$ is said to be  a \emph{strongly transverse gluing diffeomorphism} if the laminations  $\cL^s$ and $\varphi_{\ast} \cL^u$ (both embedded in $\partial^{in} U$) are strongly transverse. Two strongly transverse gluing maps $\varphi_0,\varphi_1: \partial^{out} U \to \partial^{in} U$ are said to be \emph{strongly isotopic} if they can be connected by a continuous one-parameter family $(\varphi_s)_{s\in [0,1]}$ of strongly transverse gluing maps. The main result of \cite{BBY} can be stated as follows.

\begin{theorem}
\label{t.BBY}
Let $(U, X)$ be a hyperbolic plug with filling laminations and \hbox{$\widebar\varphi:\partial^{out} U \to \partial^{in} U$} be a strongly transverse gluing diffeomorphism. Then, up to replacing $X$ by an orbitally equivalent vector field, there exist a strongly transverse gluing diffeomorphism $\varphi: \partial^{out} U \to \partial^{in} U$ strongly isotopic to $\widebar\varphi$, and such that the flow $Y_t$ induced by $X_t$ on the closed manifold $M := U/\varphi$ is an Anosov flow. Moreover, $Y_t$ is transitive provided that $\widebar\varphi(W^u(\Lambda_1)\cap\partial^{out} U)\cap (W^s(\Lambda_2)\cap\partial^{in} U)\neq \emptyset$ for any two connected components $\Lambda_1$ and $\Lambda_2$ of the maximal invariant of $(U,X)$.
\end{theorem}

\subsubsection{Canonical models for (pairs of) foliations with $(2i+2)$ Reeb annuli}
\label{sss.canonical}
Consider the vector fields $\eta^s, \eta^u$ on $\RR^2$ defined by
\begin{eqnarray*}
\eta^s(x,y) & := & \sin (\pi x) \frac{\partial}{\partial x}   + \cos (\pi x)\frac{\partial}{\partial y},\\
\eta^u(x,y) & := & \sin (\pi (x+1/2)) \frac{\partial}{\partial x}   + \cos (\pi (x+1/2))\frac{\partial}{\partial y}.
\end{eqnarray*}
For $i=1,\dots,4n$, the un-oriented orbits of the vector fields $\eta^s$ and $\eta^u$ induce two one-dimensional  foliations $\ff^s_i$ and $\ff_i^u$ on the torus 
$$\TT_i:=(\RR/(2i+2)\ZZ)\times (\RR/\ZZ).$$
 Observe that $\ff_i^u$ is the image of $\ff_i^s$ under the horizontal translation $(x,y)\mapsto (x+1/2,y)$. One may easily verify that the foliations  $\ff_i^s$ and $\ff_i^u$ are transverse to each other. One may also check that, for $j\in\ZZ/(2i+2)\ZZ$, the annulus $[j,j+1]\times(\RR/\ZZ)$ is a Reeb annulus for the foliation $\ff^s_i$ and the annulus $[j-1/2,j+1/2]\times(\RR/\ZZ)$ is a Reeb annulus for the foliation $\ff^u_i$. In  particular, each of the two foliations $\ff^s_i,\ff_i^u$ has exactly $2i+2$ compact leaves, bounding $2i+2$ Reeb annuli. 
 
 \begin{figure}[htp]
\begin{center}
  \includegraphics[totalheight=7cm]{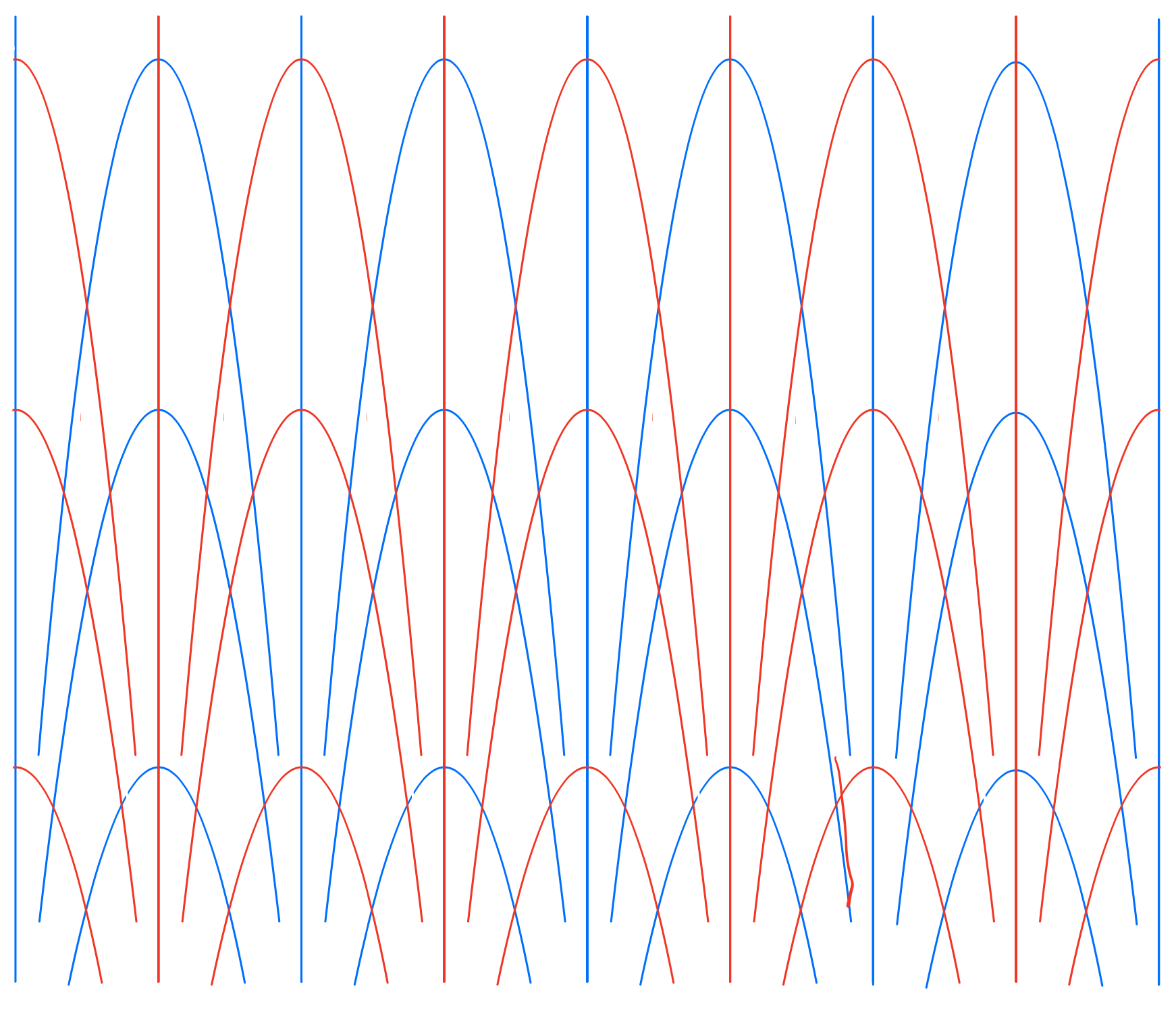}
  \caption{The pair of transverse foliations $(\ff_i^s,\ff_i^u)$ on the torus $\TT_i:=(\RR/(2i+2)\ZZ)\times (\RR/\ZZ)$ for $i=1$.}\label{f.model-foliattions}
\end{center} 
\end{figure}

 The following fact results from standard arguments:

\begin{fact}
\label{f.canonical}
If $f^s$ is a one-dimensional foliation on a two-torus $T$ with $2i+2$ compact leaves bounding $2i+2$ Reeb annuli with contracting holonomy on each compact leaf, then there is a homeomorphism from $T$ to the standard torus $\TT_i$ mapping $f^s$ to the standard foliation $\ff_i^s$. If $(f^s,f^u)$ is a pair of transverse one-dimensional foliations on a two-torus $T$, each with $2i+2$ compact leaves bounding $2i+2$ Reeb annuli and with contracting holonomy along each compact leaf, then there is a homeomorphism from $T$ to the standard torus $\TT_i$ mapping $(f^s,f^u)$ to the standard bi-foliation $(\ff_i^s,\ff_i^u)$.
\end{fact}

\subsubsection{A bifoliation $(\cF^s,\cF^u)$ on $U$}
\label{sss.bifoliation}
The proposition below asserts the existence of a bifoliation on $U$ which will play a crucial role in the definition of the gluing maps. 

\begin{proposition}
\label{p.construction-foliations}
There exists a pair transverse $X_t$-invariant two-dimensional foliations $(\cF^s,\cF^u)$ on $U$ with the following properties.
\begin{enumerate}
\item The involution $\sigma$ maps $\cF^s$ to $\cF^u$ and vice-versa. 
\item The foliation $\cF^s$ (resp. $\cF^u$) contains $W^s(\Lambda)$ (resp. $W^u(\Lambda)$) as a sublamination. 
\item Denote by $f_i^{s,in}$, $f_i^{s,out}$, $f_i^{u,in}$ and $f_i^{u,out}$ the one-dimensional foliations induced\footnote{The vector field $X$ is tangent to the leaves of  $\cF^s$ and $\cF^u$ and transverse to $\partial U$. Hence the foliations $\cF^s$ and $\cF^u$ are transverse to $\partial U$. As a consequence, $\cF^s$ and $\cF^u$ induce a pair of transverse one-dimensionnal foliations on the torus $T_i^{in/out}$ for every $i$.} by $\cF^s$ and $\cF^u$ on the tori $T_i^{in}$ and $T_i^{out}$. Then, each of four foliations has $2i+2$ compact leaves bounding $2i+2$ Reeb annuli with contracting holonomy along each compact leaf. 
\end{enumerate}
\end{proposition}

\begin{proof}
If $i$ is an even integer, the torus $T_i^{in}$ is a connected component of $\partial^{in} U^+$. The lamination $\cL^s_i=W^s(\Lambda^+)\cap T_i^{in}$ has $2i+2$ compact leaves bounding $2i+2$ Reeb lamination annuli. We may consider a one-dimensional foliation $f_i^{u,in}$ on $T_i^{in}$ which has exactly  $2i+2$ compact leaves bounding $2i+2$ Reeb lamination annuli, and which is transverse to the lamination $\cL^s_i$ (the existence of such a foliation $f_i^{u,in}$ follows from example from the first assertion of Fact~\ref{f.canonical} which allows to conjugate to a model). By doing so for $i=2,4,\dots,2n$, we obtain a foliation on $\partial^{in} U^+$. Pushing this foliation by the flow $X_t$ yields a two-dimensional $X_t$-invariant foliation on $\bigcup_{t\geq 0} X^t(\partial^{in} U^+)=U^+\setminus W^u(\Lambda^+)$. The $\lambda$-lemma implies that this foliation extends to a foliation $\cF^{u}$ on $U^+$. By construction,  $\cF^{u}$ contains $W^u(\Lambda^+)$ as a sublamination, and is transverse to $W^s(\Lambda^+)$.

If $i$ is an odd integer, the torus $T_i^{out}$ is a connected component of $\partial^{out} U^+$. Hence, we may consider the foliation $f^{u,out}_i$ induced by $\cF^u$ on $T_i^{out}$. Note that $f^{u,out}_i$ contains $\cL^u_i=W^u(\Lambda^+)\cap T_i^{out}$ as a sublamination. It follows that $f^{u,out}_i$ has exactly $2i+2$ compact leaves bounding $2i+2$ Reeb lamination annuli: since every connected component of $T_i^{out}\setminus\cL_i^u$ is simply connected, the Poincaré-Bendixon theorem implies that the only compact leaves of $f^{u,out}_i$ are those of $\cL^u_i$. So we may choose a one-dimensional foliation $f_i^{s,out}$ on $T_i^{out}$ which has exactly  $2i+2$ compact leaves bounding $2i+2$ Reeb lamination annuli, and which is transverse to $f^{s,out}_i$.  By doing so for $i=1,3,\dots,2n-1$, we obtain a foliation on $\partial^{out} U^+$. Pushing this foliation by the reverse flow $X_{-t}$ yields a two-dimensional $X_t$-invariant foliation on $\bigcup_{t\leq 0} X_t(\partial^{out} U^+)=U^+\setminus W^s(\Lambda^+)$. The $\lambda$-lemma implies that this foliation extends to a foliation $\cF^{s}$ on $U^+$. By construction,  $\cF^{s}$ contains $W^s(\Lambda^+)$ as a sublamination, and is transverse to the foliation $\cF^u$. Notice that, for $i$ even, the foliation $f^{s,in}_i$ induced by $\cF^{s}$ on $T_i^{in}$ contains $\cL^s_i$ as a sublamination; it follows that $f^{s,in}_i$ has exactly  $2i+2$ compact leaves bounding $2i+2$ Reeb lamination annuli.
Moreover, due to the hyperbolicity of $X_t$ on $\Lambda^+$, the holonomy along each compact leaf
of these $1$-foliations is contracting. See items (iii) and (iii') in Section \ref{ss.Gfhplug}.

At this stage, we have defined a pair of foliations $(\cF^{s},\cF^{u})$ on $U^+$ satisfying all the desired properties. To complete the proof, it remains to define $\cF^s$ and $\cF^u$ on $U^-$ by setting $\cF^s_{|U^-}:=\sigma_* \cF^u\mid U^+$ and $\cF^s_{| U^+}:=\sigma_* \cF^u_{|U^-}$. Using the equality $\sigma_*X_{|U^+}=-X_{|U^-}$, one may easily translate the properties of $\cF^s_{|U^+}$ and $\cF^u_{|U^+}$ into analogous properties for $\cF^s_{|U^-}$ and $\cF^u_{|U^-}$. 
\end{proof}

\subsubsection{Construction of the gluing maps ${\widebar\varphi}^0,\dots,{\widebar\varphi}^{2n}$ and the (non-Anosov) flows $\widebar Y_t^0,\dots,\widebar Y_t^{2n}$}
\label{sss.construction-psi}
Our next goal is to prove the following result, which might be considered as a downgraded version of item~(1) of Theorem~\ref{t.con}:

\begin{proposition}
\label{p.construction-psi}
For $m=0, \dots, 2n$, there exists a strongly transverse gluing map \hbox{${\widebar\varphi}^m: \partial^{out}U \to \partial^{in} U$} such that 
\begin{enumerate}
\item[(a)] the flow $\widebar Y_t^m$ induced by $X_t$ on the closed manifold $\widebar W^m :=U/{\widebar\varphi}^m$ preserves a pair of transverse foliations $(\widebar \cF^{s,m},\widebar \cF^{u,m})$ (induced by the pair of foliations $(\cF^s,\cF^u)$ on $U$), 
\item[(b)] ${\widebar\varphi}^m$ is isotopic to the restriction of the involution $\sigma$ to $\partial^{out}U$ (in particular, it maps $T_i^{out}$ to $T_i^{in}$ for every $i=1,\dots,4n$),
\item[(c)] for $i\leq 2m$, for $j\in\ZZ/(2n+2)\ZZ$, the annulus ${\widebar\varphi}^m (A_i^{j,u})$ intersects the annuli $A_i^{j-1,s}$ and $A_i^{j,s}$ (and does not intersect $A_i^{\ell,s}$ for $\ell\neq j-1,j$),
\item[(d)] for $i>2m$, for $j\in\ZZ/(2n+2)\ZZ$, the annulus ${\widebar\varphi}^m (A_i^{j,u})$ intersects the annuli $A_i^{j,s}$ and $A_i^{j+1,s}$ (and does not intersect $A_i^{\ell,s}$ for $\ell\neq j,j+1$).
\end{enumerate}
\end{proposition}

The gluing map ${\widebar\varphi}^m$ will be given by an explicit formula (in some coordinate system) which will prove to be quite helpful to prove item (3) and (4) of Theorem~\ref{t.con}. Note that the flow $\widebar Y_t^m$ is not an Anosov flow (even though it preserves a pair of transverse foliations which ``play the role of stable/unstable foliations"). 

\begin{proof}[Proof of Proposition~\ref{p.construction-psi}]
Let $\cF^s$, $\cF^u$, $f_i^{s,in}$, $f_i^{s,out}$, $f_i^{u,in}$ and $f_i^{u,out}$ be the foliations provided by Proposition~\ref{p.construction-foliations}. For $i=1,\dots,4n$,  $(f_i^{s,in},f_i^{u,in})$ is a pair of transverse one-dimensional foliations on the torus $T_i^{in}$, and each of these two foliations has exactly $2i+2$ compact leaves bounding $2i+2$ Reeb annuli. The foliation $f^{s,in}_i$ contains $\cL^s_i$ as a sublamination, and therefore the Reeb annuli of $f^{s,in}_i$ are nothing but the Reeb annuli $\{A_i^{j,s}, j\in \ZZ/(2i+2)\ZZ\}$ of the lamination $\cL^s_i$. According to Fact~\ref{f.canonical}, there exists a homeomorphism $\xi_i:T_i^{in}\to \TT_i=(\RR/(2i+2)\ZZ)\times (\RR/\ZZ)$ 
which maps the bi-foliation $(f_i^{s,in},f_i^{u,in})$ to the model bifoliation $(\ff_i^s,\ff_i^u)$. Postcomposing $\xi_i$ by the symmetry $(x,y)\mapsto (-x,y)$ and an integer translation, we may assume that $\xi_i$ maps the Reeb annulus  $A_i^{j,s}$ on the annulus $[j,j+1]\times (\RR/\ZZ)$ for every $j$. We will use $\xi_i$ as a chart or a coordinate system on the torus $T_i^{in}$. For every $v\in\RR$, we consider the diffeomorphism $\tau_v:\partial^{in} U\to\partial^{in} U$ such that $\tau_v$ preserves $T_i^{in}$ for each $i\in\{1,\dots,4n\}$ and $\tau_{v|T_i^{in}}$ is the translation $(x,y)\mapsto (x+v,y)$ in the chart $\xi_i$. In other words, $\tau_v$ is defined by
$$\xi_i\circ\tau_{v \mid T_i^{in}}\circ\xi_i^{-1}(x,y)=(x+v,y).$$ 
Now, for $m=0, \dots, 2n$, we define the gluing map ${\widebar\varphi}^m:\partial^{out}U \to \partial^{in} U$ defined by 
\begin{equation}
\label{e.definition-psi}
{\widebar\varphi}^m_{\mid T_i^{out}}  =  \left\{\begin{array}{ll}
\tau_{-1/2}\circ\sigma & \mbox{if } i\leq 2m \\
\tau_{+1/2}\circ\sigma & \mbox{if } i>2m 
\end{array}\right. .
\end{equation}
We denote by $\widebar Y_t^m$ the flow induced by $X_t$ on he quotient manifold $\widebar W^m:=U/\widebar\varphi^m$.

This definition entails that ${\widebar\varphi}^m$ is isotopic to $\sigma:\partial^{out} U\to\partial^{in} U$. In particular, it maps $T_i^{out}$ to $T_i^{in}$ for every $i$. Now, recall that $\sigma$ maps $\cF^s$ to $\cF^u$ and vice-versa, see Proposition~\ref{p.construction-foliations}. It follows that $\sigma_*f_i^{s,out}=f_i^{u,in}$ and $\sigma_*f_i^{u,out}=f_i^{s,in}$. Also note that $\tau_{\pm 1/2}$ maps $f^{s,in}_i$ to $f^{u,in}_i$ and vice versa (since the translation $(x,y)\mapsto (x\pm 1/2,y)$ maps $\ff^s_i$ to $\ff^u_i$ and vice-versa). We deduce that
\begin{equation}
\label{e.action-on-foliations}
{\widebar\varphi}^m_*f^{s,out}_i=f^{s,in}_i \quad\mbox{and}\quad {\widebar\varphi}^m_*f^{u,out}_i=f^{u,in}_i
\end{equation}
for every $i$. 
An important consequence of~\eqref{e.action-on-foliations} is that $\cL^s_i$ and ${\widebar\varphi}^m_*\cL^u_i$ are sub-laminations of two transverse foliations (namely the foliations $f^{s,in}_i$ and ${\widebar\varphi}^m_*f^{u,out}_i=f^{u,in}_i$). It follows that  $\cL^s_i$ and ${\widebar\varphi}^m_*\cL^u_i$ are strongly transverse. In other words, ${\widebar\varphi}^m$ is a strongly transverse gluing map. Another important consequence of~\eqref{e.action-on-foliations} is that the foliations $\cF^s$ and $\cF^u$ project to a pair of foliations $(\widebar\cF^{s,m}, \widebar\cF^{u,m})$ on the closed manifold $\widebar W^m:=U/{\widebar\varphi}^m$. Since $\cF^s$ and $\cF^u$ are $X_t$-invariant, it follows that $\widebar\cF^{s,m}$ and $\widebar\cF^{u,m}$ are $\widebar Y_t^m$-invariant.

We are left to prove items (c) and (d). For this purpose, we first recall that $\sigma$ maps $A_i^{j,s}$ to $A_i^{j,s}$ for every $i$ and $j$, see Theorem~\ref{t.fhplug}. Then, we observe that the definition of $\xi_i$ and $\tau_{-1/2}$ implies the annulus $\tau_{-1/2}(A_i^{j,s})$ intersects both $A_i^{j-1,s}$ and $A_i^{j,s}$. Item (c) follows. Item~(d) is proved similarly replacing $\tau_{-1/2}$ by $\tau_{+1/2}$. 
\end{proof}

\subsubsection{Construction of the gluing maps $\varphi^0,\dots,\varphi^{2n}$ and the Anosov flows $Y_t^0,\dots,Y_t^{2n}$}
\label{sss.construction-phi}

\begin{proof}[Proof of item~(1) of Theorem~\ref{t.con}]
Let $m\in\{1,\dots,2n\}$. The hyperbolic plug $(U,X)$ and the gluing map ${\widebar\varphi}^m$ provided by Proposition~\ref{p.construction-psi} satisfy the hypotheses of Theorem~\ref{t.BBY}: items~(7) and~(7') of Theorem~\ref{t.fhplug} entail that $\cL^s$ and $\cL^u$ are filling laminations, and Proposition~\ref{p.construction-psi} ensures that the gluing map ${\widebar\varphi}^m$ is strongly transverse. As a consequence, Theorem~\ref{t.BBY} guarantees that, up to replacing $X$ by an orbitally equivalent vector field, one can find a gluing map $\varphi^m$ that is strongly isotopic to ${\widebar\varphi}^m$ and so that the flow  $Y_t^m$ induced by $X_t$ on the closed manifold $W^m:=U/\varphi^m$ is an Anosov flow. 

The maximal invariant set of $(U,X)$ has two connected components $\Lambda^-$ and $\Lambda^+$. The properties of the gluing map ${\widebar\varphi}^m$ entail that ${\widebar\varphi}_*(\cL^u_i)$ intersects $\cL^s_i$ for each $i\in\{1,\dots,4n\}$. When $i$ is an odd integer, this implies that ${\widebar\varphi}_*(W^u(\Lambda^+)\cap\partial^{out} U)$ intersects $W^s(\Lambda^-)\cap\partial^{in} U$. When $i$ is an even integer, this implies that ${\widebar\varphi}_*(W^u(\Lambda^-)\cap\partial^{out} U)$ intersects $W^s(\Lambda^+)\cap\partial^{in} U$. Therefore Theorem~\ref{t.BBY} ensures that the flow $Y^m_t$ is transitive. This completes the proof of  item~(1).(a) of Theorem~\ref{t.con}.

By construction, $\varphi^m$ is strongly isotopic to ${\widebar\varphi}^m$ which is isotopic to $\sigma_{|\partial^{out} U}$. Hence $\varphi^m$ is isotopic to $\sigma_{|\partial^{out} U}$, as required by item (1).(b) of Theorem~\ref{t.con}.

Given any leaf $\ell^s$ of $\cL^s$ and any leaf $\ell^u$ of $\cL^u$, the strong isotopy between the gluing maps ${\widebar\varphi}^m$ and $\varphi^m$ ensures that $\varphi^m(\ell^u)$ intersects $\ell^s$ if and only if ${\widebar\varphi}^m(\ell^u)$ intersects $\ell^s$. As a consequence, items~(1).(c) and~(1).(d) of Theorem~\ref{t.con} follow from the items (c) and (d) of Proposition~\ref{p.construction-psi}. This completes the proof of item~(1) of Theorem~\ref{t.con}.
\end{proof}

\subsection{Construction of periodic orbits}
\label{ss.periodic-orbits}

Our next goal is to construct some periodic orbits $\alpha_1^m,\dots,\alpha_{2n}^m$ of the Anosov flow $Y^m_t$ satisfying  item (2) of Theorem~\ref{t.con}. We will also construct analogous periodic orbits $\widebar\alpha_1^m,\dots,\widebar\alpha_{2n}^m$ of the flow $\widebar Y_t^m$, which will be useful during the proof of items (3) and (4) of Theorem~\ref{t.con}. We begin by selecting stable and unstable strips. 

\subsubsection{Selection of stable and unstable strips}
\label{sss.choice-strips}
A \emph{stable strip} of the plug $(U,X)$ is a connected component of $\partial^{in} U\setminus\cL^s=\partial^{in} U\setminus W^s(\Lambda)$. Due to the properties of the lamination $\cL^s$ (see item (7) of Theorem~\ref{t.fhplug}), any stable strip is the interstitial region between two non-compact leaves of $\cL^s$ contained in the same Reeb lamination annulus. Similarly, an \emph{unstable strip} of the plug $(U,X)$  is a connected component of $\partial^{out} U\setminus\cL^u=\partial^{out} U\setminus W^u(\Lambda)$, and every such unstable strip  is the interstitial region between two non-compact leaves of the lamination $\cL^u$ contained in the same Reeb lamination annulus. 

The involution $\sigma:U\righttoleftarrow$ maps every stable strip to an unstable strip, and vice-versa; this is a straightforward consequence of the equalities $\sigma(\partial^{in} U)=\partial^{out} U$ and $\sigma(\cL^s)=\cL^u$.

Since $\Lambda$ is the maximal invariant set of $(U,X)$, the forward $X_t$-orbit of any point $x\in\partial^{in}U\setminus\cL^s=\partial^{in} U\setminus W^s(\Lambda)$ eventually exits $U$ at some point $\Theta(x)\in \partial^{out}U\setminus \cL^u= \partial^{out}U\setminus W^u(\Lambda)$. This defines a map
$$\Theta:\partial^{in} U\setminus\cL^s\longrightarrow\partial^{out}U\setminus \cL^u$$
called the \emph{crossing map of the plug $(U,X)$}.  This crossing map is invertible since the backward orbit of any point in $\partial^{out}U\setminus \cL^u=\partial^{out}U\setminus W^u(\Lambda)$ eventually exits $U$. Thanks to the transversality of the vector field $X$ to the surface $\partial U$, the maps $\Theta$ and $\Theta^{-1}$ are continuous. It follows that $\Theta$ maps every stable strip to an unstable strip.

The following proposition is a key step in the construction of the periodic orbit $\alpha_j^m$.

\begin{proposition}
\label{p.selection-strips}
For $j\in\{1,\dots,2n\}$, there exist two stable strips $D_{2j-1}^s, D_{2j}^s$ such that
$$D_{2j-1}^s\subset A_{2j-1}^{1,s},\quad\Theta(D_{2j-1}^s)\subset A_{2j}^{1,u},\quad D_{2j}^s\subset A_{2j}^{1,s}, \quad\ \Theta(D_{2j}^s)\subset A_{2j-1}^{1,u},$$ 
$$D_{2j}^s=\sigma\circ\Theta(D_{2j-1}^s),\quad D_{2j-1}^s=\sigma\circ\Theta(D_{2j}^s).$$
\end{proposition}

\begin{figure}[htp]
\begin{center}
  \includegraphics[totalheight=7cm]{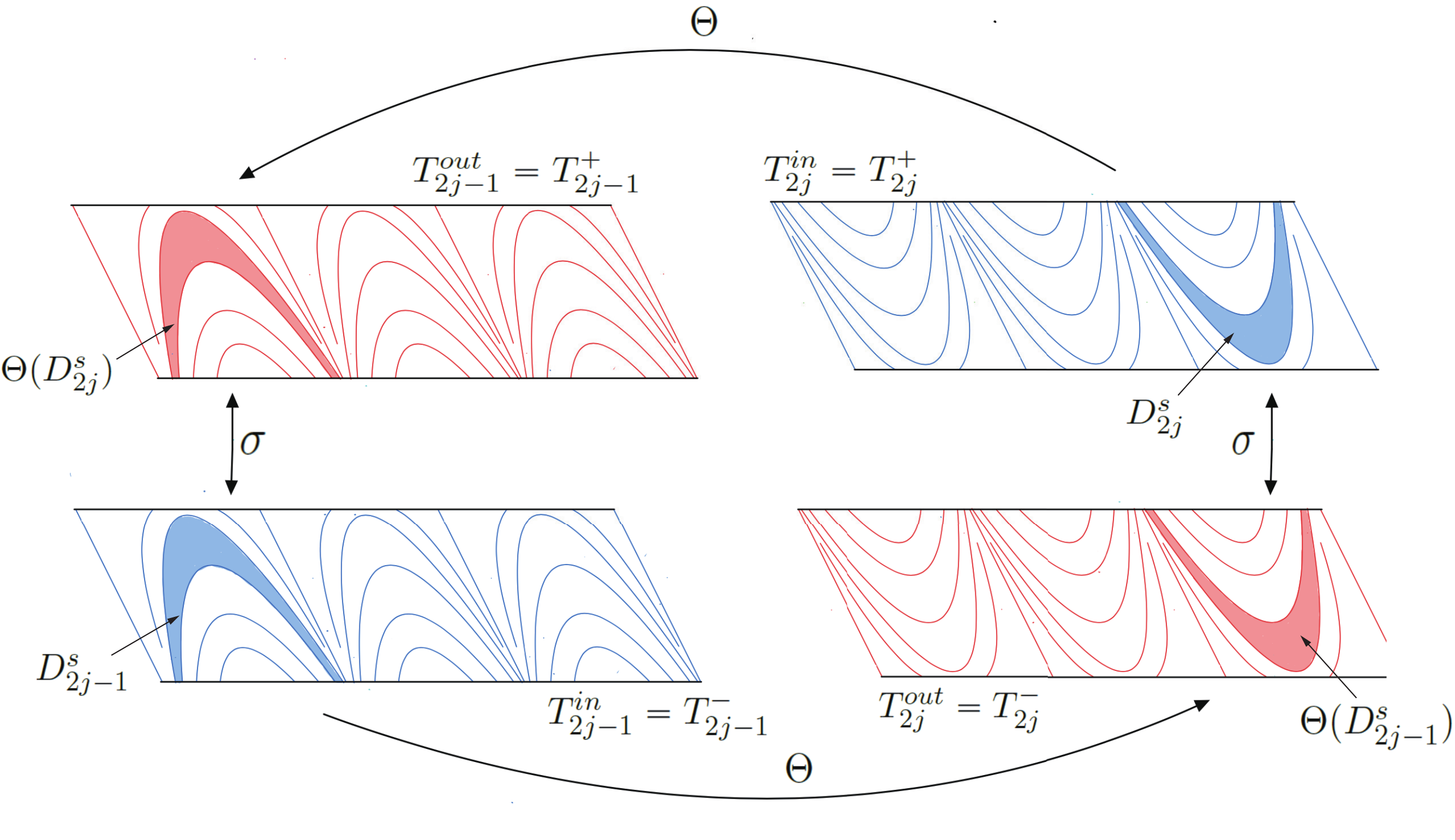}
  \caption{The stable strips $D_{2j-1}^s,D_{2j}^s$ and the unstable strips $\Theta(D_{2j-1}^s),\Theta(D_{2j}^s)$.}\label{f.Selstrips}
\end{center}
\end{figure}

The core of the proof is the following lemma:

\begin{lemma}
\label{l.selection-strip}
There exists a stable strip $D_{2j}^s\subset A_{2j}^{1,s}$ such that  $\Theta(D_{2j}^s)\subset A_{2j-1}^{1,u}$.
\end{lemma}

\begin{remark}
The choice of the annuli $A_{2j}^{1,s}$ and $A_{2j-1}^{1,u}$ is arbitrary: one can find a stable strip $D^s\subset A_{i}^{k,s}$ such that $\Theta(D^s)\subset A_{i'}^{k',u}$ for any $i,i',k,k'$ provided that $i,i'$ have opposite parity.
\end{remark}

\begin{proof}[Proof of Lemma~\ref{l.selection-strip}]
Note that, due to the parity of $2j$ and $2j-1$,  the Reeb annuli $A_{2j}^{1,s}$ and $A_{2j-1}^{1,u}$ sit in the boundary of $U^+$. The proof of the lemma will take place in $U^+$.

The compact leaf $c_{2j}^{1,s}$ is one of the two boundary components of the Reeb annulus $A_{2j}^{1,s}$. We know that  $c_{2j}^{1,s}=\partial^{in} U\cap W^s_0(\gamma_{2j}^{1,+})$ where $W^s_0(\gamma_{2j}^{1,+})$ is one of the two stable separatrices of the periodic u-boundary orbit $\gamma_{2j}^{1,+}$. We denote by $W^u_0(\gamma_i^{1,+})$ the unstable separatrix of the orbit $\gamma_{2j}^{1,+}$ which is on the same side of $W^{s}(\gamma_{2j}^{1,+})$ as the annulus $A_{2j}^{1,s}$. Similarly, the compact leaf $c_{2j-1}^{1,u}$ is a boundary component of the Reeb annulus $A_{2j-1}^{1,u}$. We know that  $c_{2j-1}^{1,u}=\partial^{out} U\cap W^u_0(\gamma_{2j-1}^{1,+})$ where $W^u_0(\gamma_{2j-1}^{1,+})$  is one of the two unstable separatrices of the periodic s-boundary orbit $\gamma_{2j-1}^{1,+}$. We denote by $W^s_0(\gamma_{2j-1}^{1,+})$ the unstable separatrix of $\gamma_{2j-1}^{1,+}$ which is on the same side of $W^u(\gamma_{2j-1}^{1,+})$ as the annulus $A_{2j-1}^{1,u}$. 

We claim that the stable lamination $W^s (\Lambda^+)$ intersects the unstable separatrix $W^u_0(\gamma_{2j}^{1,+})$. Indeed, the Reeb annulus $A_{2j}^{1,s}$ contains some non-compact leaves of the lamination $\cL^s_{2j}=W^s (\Lambda^+) \cap T_{2j}^+$ accumulating the compact $c_{2j}^{1,s}$. Due to the lamination structure of $W^s(\Lambda^+)$, to the transversality of $W^s (\Lambda^+)$ with $W^u (\Lambda^+)$ and to the fact that $W^u_0(\gamma_{2j}^{1,+})$ is on same side of $W^s(\gamma_{2j}^{1,+})$ as $A_{2j}^{1,s}$, we deduce that $W^s (\Lambda^+)$ intersects $W^u_0(\gamma_{2j}^{1,+})$. The claim is proved. 

Since $\Lambda^+$ is a maximal invariant set, it satisfies $\Lambda^+=W^s(\Lambda^+)\cap W^u(\Lambda^+)$. Together with the claim of the previous paragraph, this implies that  the unstable separatrix $W^u_0(\gamma_{2j}^{1,+})$ intersects $\Lambda$. Since $\Lambda^+$ is transitive (item~(6) of Theorem~\ref{t.fhplug}), it follows that $W^u_0(\gamma_{2j}^{1,+})$ is dense in $W^u (\Lambda^+)$. Similar arguments imply that $W^s_0(\gamma_{i-1}^{1,+})$ is dense in $W^s (\Lambda^+)$. The separtrices $W^u_0(\gamma_{2j}^{1,+})$ and $W^s_0(\gamma_{2j-1}^{1,+})$ are dense respectively in $W^u (\Lambda^+)$ and $W^s (\Lambda^+)$, the intersection $W^u_0(\gamma_{2j}^{1,+})\cap W^s_0(\gamma_{2j-1}^{1,+})$ is non-empty.

Let $x$ be a point in $W^u_0(\gamma_{2j}^{1,+})\cap W^s_0(\gamma_{2j-1}^{1,+})$. Since $W^u_0(\gamma_{2j}^{1,+})$ is a separatrix of a u-boundary periodic orbit, and $W^s_0(\gamma_{2j-1}^{1,+})$ is a separatrix of a s-boundary periodic orbit, there exists a connected component $D$ of $U\setminus (W^{u} (\Lambda^+)\cup W^{s} (\Lambda^+))$ such that $x$ is in the accessible boundary of $D$. See Figure \ref{f.selstrip} as an illustration. Then $D^s:=D\cap\partial^{in} U$ is a connected component of $\partial^{in} U^+\setminus (W^{u} (\Lambda^+)\cup W^{s} (\Lambda^+))=\partial^{in} U^+\setminus W^{s} (\Lambda^+)$. In other words, $D^s$ is a stable strip of $(U^+,X)$. Similarly, $D^u:=D\cap\partial^{out} U$ is an unstable strip. The $X_t$-invariance of $D$ implies that $D^u=\Theta(D^s)$. 

\begin{figure}[htp]
\begin{center}
  \includegraphics[totalheight=7cm]{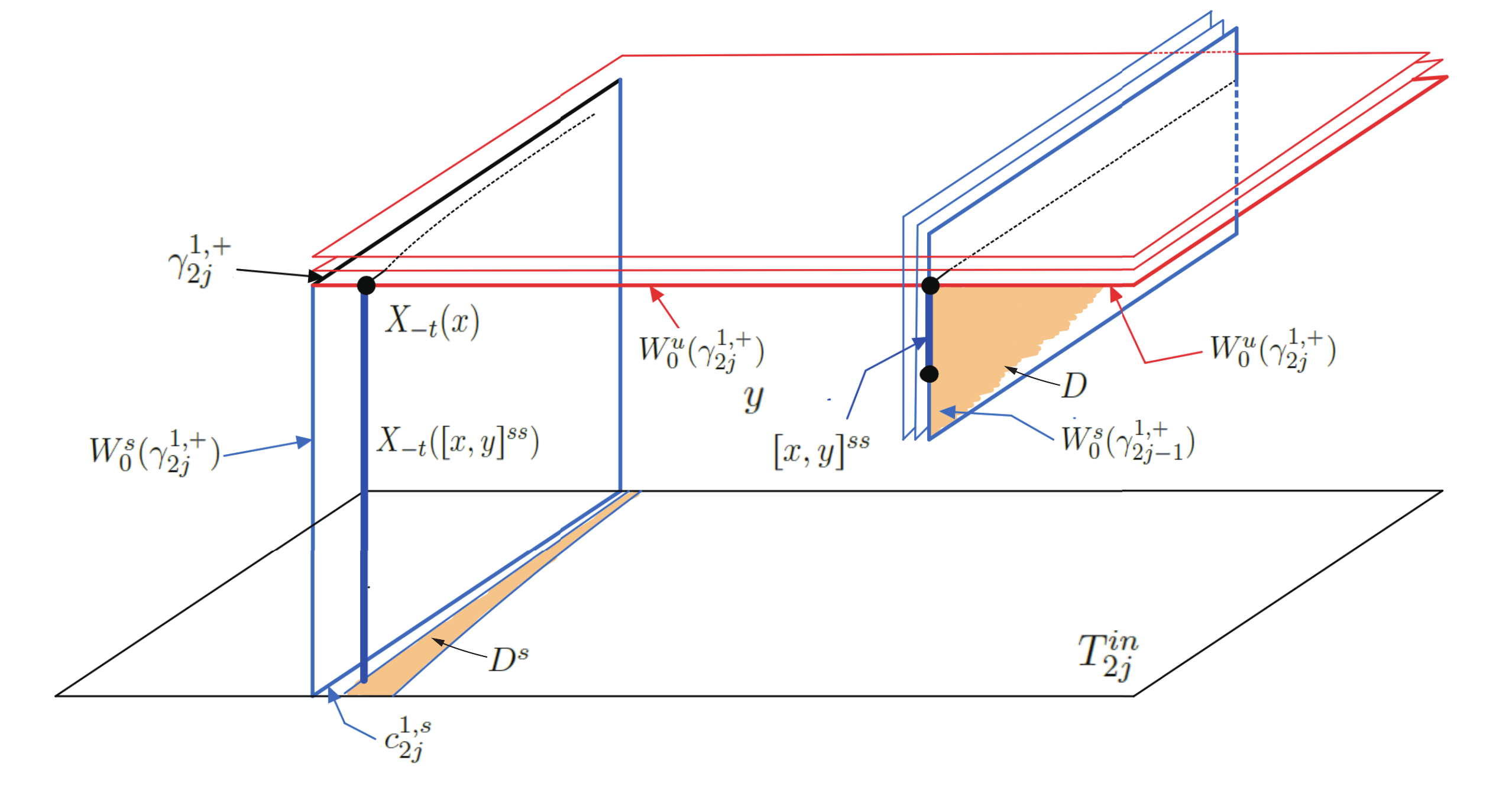}
  \caption{Proof of Lemma~\ref{l.selection-strip}.}\label{f.selstrip}
\end{center}
\end{figure}

Now we prove that the stable strip $D^s$ is contained in the annulus $A_{2j}^{1,s}$. For that purpose, we consider a strong stable arc $[x,y]^{ss}$ included in $W^s_0(\gamma_{2j-1}^{1,+})$ such that  $]x,y]^{ss}\cap W^u(\Lambda^+) =\emptyset$ (see Figure \ref{f.selstrip}). Observe that $[x,y]^{ss}$ is in the accessible boundary of $D$. The $\lambda$-Lemma implies that $\bigcup_{t=T}^{+\infty} X_{-t}(]x,y]^{ss})$ converges for the $C^1$ topology towards $W^s_0(\gamma_{2j}^{1,+})$ as $T\to +\infty$. Moreover, since $]x,y]^{ss} \cap W^u (\Lambda^+) =\emptyset$, the backward orbit of any point in $]x,y]^{ss}$ exits $U$ by crossing $\partial^{in} U^+$. It follows that $(\bigcup_{t=T}^{+\infty} X_{-t}(]x,y]^{ss}))\cap \partial^{in} U^+$ is a half leaf of the lamination $\cL^s$ accumulating $W^s_0(\gamma_{2j}^{1,+}) \cap \partial^{in} U =c_{2j}^{1,s}$. Moreover, the accumulation is on the side of $c_{2j}^{1,s}$ where the annulus $A_i^{1,s}$ sits (due to the choice of $W^u_0(\gamma_{2j}^{1,+})$).  Finally, since $]x,y]^{ss}$ is in the accessible boundary of $D$, and since $D$ is $(X_t)$-invariant, $(\bigcup_{t=T}^{+\infty} X_{-t}(]x,y]^{ss}))\cap \partial^{in} U^+$  must be in the accessible boundary of $D\cap \partial^{in} U^+ =D^s$. We conclude that $D^s$ accumulates $c_{2j}^{1,s}$ on the side of the annulus $A_{2j}^{1,s}$. Hence, $D^s\subset A_{2j}^{1,s}$.

By similar arguments, $D^u=\Theta(D^s)$ accumulates $c_{2j-1}^{1,u}$ on the side of the annulus $A_{2j-1}^{1,u}$. Hence, $\Theta(D^s)\subset A_{2j-1}^{1,u}$. The lemma is proved.
\end{proof}

\begin{proof}[Proof of Proposition~\ref{p.selection-strips}]
Lemma~\ref{l.selection-strip} provides us  with a stable strip $D_{2j}^s$ such that 
$$D_{2j}^s\subset A_{i}^{1,s}\quad\mbox{ and }\quad\Theta(D_{2j}^s)\subset A_{2j-1}^{1,u}.$$
We define a stable strip $D_{2j-1}^s$ by setting
$$D_{2j-1}^s:=\sigma\circ\Theta(D_{2j}^s).$$
Observe that $D_{2j-1}^s$ is indeed a stable strip since $\Theta$ maps every stable strip to an unstable strip and $\sigma$ maps every unstable strip to a stable strip. The indexing of the Reeb annuli of the laminations $\cL^s$ and $\cL^u$ was chosen so that $\sigma(A_{i}^{k,u})=A_{i}^{k,s}$ for every $i$ and $k$. It follows that 
$$D_{2j-1}^s\subset\sigma(A_{2j-1}^{1,u})=A_{2j-1}^{1,s}.$$
Now recall that the involution $\sigma$ satisfies $\sigma_*X=-X$ (item~(3) of Theorem~\ref{t.fhplug}). Hence $\sigma$ conjugates the crossing map $\Theta$ to its inverse, which is equivalent to $\sigma\circ\Theta\circ\sigma\circ\Theta=\mathrm{Id}$. It follows that
$$D_{2j}^s:=\sigma\circ\Theta(D_{2j-1}^s).$$
As a further consequence, 
$$\Theta(D_{2j-1}^s)=\sigma(D_{2j}^s)\subset\sigma(A_{2j}^{1,s})=A_{2j}^{1,u}.$$
Therefore the strips $D_{2j-1}^s$ and $D_{2j}^s$ satisfy all the required properties.
\end{proof}

\subsubsection{Selection of connected components in the intersection of two strips}
 
Fix an integer $j\in\{1,\dots,2n\}$ and consider the strips $D_{2j-1}^s,D_{2j}^s$ provided by Proposition~\ref{p.selection-strips}. Denote 
$$D_{2j-1}^u:=\Theta(D_{2j}^s)\mbox{ and }D_{2j}^u:=\Theta(D_{2j-1}^s).$$ 
Recall that $D_{2j-1}^s, D_{2j}^s$ are stable strips in the  annuli $A_{2j-1}^{1,s}, A_{2j}^{1,s}$ respectively, and $D_{2j-1}^u, D_{2j}^u$ are unstable strips in the annuli $A_{2j-1}^{1,u}, A_{2j}^{1,u}$ respectively.

As explained previously (see the end of the proof of Proposition~\ref{p.construction-psi}), the annulus $\tau_{1/2}\circ\sigma(A_{2j-1}^{1,u})$ intersects the annulus $A_{2j-1}^{1,s}$. Recall that $\tau_{1/2}\circ\sigma(A_{2j-1}^{1,u})$ is a Reeb annulus of the foliation $f^{u,in}_{2j-1}$ and $A_{2j-1}^{1,s}$ is a Reeb annulus of the foliation $f^{s,in}_{2j-1}$. Due to the geometry of the pair of foliations $(f^{s,in}_{2j-1},f^{u,in}_{2j-1})$ (recall that it is homeomorphic to the model $(\ff^s_{2j-1},\ff^u_{2j-1})$), it follows that the image under $\tau_{1/2}\circ\sigma$ of any unstable strip in $A_{2j-1}^{1,u}$ intersects any stable strip in $A_{2j-1}^{1,s}$. In particular, $\tau_{1/2}\circ\sigma(D_{2j-1}^u)$ intersects $D_{2j-1}^s$. Similar arguments imply that $\tau_{1/2}\circ\sigma(D_{2j}^u)$ intersects $D_{2j}^s$. Due to the spiraling of the strips along the compact leaves, both the intersections have infinitely many connected components. We choose arbitrarily:
\begin{itemize}
\item[--] a connected component $\widebar R_{2j-1}$ of $(\tau_{1/2}\circ\sigma)(D_{2j-1}^u)\cap D_{2j-1}^s$,
\item[--] a connected component $\widebar R_{2j}$ of  $(\tau_{1/2}\circ\sigma)(D_{2j}^u)\cap D_{2j}^s$.
\end{itemize}
Now, we consider the map $\vartheta:\partial^{in} U\righttoleftarrow$ so that $\vartheta_{|T_i^{in}}$ is defined by 
$$\xi_i\circ \vartheta\circ \xi_i^{-1}(x,y)=(1-x,y).$$
This formula means that $\vartheta_{|T_i^{in}}$ is the axial symmetry with respect to the core of the annulus $A_i^{1,s}$. We claim that, for $i=2j-1,2j$,
\begin{equation}
\label{e.action-vartheta}
\vartheta\left( (\tau_{1/2}\circ\sigma)(D_{i}^u)\cap D_{i}^s\right)=(\tau_{-1/2}\circ\sigma)(D_{i}^u)\cap D_{i}^s.
\end{equation}
Let us prove this equality. For $i=2j-1$ or $2j$, the symmetry $\vartheta$ leaves the foliation $f^{s,in}_i$ globally invariant, hence must preserves every leaf of $f^{s,in}_i$ which intersects its axis. Hence $\vartheta$ must preserves every stable strip in the annulus $A_i^{1,s}$. In particular, $\vartheta(D_i^s)=D_i^s$. Now recall that the strips $D_i^s$ and $D_i^u$ were chosen so that $\sigma(D_i^u)=D_i^s$. It follows that $\vartheta(\sigma(D_i^u))=\vartheta(D_i^s)=D_i^s=\sigma(D_i^u)$. Also notice that the definitions $\vartheta$ and $\tau_t$ imply that $\vartheta\circ\tau_{1/2}\vartheta=\tau_{-1/2}$. We conclude that $\vartheta\left((\tau_{1/2}\circ\sigma)(D_{2j-1}^u)\right)=\tau_{-1/2}\circ\sigma(D_{2j-1}^u)$. Equality~\eqref{e.action-vartheta} are proved. See Figure
\ref{f.1}. Now, we define
 $$\widebar L_{2j-1}:=\vartheta(\widebar R_{2j-1})\quad\mbox{and}\quad \widebar L_{2j}:=\vartheta(\widebar R_{2j}).$$
Equality~\eqref{e.action-vartheta} implies that 
\begin{itemize}
\item[--] $\widebar L_{2j-1}$ is a connected component of $(\tau_{-1/2}\circ\sigma)(D_{2j-1}^u)\cap D_{2j-1}^s$,
\item[--] $\widebar L_{2j}$ is a connected component of $(\tau_{-1/2}\circ\sigma)(D_{2j}^u)\cap D_{2j}^s$.
\end{itemize}

\begin{figure}[htp]
\begin{center}
  \includegraphics[totalheight=5.8cm]{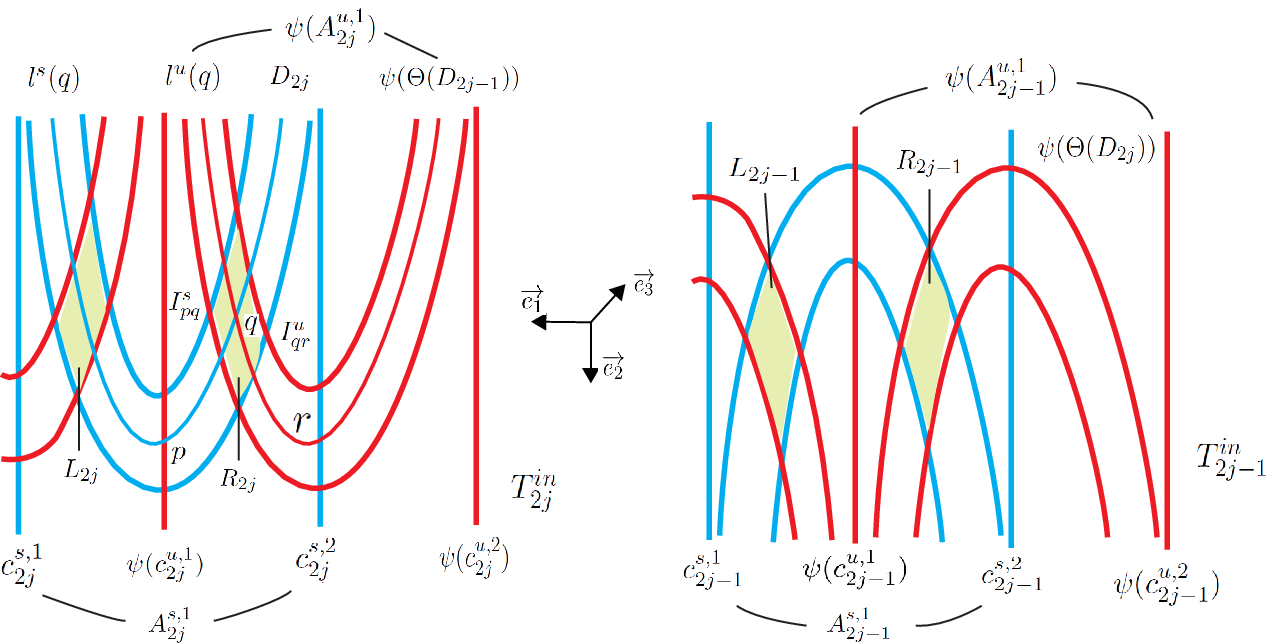}\\
  \caption{$\widebar R_{2j}$, $\widebar R_{2j-1}$, $\widebar L_{2j}$, $\widebar L_{2j-1}$}\label{f.1}
\end{center}
\end{figure}

\begin{remark}
The reason why we demand $\widebar L_{2j-1}$ and $\widebar L_{2j}$ to be the images of $\widebar R_{2j-1}$ and $\widebar R_{2j}$ under $\vartheta$ will appear at the very end of the proof of item~(4) of Theorem~\ref{t.con}. For the sole purpose of proving item~(2) of Theorem~\ref{t.con}, it would have been enough to define $\widebar L_{i}$ to be arbitrary  connected components of $\tau_{-1/2}\circ\sigma(D_{i}^u)\cap D_{i}^s$.
\end{remark}

\subsubsection{Construction of the periodic orbit $\widebar\alpha_j^1,\dots,\widebar\alpha_j^{2n}$ of the flow $\widebar Y_t^m$}
\label{sss.construction-beta}

Let $\widebar\pi^m$ be the natural projection of $U$ onto $\widebar W^m=U/{\widebar\varphi}^m$. For $i=1,\dots,4n$, let $\widebar T_{i}^m:=\widebar\pi^m (T_i^{out})=\widebar\pi^m (T_i^{in})$. The following statement is an analog of item~(2) of Theorem~\ref{t.con} for the flow $\widebar Y_t^m$.

\begin{proposition}
\label{p.construction-beta}
For $m=0, \dots, 2n$ and $j=1,\dots,2n$, there exists a periodic orbit $\widebar\alpha_j^m$ of the flow $\widebar Y_t^m$ such that:
\begin{enumerate}
\item $\widebar\alpha_j^m$ intersects the torus $\widebar T_{2j-1}^m$ at one point $\widebar q_{2j-1}^m\in \mathrm{int}\left(\widebar\pi^m (A_{2j-1}^{1,s}) \cap \widebar\pi^m (A_{2j-1}^{1,u})\right)$,
\item $\widebar\alpha_j^m$ intersects the torus $\widebar T_{2j}^m$ at  one point $\widebar q_{2j}^m\in \mathrm{int}\left(\widebar\pi^m (A_{2j}^{1,s}) \cap \widebar\pi^m (A_{2j}^{1,u})\right)$,
\item  $\widebar\alpha_j^m$ does not intersect the torus $\widebar T_{i}^m$ for $i\neq 2j-1,2j$.
\end{enumerate}
\end{proposition}

\begin{proof}
Fix an integer $m\in\{0,\dots,2n\}$ and an integer $j\in\{1,\dots,2n\}$. Let $D_{2j-1}^s,D_{2j}^s$ be the strips provided by Proposition~\ref{p.selection-strips} and recall that  $D_{2j-1}^u=\Theta(D_{2j}^s)\mbox{ and }D_{2j}^u=\Theta(D_{2j-1}^s)$.

Let us first consider the case $j>m$. The definitions of the gluing map ${\widebar\varphi}^m$ and of the unstable strips $D^u_{2j-1},D^u_{2j}$ entails that
$$({\widebar\varphi}^m\circ\Theta)(D_{2j}^s)= (\tau_{+1/2}\circ\sigma)(D_{2j-1}^u)\quad\mbox{and}\quad ({\widebar\varphi}^m\circ\Theta)(D_{2j-1}^s)= (\tau_{+1/2}\circ\sigma)(D_{2j}^u).$$
As a consequence, 
\begin{itemize}
\item[--] $\widebar R_{2j-1}$ is a connected component of $({\widebar\varphi}^m\circ\Theta)(D_{2j}^s)\cap D_{2j-1}^s$,
\item[--] $\widebar R_{2j}$ is a connected component of $({\widebar\varphi}^m\circ\Theta)(D_{2j-1}^s)\cap D_{2j}^s$.
\end{itemize}
Now observe that $D_{2j-1}^s$ is a strip bounded by two leaves of the foliation $f^{s,in}_{2j-1}$ and $({\widebar\varphi}^m\circ\Theta)(D_{2j}^s)=(\tau_{+1/2}\circ\sigma)(D_{2j-1}^u)$ is a strip bounded by two leaves of foliation $f^{u,in}_{2j-1}$. This implies that every connected component of the intersection $({\widebar\varphi}^m\circ\Theta)(D_{2j}^s)\cap D_{2j-1}^s$ is markovian. Similarly,  every connected component of the intersection $({\widebar\varphi}^m\circ\Theta)(D_{2j-1}^s)\cap D_{2j}^s$ is topologically markovian. It follows that there exists a point $\widehat q_{2j-1}\in \widebar R_{2j-1}$ such that $\widehat q_{2j}:=({\widebar\varphi}^m\circ\Theta)(\widehat q_{2j-1})\in \widebar R_{2j}$ and $({\widebar\varphi}^m\circ\Theta)(\widehat q_{2j})=\widehat q_{2j-1}$.  We define $\widebar\alpha_j^m$ to the orbit of the point $\widebar q_{2j-1}^m:=\widebar\pi^m(\widehat q_{2j-1})$. By construction, $\widebar\alpha_j^m$ is periodic, intersects $\widebar T_{2j-1}^m$ at $\widebar q_{2j-1}^m\in\widebar\pi^m(\widebar R_{2j-1})$,   intersects $\widebar T_{2j-1}^m$ at $\widebar q_{2j}^m:=\widebar\pi^m(\widehat q_{2j})\in\widebar\pi^m(\widebar R_{2j})$, and does not not intersect the torus $\widebar T_{i}^m$ for $i\neq 2j-1,2j$.

For $j\leq m$, we perform exactly the same construction. The only difference is that we must replace $\tau_{+1/2}$ by $\tau_{-1/2}$ and therefore replace the $\widebar R_{2j-1}$ and $\widebar R_{2j}$ by $\widebar L_{2j-1}$ and $\widebar L_{2j}$. At the end of the day, we obtain that the orbit $\widebar\alpha_j^m$ is periodic, intersects $\widebar T_{2j-1}^m$ at one point $\widebar q_{2j-1}^m\in\widebar\pi^m(L_{2j-1})$,  intersects $\widebar T_{2j}^m$ at one point $\widebar q_{2j}^m\in\widebar\pi^m(L_{2j})$, and does not intersect the torus $\widebar T_{i}^m$ for $i\neq 2j-1,2j$.
\end{proof}

\begin{remark}
The map ${\widebar\varphi}^m\circ\Theta\circ{\widebar\varphi}^m\circ\Theta$ preserves the foliations $f_{2j-1}^{s,in}$ and $f_{2j-1}^{u,in}$, but does not necessarily contracts/expands the directions tanngent to the leaves of these foliations. due to this lack of contraction/expansion, the point $\widebar q_{2j-1}$ which appears in the previous proof is not necessarily unique.  As a consequence, the orbit $\widebar\alpha_j^m$ is not necessarily unique.
\end{remark}

\begin{remark}
\label{r.beta-indep-m}
Let $m,m'\in\{1,\dots,2n\}$ and $j\in\{1,\dots,2n\}$. Assume that $j>\max\{m,m'\}$ or $j\leq\min\{m,m'\}$.  Then the gluing diffeomorphisms ${\widebar\varphi}^0$ and ${\widebar\varphi}^m$ coincide on $T_{2j-1}^{out}$ and $T_{2j}^{out}$. In view of the construction of the orbits $\widebar\alpha_j^m$ and $\widebar\alpha_j^{m'}$, it follows that $\widebar\alpha_j^m$ and $\widebar\alpha_j^{m'}$ are the projections of the same orbits arcs of $X_t$ in $U$, \emph{i.e.} 
$$(\widebar\pi^m)^{-1}(\widebar\alpha_j^m)=(\widebar\pi^m)^{-1}(\widebar\alpha_j^{m'}).$$ 
\end{remark}

\subsubsection{Construction of the periodic orbits $\alpha_j^1,\dots,\alpha_j^{2n}$ of the flow $Y_t^m$}
\label{sss.construction-alpha}

\begin{proof}[Proof of item~(2) of Theorem~\ref{t.con}]
Fix an integer $m\in\{0,\dots,2n\}$. By definition, the gluing map $\varphi^m$ is strongly isotopic to the gluing map ${\widebar\varphi}^m$. This means that there is a continuous family $(\chi_s^m)_{s\in [0,1]}$ of strongly transverse gluing maps such that $\chi_0^m={\widebar\varphi}^m$, $\chi_1^m=\varphi^m$. 

Let $j\in\{1,\dots,2n\}$. First consider the case where $j>m$. Thanks to strongly transversality, we may define, for every $s\in [0,1]$,\begin{itemize}
\item[--] a connected component $R_{2j-1,s}^m$ of $(\chi_s^m\circ\Theta)(D_{2j}^s)\cap D_{2j-1}^s$,
\item[--] a connected component $R_{2j,s}^m$ of $(\chi_s^m\circ\Theta)(D_{2j-1}^s)\cap D_{2j}^s$,
\end{itemize}
so that $R_{2j-1,s}^m,R_{2j,s}^m$ depend continuously on $s$ and coincide with $\widebar R_{2j-1},\widebar R_{2j}$ for $s=0$. We set
$$R_{2j-1}^m:=R_{2j-1,1}^m\quad\mbox{and}\quad R_{2j}^m:=R_{2j,1}^m.$$
Strong isotopy also implies that every connected component of the intersections $(\chi_s^m\circ\Theta)(D_{2j}^s)\cap D_{2j-1}^s$ and $(\chi_s^m\circ\Theta)(D_{2j-1}^s)\cap D_{2j}^s$ remains markovian as $s$ varies. So we may construct the orbit $\alpha_j^m$ exactly as we did for the orbit $\widebar\alpha_j^m$, replacing $\widebar R_{2j-1}$ and $\widebar R_{2j}$ by $R_{2j-1}^m$ and $R_{2j}^m$. The case where $j\leq m$ can be treated similarly using $\widebar L_{2j-1}$ and $\widebar L_{2j}$ instead of $\widebar R_{2j-1}$ and $\widebar R_{2j}$.
\end{proof}

\begin{remark}
Since $Y^m_t$ is an Anosov flow, the map $\varphi^m\circ\Theta$ enjoys hyperbolicity properties. It follows that the orbit $\alpha_j^m$ is unique, provided that we require its intersections with $T_{2j-1}^m$ and $T_{2j}^m$ to be in $R_{2j-1}^m$ and  $R_{2j}^m$ if $j>m$ (respectively in $L_{2j-1}^m$ and  $L_{2j}^m$ if $j\leq m$). We will not use this uniqueness property.
\end{remark}

\subsection{Topology of the local stable manifold $W^s_{loc}(\alpha_j^0)$}
\label{ss.local-stable-annulus}

Our next goal is to prove item~(3) of Theorem~\ref{t.con}, that is to prove that the local stable manifold $W^s_{loc}(\alpha_j^0)$ of the periodic orbit $\alpha_j^0$ constructed above is an annulus (not a Möbius band) for every $j$. The proof will be divided into two steps. In the first step, we prove that the ``local stable manifold" $W^s_{loc}(\widebar\alpha_j^0)$ of the periodic orbit $\widebar\alpha_j^0$ of the (non-Anosov) flow $\widebar Y_t^0$ is an annulus. In the second step, we prove that $W^s_{loc}(\widebar\alpha_j^0)$ is homeomorphic to $W^s_{loc}(\alpha_j^0)$. 

\begin{remark}
The arguments we give below can be adapted to prove that the local stable manifold $W^s_{loc}(\alpha_j^m)$ of the periodic orbit $\alpha_j^m$ of the Anosov flow $Y_t^m$ is an annulus for every $j$ and every $m$. We stick to the case $m=0$ for sake of simplicity. The case $m\geq 1$ will follow from the case $m=0$ and item~(4) of Theorem~\ref{t.con} which entails that  $W^s_{loc}(\alpha_j^m)$ is homeomorphic to $W^s_{loc}(\alpha_j^0)$ for every $m$.
\end{remark}

\subsubsection{Topology of the ``local stable manifold" $W^s_{loc}(\widebar\alpha_j^0)$}
In Subsection~\ref{ss.gluing-maps}, we have defined a gluing map ${\widebar\varphi}^0$. The flow $\widebar Y_t^0$ induced by $X_t$ on the closed manifold $\widebar W^0=U/{\widebar\varphi}^0$ is not necessarily an Anosov flow, but it preserves a pair of transverse foliations $(\widebar\cF^{0,s},\widebar\cF^{0,u})$. In Subsection~\ref{ss.periodic-orbits}, we have constructed some periodic orbits $\widebar\alpha_1^0,\dots,\widebar\alpha_{2n}^0$ of the flow $\widebar Y_t^0$. For $j=1,\dots,2n$, we will denote by $W^s_{loc}(\widebar\alpha_j^0)$ a small tubular neighborhood of $\widebar\alpha_j^0$, which is in the leaf of $\widebar\cF^{0,s}$ containing this periodic orbit\footnote{This is an abuse of notation: $W^s_{loc}(\widebar\alpha_j^0)$ is not necessarily a local stable set of $\widebar\alpha_j^0$ since $\widebar Y_t^0$ is not necessarily an Anosov flow.}. Being a tubular neighborhood of a closed curve in a two-dimensional manifold, $W^s_{loc}(\widebar\alpha_j^0)$ is either an annulus or a Möbius band. We will prove the latter
case is impossible.

\begin{proposition}\label{p.annulus-psi}
For $j=1,\dots, 2n$, $W_{loc}^s(\widebar\alpha_j^0)$ is an annulus (not a Möbius band).
\end{proposition}

\begin{proof}
The very rough idea of the proof is that $W_{loc}^s(\widebar\alpha_j^0)$ can be decomposed in two halves, which are in some sense symmetric. 

More precisely, recall that $\widebar\alpha_j^0$ intersects $\widebar T_{2j-1}^0$ at $\widebar q_{2j-1}^0$, intersects  $\widebar T_{2j}^0$ at $\widebar q_{2j}^0$, and does not intersect $\widebar T_i^0$ for $i\neq 2j,2j-1$. And recall that $\widehat q_{2j-1}^0,\widehat q_{2j}^0$ are the lifts of $\widebar q_{2j-1}^0,\widebar q_{2j}^0$ in $T_{2j-1}^{in},T_{2j}^{in}$ respectively. Notice that $\widehat q_{2j-1}^0=({\widebar\varphi}^0\circ\Theta)(\widehat q_{2j}^0)$ and  $\widehat q_{2j}^0=({\widebar\varphi}^0\circ\Theta)(\widehat q_{2j-1}^0)$. Denote by $\ell^s_{2j-1}$ and $\ell^s_{2j}$ the leaves of the foliation $f_{2j-1}^{s,in}$ and $f_{2j}^{s,in}$ containing the points $\widehat q_{2j-1}^0$ and $\widehat q_{2j}^0$ respectively. Notice that  $I_{2j-1}^s:=({\widebar\varphi}^0\circ\Theta)(\ell^s_{2j})\subset\ell^s_{2j-1}$  is a neighborhood of $\widehat q_{2j-1}^0$ in $\ell^s_{2j-1}$ and $I_{2j}^s:={\widebar\varphi}^0\circ\Theta(\ell^s_{2j-1})\subset\ell^s_{2j}$ is a neighborhood of $\widehat q_{2j}^0$ in $\ell^s_{2j}$ (see Figure~\ref{f.Annulus-not-Mobius-1}).

\begin{figure}[htp]
\begin{center}
  \includegraphics[totalheight=9cm]{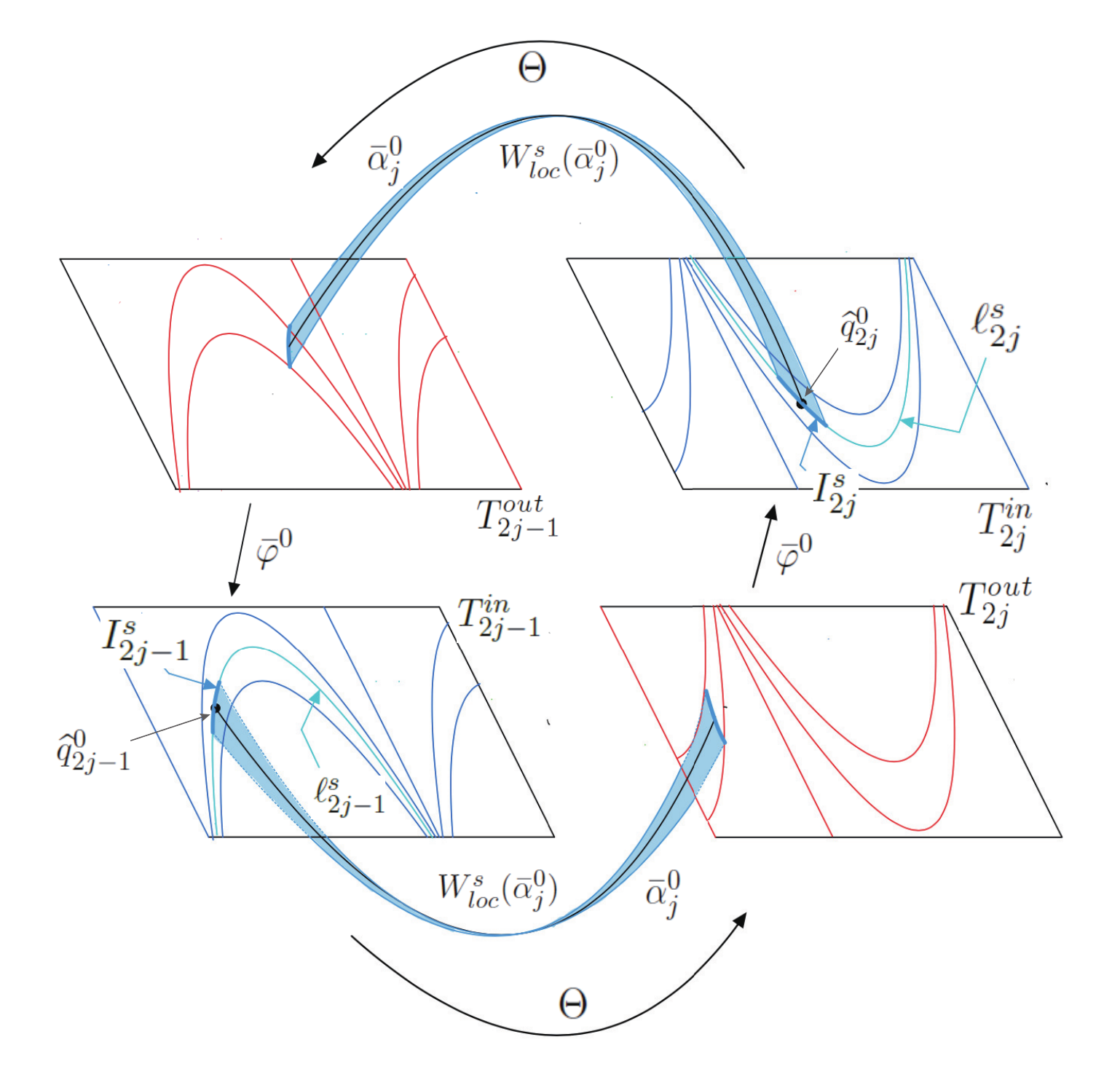}\\
  \caption{Decomposition of $W^s_{loc}(\widebar\alpha_j^0)$}\label{f.Annulus-not-Mobius-1}
\end{center}
\end{figure}

If we choose some orientations of the leaves $\ell^s_{2j-1}$ and $\ell^s_{2j}$, then Proposition~\ref{p.annulus-psi} is equivalent to the following claim: 

\begin{claim}
\label{c.annulus-psi}
If ${\widebar\varphi}^0\circ\Theta:\ell^s_{2j}\hookrightarrow \ell^s_{2j-1}$ preserves orientations then \hbox{${\widebar\varphi}^0\circ\Theta:\ell^s_{2j-1}\hookrightarrow \ell^s_{2j}$} also does, and vice versa. 
\end{claim}

Le us prove this claim. The leaf $\ell^s_{2j-1}$ is contained in the Reeb annulus $A_{2j-1}^{1,s}$. As a consequence, one end of $\ell^s_{2j-1}$ is spiraling around the compact leaf  $c_{2j-1}^{1,s}$ and the other end is spiraling around the compact leaf $c_{2j-1}^{2,s}$. We choose to endow $\ell^s_{2j-1}$ with the orientation pointing towards $c_{2j-1}^{2,s}$. Similarly, we choose to endow $\ell^s_{2j}$ with the orientation pointing towards $c_{2j}^{2,s}$. Hence, ${\widebar\varphi}^0\circ\Theta:\ell^s_{2j}\hookrightarrow \ell^s_{2j-1}$ preserves the orientations if and only if 
\begin{enumerate}
\item[(1)] ${\widebar\varphi}^0\circ\Theta$ maps $\ell^s_{2j}$ oriented towards $c_{2j}^{2,s}$ to $I^s_{2j-1}$ oriented towards $c_{2j-1}^{2,s}$.
\end{enumerate}
Composing by $({\widebar\varphi}^0)^{-1}$, we see that (1) is equivalent to
\begin{enumerate}
\item[(2)] $\Theta$ maps $\ell^s_{2j}$ oriented towards $c_{2j}^{2,s}$ to $({\widebar\varphi}^0)^{-1}(I^s_{2j-1})$ oriented towards $({\widebar\varphi}^0)^{-1}(c_{2j-1}^{2,s})$.
\end{enumerate}
Here comes the first key ingredient of the proof. Due to item~(2) of Theorem~\ref{t.fhplug}, the involution $\sigma:U\righttoleftarrow$ conjugates the crossing map $\Theta$ to its inverse $\Theta^{-1}$. As a consequence (using the equality  ${\widebar\varphi}^0=\tau_{1/2}\circ\sigma$), (1) is equivalent to
\begin{enumerate}
\item[(3)] $\Theta^{-1}$ maps $\sigma(\ell^s_{2j})$ oriented towards $\sigma(c_{2j}^{2,s})$ to $\tau_{-1/2}(I^s_{2j-1})$ oriented towards $\tau_{-1/2}(c_{2j-1}^{2,s})$.
\end{enumerate} 
A map preserves orientation if and only if its inverse does. Hence (2) is equivalent to
\begin{enumerate}
\item[(4)] $\Theta$ maps $\tau_{-1/2}(I^s_{2j-1})$ oriented towards $\tau_{-1/2}(c_{2j-1}^{2,s})$ to $\sigma(\ell^s_{2j})$ oriented towards $\sigma(c_{2j}^{2,s})$.
\end{enumerate} 
Composing by ${\widebar\varphi}^0$ and using again the equality ${\widebar\varphi}^0=\tau_{1/2}\circ\sigma$, (3) is equivalent to
\begin{enumerate}
\item[(5)] ${\widebar\varphi}^0\circ\Theta$ maps $\tau_{-1/2}(I^s_{2j-1})$ oriented towards $\tau_{-1/2}(c_{2j-1}^{2,s})$ to $\tau_{1/2}(\ell^s_{2j})$ oriented towards $\tau_{1/2}(c_{2j}^{2,s})$.
\end{enumerate} 
Now comes the second key point of the proof. Recall that the ``translation" $\tau_t$ was defined so that the $\tau_{1/2}(f_i^{s,in})=\tau_{-1/2}(f_i^{s,in})=f_i^{u,in}$.  Hence (4) concerns the action of the map ${\widebar\varphi}^0\circ\Theta$ on the orientation of certain unstable leaves. We will deduce the action of the same map on the orientation of the stable leaves. Recall that Proposition \ref{p.orient} allows to identify some direct bases $(\vec{e}_1,\vec{e}_2,\vec{e}_3)$ of $U$ with $\vec{e}_3=X$. The tori $T_{2j-1}^{in}$ and $T_{2j}^{in}$ are oriented by $(\vec{e}_1,\vec{e}_2)$. Observing Figure~\ref{f.Annulus-not-Mobius-2}, one notices that:
\begin{itemize}
\item[--] if $\vec{e}_s$ is tangent to $\ell^s_{2j-1}$ pointing towards $c_{2j-1}^{2,s}$ and $\vec{e}_u$ is  tangent to $\tau_{-1/2}(\ell^s_{2j-1})$ pointing towards $\tau_{-1/2}(c_{2j-1}^{2,s})$, then $(\vec{e}_s,\vec{e}_u)$ is a direct basis of $T_{2j-1}^{in}$;
\item[--] if $\vec{e}_s$ is tangent to $\ell^s_{2j}$ pointing towards $c_{2j}^{2,s}$ and $\vec{e}_u$ is  tangent to $\tau_{1/2}(\ell^s_{2j})$ pointing towards $\tau_{1/2}(c_{2j}^{2,s})$, then $(\vec{e}_s,\vec{e}_u)$ is a direct basis of $T_{2j}^{in}$.
\end{itemize}
On the other hand, ${\widebar\varphi}^0\circ\Theta:\partial^{in} U\to\partial^{in} U$ preserves the 2-dimensional orientations. As a consequence, (5) is equivalent to 
 \begin{enumerate}
\item[(6)] ${\widebar\varphi}^0\circ\Theta$ maps $\ell^s_{2j-1}$ oriented towards $c_{2j-1}^{2,s}$ to $I^s_{2j}$ oriented towards $c_{2j}^{2,s}$.
\end{enumerate} 
Due to our choice of orientations, (6) means that ${\widebar\varphi}^0\circ\Theta:\ell^s_{2j-1}\to I^s_{2j-1}$ preserves the orientations. This completes the proof of the claim and of Proposition~\ref{p.annulus-psi}.
\end{proof}

\begin{figure}[htp]
 \hspace{-1cm} \includegraphics[totalheight=5.5cm]{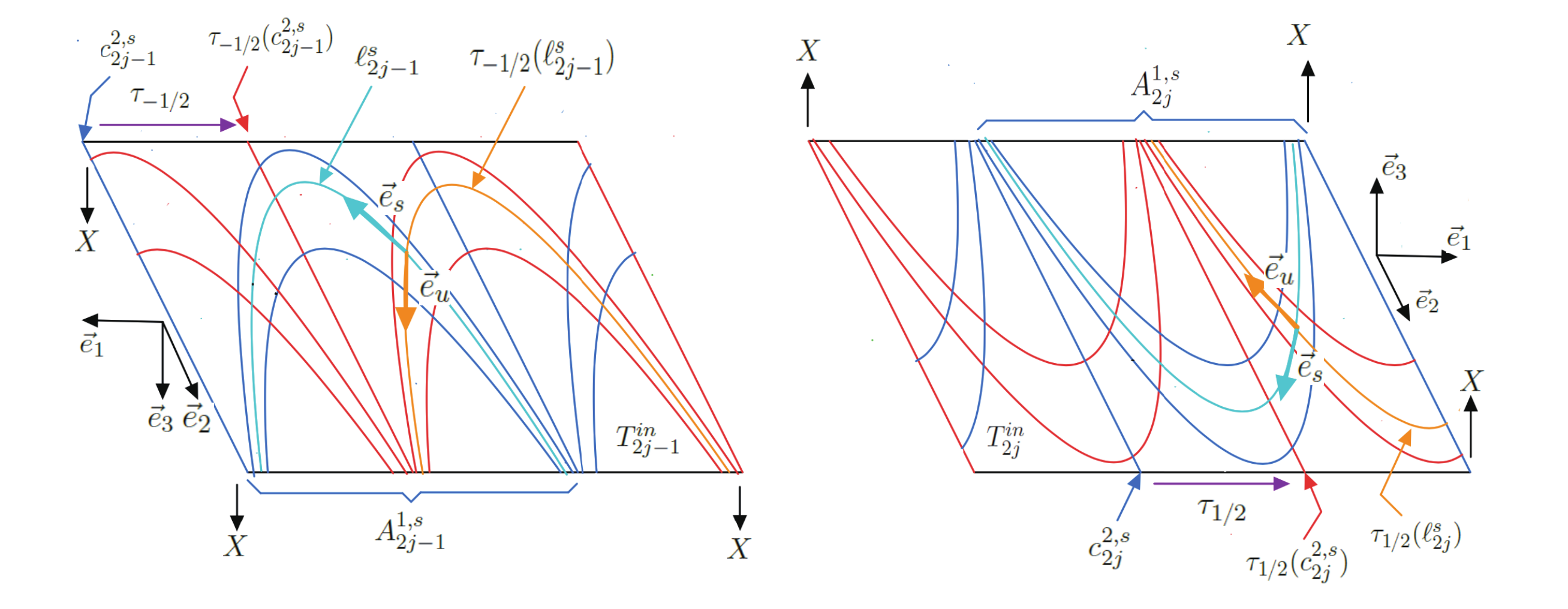}
  \caption{The basis $(\vec{e}_s,\vec{e}_u)$ in the proof of of Claim~\ref{c.annulus-psi}}\label{f.Annulus-not-Mobius-2}
\end{figure}

\subsubsection{Homeomorphism between $W^s_{loc}(\widebar\alpha_j^m)$ and $W^s_{loc}(\alpha_j^m)$.}

The second step of our proof of item~(3) of Theorem~\ref{t.con} consists in showing that $W^s_{loc}(\alpha_j^0)$ is homeomorphic to $W^s_{loc}(\widebar\alpha_j^0)$ for every $j$. We will actually prove the following stronger statement.  

\begin{proposition}
\label{p.relation-alpha-beta}
For $m=0,\dots,2n$, there exists a homeomorphism $h^m:W^m\to \widebar W^m$ such that $h^m(W^s_{loc}(\alpha_j^m))=W^s_{loc}(\widebar\alpha_j^m)$ for $j=1,\dots,2n$.
\end{proposition}

For sake of simplicity, we will give a detailed proof in the case $m=0$, and explain afterwards how it should be adapted in the case $m\geq 1$. We consider a continuous one-parameter family of strongly transverse gluing maps $(\chi_s^0)_{s\in [0,1]}$ such that $\chi_0^0 = {\widebar\varphi}^0$ and $\chi_1^0 = \varphi^0$. For every $s\in [0,1]$, we denote by $Y_{t,s}^0$ to be the flow induced by $X_t$ on  the closed manifold $W_s^0 := U/\chi_s^0$. As in subsubsection~\ref{sss.construction-alpha}, we denote by $R_{2j-1,s}^0$ and $R_{2j,s}^0$ the connected components of $D_{2j-1}^s \cap (\chi_s^0\circ\Theta)(D_{2j}^s)$ and $D_{2j}^s \cap (\chi_s^0\circ\Theta)(D_{2j-1}^s)$ which are equal to $\widebar R_{2j-1}$ and $\widebar R_{2j}$ for $s=0$ and which depends continuously on $s$. Recall that $R_{2j-1}^0:=R_{2j-1,1}^0$ and  $R_{2j}^0:=R_{2j,1}^0$. 

The main obstacle to prove Proposition \ref{p.relation-alpha-beta} is to find a good way for connecting $W_{loc}^s(\alpha_j^0)$ with $W_{loc}^s(\widebar\alpha_j^0)$. Indeed, the flow $Y_{t,s}^0$ needs not preserve any foliation, hence there is no way to define a ribbon playing the role of a local stable stable manifold for a periodic orbit of $Y_{t,s}^0$. To overcome this difficulty, we will introduce the concept of \emph{pseudo-periodic orbit}. 

For every $j\in \{1, \dots, 2n\}$, define a \emph{pseudo-periodic orbit} $(\alpha_j^0)'$ of $(W^0, Y_t^0)$ associated to the periodic orbit $\alpha_j^0$ as follows. We denote by $\cF^{s,0}$ and $\cF^{u,0}$ the invariant foliations of the Anosov flow $Y_t^0$, and  by $f^{s,0}_i$ and $f^{u,0}_i$ the one-dimensional foliations induced by $\cF^{s,0}$ and $\cF^{u,0}$ on the torus $T_{i}^0$. We denote the four corners of the rectangle $R_{2j}^0$ by $a_{2j}$, $b_{2j}$, $c_{2j}$, $d_{2j}$ where $a_{2j}$ and $b_{2j}$ are on the same leaf of the stable $f_{2j}^{s,0}$, as well as $c_{2j}$ and $d_{2j}$. We call $\ell_{2j}$ the leaf of $f^{s,0}_{2j}$ containing $a_{2j}$ and $b_{2j}$, and $I_{2j}$ the interval in $\ell_{2j}$ which is bounded by $a_{2j}$ and $b_{2j}$. Then it is easy to observe that $\Theta^{-1} \circ (\varphi^0)^{-1}(I_{2j})$ is a leaf of the foliation $f^{s,0}_{2j-1}$ that crosses the rectangle $R_{2j-1}^0$. We denote by $I_{2j-1}$ the intersection of $\Theta^{-1} \circ (\varphi^0)^{-1}(I_{2j})$ with $R_{2j-1}^0$. Then $\ell_{2j}':= \Theta^{-1}\circ (\varphi^0)^{-1} (I_{2j-1})$ is a leaf of $f_{2j}^{s,0}$ that crosses $R_{2j}^0$. We take a point $q'\in \ell_{2j}'$ and define an orbit arc $(q'q)_v$ of $Y_t^m$ that starts at $q'$ in $\ell_{2j}'$ and ends in $q=\varphi^0 \circ \Theta \circ \varphi^0\circ \Theta (q')$ in $\ell_{2j}$.

Take a small arc $(q_+'q_-')$ in $\ell_{2j}'$ which is a neighborhood of $q'$ in $\ell_{2j}'$. Define $q_{+/-}:= \Theta (q_{+/-}')\in \ell_{2j}$. The arc $(q_+ q_-)$  is a neighborhood of $q$ in $\ell_{2j}$.  Define $R_v$ to be the rectangle that is the union of the orbit arcs of $Y_t^0$ starting in $(q_+' q_-')$ and ending in $(q_+ q_-)$. Define $D_{\varphi^0}^s$ the sub-strip of $D_{2j}^s$ which is bounded by $\ell_{2j}$ and $\ell_{2j}'$. Take an arc $(q'q)_h$ in $D_{\varphi^0}^s$ and a rectangle neighborhood $R_h$ of  $(q'q)_h$ such that $R_h \subset D_{\varphi^0}^s$ and $(q_+ q_-)$ and $(q_+' q_-')$ are two edges of $R_h$. Let $(\alpha_j^0)' :=(q'q)_v \cup (q'q)_h$, and  $W_{loc}^s((\alpha_j^0)'):=R_v \cup R_h$, see Figure \ref{f.9}. We say that  $(\alpha_j^0)'$ a \emph{pseudo-periodic orbit} of $Y_t^0$ associated to $\alpha_j^0$. Notice that $W_{loc}^s((\alpha_j^0)')$ is a closed ribbon (\emph{i.e.} a Möbius band or an annulus) and  the closed curve $(\alpha_j^0)'$ is a core of this closed ribbon. Following the definition, due to continuity, we know that up to isotopy, the ribbon link $\bigcup_{j=1}^{2n} W_{loc}^s((\alpha_j^0)')$ does not depend on the choices of $q'$, $q_+'$, $q_-'$, $R_h$ and $(q'q)_h$.  This means that  the ribbon link  $\bigcup_{j=1}^{2n} W_{loc}^s((\alpha_j^0)')$ is well-defined up to isotopy. 

\begin{figure}[htp]
\begin{center}
  \includegraphics[totalheight=5.3cm]{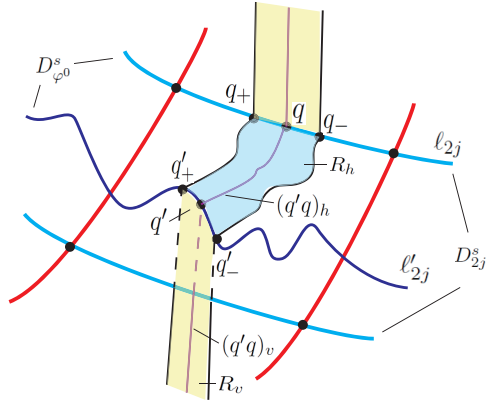}\\
  \caption{The construction of $W_{loc}^s((\alpha_j^m)')$}\label{f.9}
\end{center}
\end{figure}

The definition of pseudo-periodic orbit is in a combinatorial way, and it only depends on hyperbolic plug, intersecting rectangles of stable strips and  transverse glued unstable strips. Hence we can generalize this definition to several cases:
\begin{enumerate}
\item We can define a pseudo-periodic orbit $(\widebar\alpha_j^0)'$ and a closed ribbon $W_{loc}^s((\widebar\alpha_j^0)')$ associated to the periodic orbit $\widebar\alpha_j^0$ of the flow $\widebar Y_t^0$, through replacing $\varphi_0$, $Y_t^0$, $R_{2j-1}^0$, $R_{2j}^0$, $Y_t^0$ by ${\widebar\varphi}^0$, $\widebar Y_t^0$, $\widebar R_{2j-1}$, $\widebar R_{2j}$ respectively.
\item Even more, we can define a pseudo-periodic orbit $(\alpha_{j,s}^0)'$  and  a closed ribbon $W_{loc}^s((\alpha_{j,s}^0)')$ associated to the rectangles $R_{2j-1,s}^0$  and $R_{2j,s}^0$ by replacing $\varphi^0$, $Y_t^0$, $R_{2j-1}^0$, $R_{2j}^0$ by $\chi_s$, $Y_{t,s}^0$, $R_{2j-1,s}^0$, $R_{2j,s}^0$ respectively.
\end{enumerate}
The corresponding ribbons links are well-defined up to isotopy.

Now, we can identify the manifolds $W^0=U/\varphi^0$ with $W_s^0=U/\chi_s^0$ as follows. Decompose $U$  into $U=U_0 \cup V$ where $V$ is the closure of $U-U_0$ such that:
\begin{enumerate}
\item $V$ is the union of $4n$ connected components $V_1, \dots, V_{4n}$;
\item $V_i$ ($i=1,\dots, 4n$) is parameterized as $\TT^2 \times [0,1]$ such that $\TT^2 \times \{0\} = T_i^{in}$ and the  parameterized $I$-bundle  corresponds to the flowlines of $X_t \mid_{V_i}$ that start at $\TT^2 \times \{0\}$;
\item $(U, X_t)$ and $(U_0, X_t )\mid_{U_0}$ are orbitally equivalent.
\end{enumerate}
Then $Y_{t,s}^0$ can be regarded as a flow on the manifold $W^m$ as follows:
\begin{enumerate}
\item $Y_{t,s}^0\mid_{U_0}= Y_t^m \mid_{U_0}=X_t \mid_{U_0}$,
\item on every $V_i$, the flowline starting at $x\in \TT^2 \times \{0\}$ is the arc $\bigcup_{0\leq u\leq 1} (\chi_{su}\circ {\widebar\varphi}_0^{-1} (x), u)$.
\end{enumerate}
From now on we can think that all flows $Y_{t,s}^0$, with $s\in[0,1]$, are on the same manifold $W^0$. Since $\chi_s$, $R_{2j-1,s}$ and $R_{2j,s}$ depend continuously on $s$, we get that for every pair of real numbers $s_1, s_2 \in [0,1]$, the ribbon link $\bigcup_{j=1}^{2n} W_{loc}^s((\alpha_{j,s_1}^0)')$ is isotopic to the ribbon link $\bigcup_{j=1}^{2n} W_{loc}^s((\alpha_{j,s_2}^0)')$ in $W^0$. In particular, by taking $s_1=0$ and $s_2=1$ and noticing that $W_{loc}^s((\alpha_{j,0}^0)')=W_{loc}^s((\widebar\alpha_j^0)')$ and $W_{loc}^s((\alpha_{j,1}^0)')=W_{loc}^s((\alpha_j^0)')$, we get:

\begin{lemma}\label{l.bisoinv}
There exists a homeomorphism $g: W^0 \to \widebar W^0$ such that $g(W_{loc}^s(\alpha_j^0)')= W_{loc}^s((\widebar\alpha_j^0)')$ for every $j\in \{1,\dots, 2n\}$. 
\end{lemma} 

Moreover, up to isotopy, the ribbon link $\bigcup_{j=1}^{2n} W_{loc}^s((\alpha_j^0)')$ and the ribbon link $\bigcup_{j=1}^n W_{loc}^s(\alpha_j^0)$ are the same. More precisely, we have:

\begin{lemma}\label{l.cequal}
There exists a self-homeomorphism $h$ of the manifold $W^0$, isotopic to identity, such that $h((\alpha_j^0)')=\alpha_j^0$ and $h(W_{loc}^s((\alpha_j^0)'))= W_{loc}^s(\alpha_j^0)$ for $j=1,\dots,2n$.
\end{lemma}

\begin{proof}
Recall that the periodic orbit $\alpha_j^0$ intersects the torus $T_{2j}^0$ at a single point $q_{2j}^0$. For every point $r\in T_{2j}^0$, denote by $\ell^s(r)$ and $\ell^u(r)$ the leaves of the foliation $f^{s,0}_{2j}$ and $f^{u,0}_{2j}$ containing the point $r$. Take an arc $(q_{2j}^{0,+} q_{2j}^{0,-})$ which is a neighborhood of the point $q_{2j}^0$ in the leaf $\ell^s(q_{2j}^0)$. Let $q_+':= \ell_{2j}' \cap \ell^u (q_{2j}^{0,+})$ and $q_-':= \ell_{2j}' \cap \ell^u (q_{2j}^{0,-})$. Let $a:= \ell_{2j}' \cap \ell^u (q_+)$ and $b:=\ell_{2j}' \cap \ell^u (q_-)$. Let $abq_- q_+$ be the rectangle with vertices  $a$, $b$, $q_-$, $q_+$, and edges included alternatively arcs in leaves of $f^{s,0}_{2j}$ and $f^{u,0}_{2j}$. One can easily check that $S_1 := R_v \cup abq_- q_+$ is isotopic to some $W_{loc}^s((\alpha_j^0)')$. For every point $x\in q_+' q_-' q_{2j}^{0,+} q_{2j}^{0,-}$, there exists a positive number $t(x)$ such that $Y_{t(x)}^0 (x)\in T_{2j}^{in}$ corresponds to the first return map of $Y_t^m$ from the rectangle $q_+' q_-' p_{2j}^+ p_{2j}^-$.  See Figure \ref{f.10}. Define 
$$V:= \bigcup_{x\in q_+' q_-' q_{2j}^{0,+} q_{2j}^{0,-}, 0\leq t \leq t(x)} Y_t^0 (x)\quad\mbox{and}\quad S_2 = \bigcup_{x\in q_{2j}^{0,+} q_{2j}^{0,-}, 0\leq t \leq t(x)} Y_t^0 (x).$$ 
One can easily check that $V$ is homeomorphic to a solid torus. It is obvious that $S_2$ is isotopic to $W_{loc}^s(\alpha_j^0)$, and $S_1$ and $S_2$ are two parallel annuli in the torus $\partial V$. Hence $S_1$ and $S_2$ are isotopic, and therefore $W_{loc}^s(\alpha_j^0)$ and $W_{loc}^s((\alpha_j^0)')$ are isotopic. Further notice that the isotopies for every two different choices of $j$ have disjoint supports. Hence there exists a self-homeomorphism $h$ on $W^0$ that satisfies the conclusion of the lemma.
\end{proof}

\begin{figure}[htp]
\begin{center}
  \includegraphics[totalheight=6.7cm]{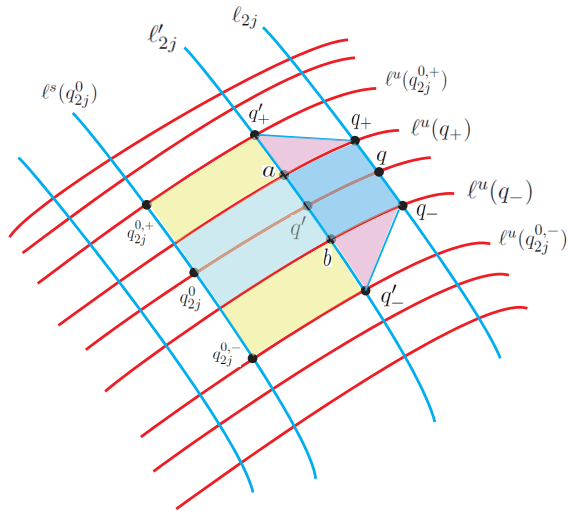}\\
  \caption{Some notations during the proof of Lemma \ref{l.cequal}}\label{f.10}
\end{center}
\end{figure}

The same arguments yield an analog of Lemma~\ref{l.cequal} for the periodic orbits $\widebar\alpha_{1}^0,\dots,\widebar\alpha_{2n}^0$:

\begin{lemma}\label{l.cequal-beta}
There exists a self-homeomorphism $\widebar h$ of $\widebar W^0$, isotopic to identity, such that $\widebar h((\widebar\alpha_j^0)')=\widebar\alpha_j^0$ and $\widebar h(W_{loc}^s((\widebar\alpha_j^0)'))= W_{loc}^s(\widebar\alpha_j^0)$ for $j=1,\dots,2n$.
\end{lemma}

\begin{proof}[Proof of Proposition \ref{p.relation-alpha-beta}]
The case $m=0$ of Proposition~\ref{p.relation-alpha-beta} follows immediately from Lemma~\ref{l.bisoinv}, Lemma~\ref{l.cequal} ane Lemma~\ref{l.cequal-beta}. Moreover, all the constructions and the arguments above can be modified to cover the case $m\geq 1$: it suffices to replace the rectangles $R_{2j-1},R_{2j}$ by the rectangles $L_{2j-1},L_{2j}$ for $j\leq m$. 
\end{proof}

\subsection{The ribbon link $W^s_{loc}(\alpha_1^m)\cup\dots\cup W^s_{loc}(\alpha_{2n}^m)$}
\label{ss.homeomorphic}

We will now prove item (4) of Theorem~\ref{t.con}. In view of Proposition~\ref{p.relation-alpha-beta}, it is equivalent to prove:

\begin{proposition}\label{p.h1}
For $m=1,\dots,2n$, there exists a homeomorphsim $\widebar h^m: \widebar W^m \longrightarrow \widebar W^0:$ such that $\widebar h^m(W^s_{loc}(\widebar\alpha_j^m))= W^s_{loc}(\widebar\alpha_j^0)$ for every $j\in \{1,\dots, 2n\}$.
\end{proposition}

Let us briefly recall our notations. For $m=0,\dots,2n$, we have defined an explicit gluing diffeomorphism ${\widebar\varphi}^m:\partial^{out} U\to\partial^{in}U$, considered the flow $\widebar Y_t^m$ induced by $X_t$ on the quotient manifold $\widebar W^m=U/{\widebar\varphi}^m$. Although $\widebar Y_t^m$ is not necessarily an Anosov flow, it preserves two transverse foliations $\widebar\cF^{s,m}$ and $\widebar\cF^{u,m}$ induced by $\cF^s$ and $\cF^u$ on $U$. See Proposition \ref{p.construction-foliations} for the definitions of 
$\cF^s$ and $\cF^u$.
 We have constructed $2n$ some periodic orbits $\widebar\alpha_1^m,\dots,\widebar\alpha_{2n}^m$ of the flow. We denote by $W^s_{loc}(\widebar\alpha_j^m)$ a tubular neighborhood of the periodic orbit $\widebar\alpha_j^m$ in the leaf of $\widebar\cF^{s,m}$ containing this orbit. We recall that this is an abuse of notation, since $W^s_{loc}(\widebar\alpha_j^m)$ is not necessarily a local stable set of $\widebar\alpha_j^m$. 

The proof of Proposition~\ref{p.h1} is quite long, and will be divided into several steps. We fix an integer $m=1,\dots,2n$ for the remainder of the section.

\subsubsection{Preliminary: tunnels and ribbons in good positions in a tunnel.}  Let $D^s$ be a stable strip of $(U,X)$ \emph{i.e.} a connected  component of $\partial^{in}U\setminus\cL^s$. Since $D^s$ is disjoint from $W^s(\Lambda)$, the forward orbit of any point $x\in D^s$ will eventually exit $U$ at some point $\Theta(x)=X^{T(x)}(x)\in\partial^{out}U$. Note that $D^u=\Theta(D^s)$ is an unstable strip, \emph{i.e.} a connected component of $\partial^{out} U\setminus\cL^u$. Let 
$$E=\{X^t(x), x\in D^s, 0\leq t\leq T(x)\}$$
be the union of the orbit segments going from $D^s$ to $D^u$ in $U$. We call $E$ the \emph{tunnel} joining $D^s$ to $D^u$. One can easily verify that such a tunnel is homeomorphic to $\RR^2\times [0,1]$ through a homeomorphism mapping $D^s$ to $\RR^2\times\{0\}$, mapping $D^u$ to $\RR^2\times\{1\}$, mapping the orbits of the flow $X_t$ in $E$ to the vertical  segments $\{(\cdot,\cdot)\}\times [0,1]$, mapping the leaves of the stable foliation $\cF^s\cap E$ to the vertical bands $\{\cdot\}\times\RR\times [0,1]$ and mapping the leaves of the unstable foliation $\cF^s\cap E$ to the vertical bands $\RR\times\{\cdot\}\times [0,1]$.

Consider a \emph{ribbon} in $E$, \emph{i.e.} the image an embedding $R:[-1,1]\times [0,1]\hookrightarrow E\simeq \RR^2\times [0,1]$. We call $R([-1,1]\times\{0\})$ and $R([-1,1]\times\{1\})$ the \emph{ends} of the ribbon $R$. We say that this ribbon $R$ is \emph{in good position} (with respect to the foliation $\cF^s$) if there exists a function $r_s:[0,1]\to\RR$ such that 
$$\forall t\in[0,1], \quad R([-1,1]\times\{t\})\subset \{r_s(t)\}\times\RR\times\{t\}.$$
Observe that this definition implies that for every $t\in[0,1]$, the arc $R([-1,1]\times\{t\})$ must be included in a leaf of the foliation $\cF^s\cap E$. This also implies that one of the two ends of $R$ must be included in $D^s$ and the other must be included $D^u$ respectively. The following fact is a immediate consequence of the convexity of the set of ribbons which are in good position in $E$ for the natural affine structure on $E\simeq \RR^2\times [0,1]$:

\begin{fact}
\label{f.ribbons-isotopic}
Let $R,R'$ be two ribbons in good positions in a tunnel $E$, with the same ends. Then $R$ and $R'$ are isotopic, with fixed ends, among ribbons in good position.
\end{fact}

\subsubsection{Step 1. Construction of an explicit diffeomorphism $\widebar h_\epsilon^m:\widebar W^m\to \widebar W^0$.} Recall that the gluing maps ${\widebar\varphi}^0,{\widebar\varphi}^m$ were defined by the formula
$${\widebar\varphi}^0_{|T_i^{out}}=\tau_{+1/2}\circ\sigma\quad\mbox{and}\quad{\widebar\varphi}^m_{|T_i^{out}}=\left\{\begin{array}{ll}\tau_{-1/2}\circ\sigma & \mbox{if }i\leq 2m\\\tau_{+1/2}\circ\sigma & \mbox{if }i>2m\end{array}\right.$$
where $\tau_{v}$ is the translation $(x,y)\mapsto (x+v,y)$ in the coordinate system $\xi_i:T_i^{in}\to \TT_i=(\RR/(2i+2)\ZZ)\times(\RR/\ZZ)$. As a consequence, 
$${\widebar\varphi}^m_{|T_i^{out}}=\left\{\begin{array}{ll}\tau_{-1}\circ{\widebar\varphi}^0_{|T_i^{out}} & \mbox{if }i\leq 2m\\ {\widebar\varphi}^0_{|T_i^{out}} & \mbox{if }i>2m\end{array}\right.$$
One may add a coordinate $t$ to the coordinate system $(x,y)$ on $T_i^{in}$ in order to obtain a coordinate system $(x,y,t)\in (\RR/(2i+2)\ZZ)\times(\RR/\ZZ)\times [0,1]$ on a collar neighborhood $V_i$ of the torus $T_i^{in}$ in the $3$-manifold $U$. One can choose the coordinate $t$ so that the torus $T_i^{in}$ corresponds to $\{t=0\}$, and the vector field $X$ corresponds to $\frac{\partial}{\partial t}$. Given $\epsilon\in ]0,1]$, we denote by $V_{i,\epsilon}$ the part of $V_i$ corresponding to $t\in [0,\epsilon]$ in the $(x,y,t)$ coordinate system. Then we define a homeomorphism $h_\epsilon:U\to U$ by setting:
$$h_{\epsilon}(x,y,t)=\left(x+1-\frac{t}{\epsilon},y,t\right)\mbox{ on }V_{i,\epsilon}\mbox{ for }i=1,\dots,2m\mbox{ and }
h_\epsilon=\mathrm{Id}\mbox{ on }U\setminus \bigcup_{i=1}^{2m} V_{i,\epsilon}.$$
Notice that $h_{\epsilon}=\tau_{+1}$ on $T_i^{in}$ for $i\leq 2m$, $h_{\epsilon}=\mathrm{Id}$ on $T_i^{in}$  for $i>2m$, and  $h_{\epsilon}=\mathrm{Id}$ on $T_i^{out}$ for every $i$. It follows that $h_{\epsilon}\circ {\widebar\varphi}^m={\widebar\varphi}^0\circ h_\epsilon$. As a consequence, $h_\epsilon$ induces a diffeomorphism 
$$\widebar h_\epsilon:\widebar W^m=U/{{\widebar\varphi}^m}\longrightarrow \widebar W^0=U/\widebar\varphi^0.$$  

\medskip

Using the diffeomorphism $\widebar h_\epsilon$ (for a given $\epsilon$), we can map the ribbons $W^s_{loc}(\widebar\alpha_1^m),\dots, W^s_{loc}(\widebar\alpha_{2n}^m)$ in $\widebar W^0$. Then we  are left to prove that the ribbon links $W^s_{loc}(\widebar\alpha_1^0) \cup \dots \cup W^s_{loc}(\widebar\alpha_{2n}^0)$ and $\widebar h_\epsilon(W^s_{loc}(\widebar\alpha_1^m))\cup \dots \cup \widebar h_\epsilon(W^s_{loc}(\widebar\alpha_{2n}^m))$ are isotopic in $\widebar W^0$. In practice, we will describe an isotopy between $\widebar h_\epsilon(W^s_{loc}(\widebar\alpha_j^m))$ and $W^s_{loc}(\widebar\alpha_{j}^0)$ for every $j$, so that the supports of the isotopies corresponding to the different integers $j$ are pairwise disjoint. One may already observe that, the ribbons $\widebar h_\epsilon(W^s_{loc}(\widebar\alpha_j^m))$ and $W^s_{loc}(\widebar\alpha_{j}^0)$ coincide for $j>m$. This follows from two facts: first, we know that $(\widebar \pi^0)^{-1}(\widebar\alpha_j^0)=(\widebar \pi^m)^{-1}(\widebar\alpha_j^m)$ for $j>m$ (see Remark \ref{r.beta-indep-m}), and second, we observe that $(\widebar\pi^0)^{-1}(\widebar\alpha_j^0)=(\widebar \pi^m)^{-1}(\widebar\alpha_j^m)$ is disjoint from the support of the  diffeomorphism $h_\epsilon$ for $j\neq m$. As a consequence, we will focus our attention on the case 
$j\leq m$.
 
\medskip

Until the very last step of the proof, we fix an integer $j\in\{1,\dots,2m\}$. For sake of simplicity, we will denote 
$$A^0:=W^s_{loc}(\widebar\alpha_{j}^0)\mbox{ and }A_\epsilon^m:=\widebar h_\epsilon(W^s_{loc}(\widebar\alpha_j^m)).$$  
The tori $\widebar T_{2j-1}^0=\widebar \pi^0(T_{2j-1}^{in})=\widebar \pi^0(T_{2j-1}^{out})$ and $\widebar T_{2j}^0=\widebar \pi^0(T_{2j}^{in})=\widebar \pi^0(T_{2j}^{out})$  will play a central role in the remainder of the proof. For $i=2j-1,2j$, we consider the one-dimensional foliation $\widebar f^{s,0}_i=\widebar\pi^0(f^{s,in}_i)=\widebar\pi^0(f^{s,out}_i)$ and $\widebar f^{u,0}_i=\widebar\pi^0(f^{u,in}_i)=\widebar\pi^0(f^{u,out}_i)$ on the torus $\widebar T_i^0$. Recall that we have constructed some stable and unstable strips
$$D_{2j-1}^s\subset A_{2j-1}^{1,s},\quad D_{2j-1}^u\subset A_{2j-1}^{1,u},\quad 
D_{2j}^s\subset A_{2j}^{1,s},\quad D_{2j}^u\subset A_{2j}^{1,u}$$
so that 
$$D_{2j-1}^u=\sigma(D_{2j-1}^s),\quad D_{2j}^u=\sigma(D_{2j}^s),\quad D_{2j-1}^u=\Theta(D_{2j}^s),\quad  D_{2j}^u=\Theta(D_{2j-1}^s).$$ 
See Proposition~\ref{p.selection-strips}. We denote by $E_j^+$ the tunnel joining the stable strip $D^s_{2j}$ to the unstable strip $D^u_{2j-1}$ in $U^+$, and by $E_j^-$ the tunnel joining the stable strip $D^s_{2j-1}$ to the unstable strip $D^u_{2j}$ in $U^-$. 

Using the definition of the basis $(\vec{e_1},\vec{e_2}, \vec{e_3})$ in Proposition~\ref{p.orient} and the definition of the translation $\tau_{-1}$ in section~\ref{sss.construction-psi}, one can check that the orientations of the Reeb annuli, the cyclic order of these annuli, and the orientation of the translation $\check\tau_{-1}:=\widebar\pi^0\circ\tau_{-1}\circ(\widebar\pi^0_{|\partial^{in}U})^{-1}$ on the tori $\widebar T_{2j-1}^0$ and $\widebar T_{2j}^0$ are just as depicted on Figure~\ref{f.strips} below.

\begin{figure}[ht]
\centerline{\includegraphics[width=12cm]{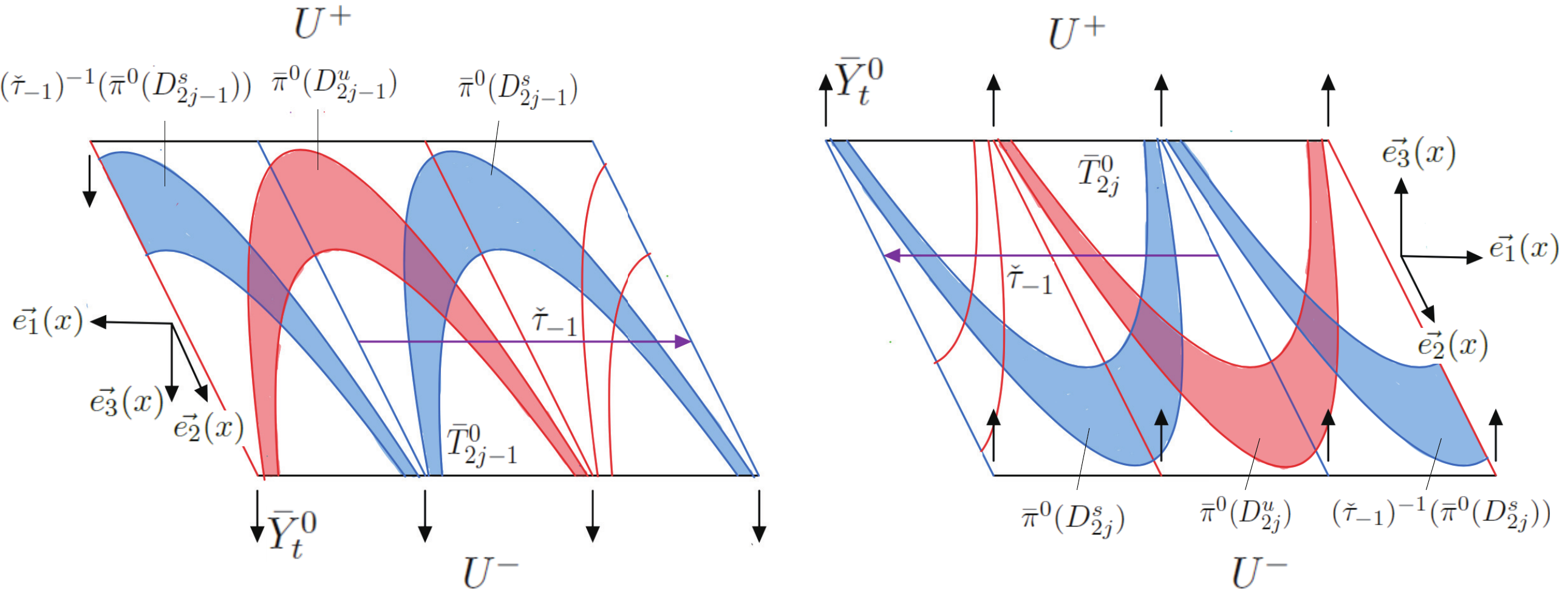}}
\caption{\label{f.strips}The tori $\widebar T_{2j-1}^0$ and $\widebar T_{2j}^0$, the strips $\widebar \pi^0(D_{2j-1}^s)$, $\widebar \pi^0(D_{2j-1}^u)$, $\widebar \pi^0(D_{2j}^s)$ and $\widebar \pi^0(D_{2j}^u)$, the direction of the vector field $\widebar Y^0$ and the direction of the translation $\check\tau_{-1}$.}
\end{figure}

\subsubsection{Step 2. Decomposition of the ribbon $A^0=W^s_{loc}(\widebar\alpha_{j}^0)$.} We decompose $A^0$ in two halves 
$$A^0=A^0_+\cup A^0_- \mbox{ where } A^0_- :=A^0\cap \widebar\pi^0(U^-)\mbox{ and }A^0_+:=A^0\cap \widebar\pi^0(U^+).$$
We will now describe more precisely $A^0_-$ and $A^0_+$. For that purpose, recall that, the periodic orbit $\widebar\alpha_j^0$:
\begin{itemize}
\item[--] intersects the torus $\widebar T_{2j-1}^0$ at a single point $\widebar q_{2j-1}^0\in\widebar \pi^0(D_{2j-1}^s)\cap \widebar \pi^00(D_{2j-1}^u)$,
\item[--] intersects the torus $\widebar T_{2j}^0$ at a single point $\widebar q_{2j}^0\in\widebar \pi^0(D_{2j}^s)\cap \widebar \pi^0(D_{2j}^u)$,
\item[--] does not intersect the torus $\widebar T_{i}^0$ for $i\neq 2j-1,2j$. See Proposition~\ref{p.construction-beta}.
\end{itemize}
Let $\ell^s(\widebar q_{2j-1}^0)$  be the leaf of the foliation $\widebar f^{s,0}_{2j-1}$ containing the point  $\widebar q_{2j-1}^0$ and $\ell^s(\widebar q_{2j}^0)$  be the leaf of the foliation $\widebar f^{s,0}_{2j}$ containing the point $\widebar q_{2j}^0$. Let
 $$I^s(\widebar q_{2j-1}^0)=A^0\cap \widebar T_{2j-1}^0 \quad\mbox{and}\quad I^s(\widebar q_{2j}^0)=A^0\cap \widebar T_{2j}^0.$$ 
 Clearly, $I^s(\widebar q_{2j-1}^0)$ is a neighborhood of $\widebar q_{2j-1}^0$ in $\ell^s(\widebar q_{2j-1}^0)$, and $I^s(\widebar q_{2j}^0)$ is a neighborhood of $\widebar q_{2j}^0$ in  $\ell^s(\widebar q_{2j}^0)$. Adjusting the width of $A$ if necessary, we may assume that $I^s(\widebar q_{2j-1}^0)$ is a crossbar\footnote{By such we mean that the arc $I^s(\widebar q_{2j-1}^0)$ is included in the unstable strip $\widebar \pi^0(D^u_{2j-1})$ and that the ends of $I^s(\widebar q_{2j-1}^0)$ belong to the two boundary leaves of $\widebar \pi^0(D^u_{2j-1})$.} of the unstable strip $\widebar \pi^0(D^u_{2j-1})$, and $I^s(\widebar q_{2j}^0)$ is a crossbar of the unstable strip $\widebar \pi^0(D^u_{2j})$. By definition, $A^0_-$ is a ribbon joining $I^s(\widebar q_{2j-1}^0)$ to $I^s(\widebar q_{2j}^0)$ in $U^-$ and $A^0_+$ is a ribbon joining $I^s(\widebar q_{2j}^0)$ to $I^s(\widebar q_{2j-1}^0)$ in $U^+$.  See Figure~\ref{f.decomposition-A}. Since $\widebar q_{2j-1}^0\in \widebar \pi^0(D_{2j-1}^s)$ and $\widebar q_{2j}^0\in \widebar \pi^0(D_{2j}^s)$, $A^0_+$ is contained in the tunnel $E^+_j$ and $A^0_-$ is contained in the tunnel $E_j^-$.
 
 \begin{figure}[ht]
\centerline{\includegraphics[width=13cm]{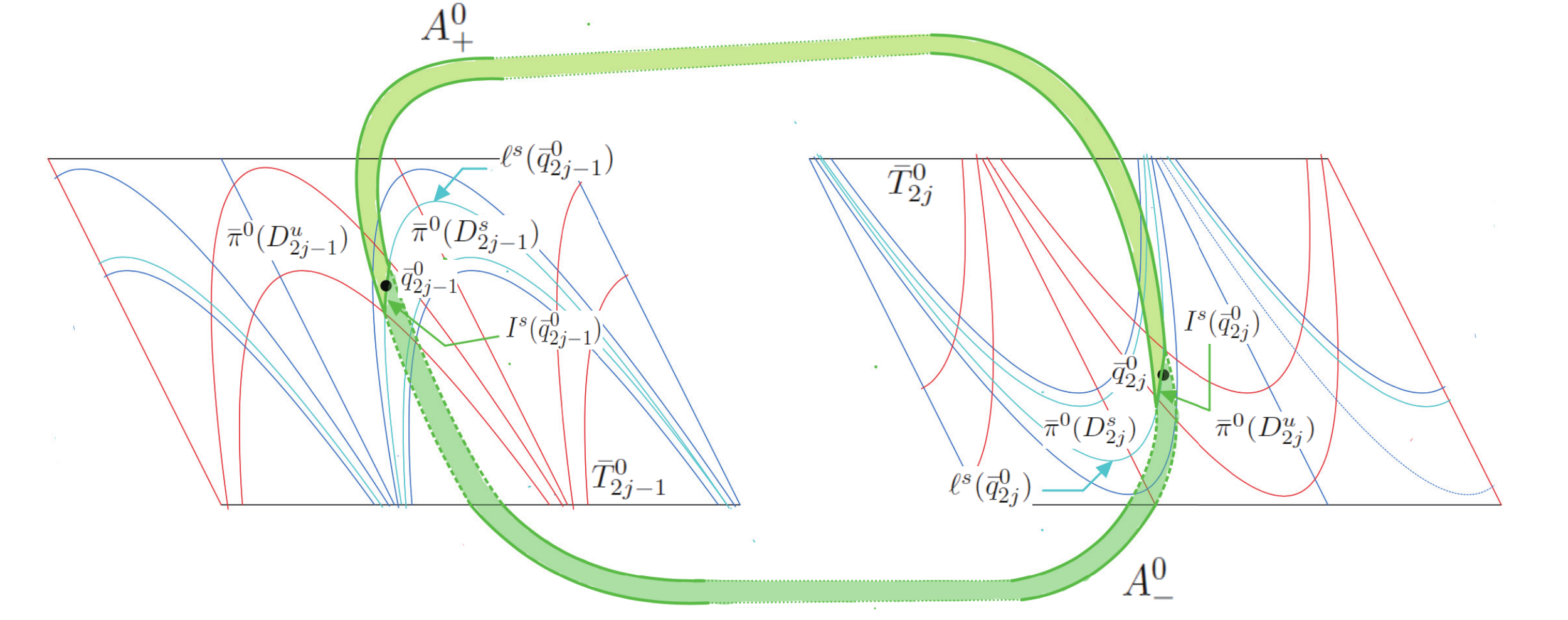}}
\caption{\label{f.decomposition-A}Decomposition of the closed ribbon $A^0=W^s_{loc}(\widebar\alpha_j^0)$.}
\end{figure}

\begin{remark}
\label{r.faithful}
Close to the tori $\widebar T_{2j-1}^0$ and $\widebar T_{2j}^0$, we have at our disposal some local coordinate systems  in which the flow $\widebar Y^0_t$, the foliations $\widebar f^{s,0}_{2j-1},\widebar f^{u,0}_{2j-1},\widebar f^{s,0}_{2j},\widebar f^{u,0}_{2j}$ and the translation $\check\tau_{-1}$ have a fully explicit simple form. As a consequence, our figures above and below represent quite faithfully the geometry of the ribbons $A^0=W^s_{loc}(\widebar\alpha_j^0)$ and $A^m_\epsilon=\widebar h_\epsilon(W^s_{loc}(\widebar\alpha_{2j}^m))$ close to the tori $\widebar T_{2j-1}^0$ and $\widebar T_{2j}^0$. On the contrary, we do not pretend to be able to represent faithfully the geometry and topology of the ribbons far from $\widebar T_{2j-1}^0$ and $\widebar T_{2j}^0$.
\end{remark}

\subsubsection{Step 3. Decomposition of the ribbon $A^m_\epsilon=\widebar h_\epsilon(W^s_{loc}(\widebar\alpha_{2j}^m))$.} We cut $A^m_\epsilon$ into four pieces:
$$A^m_\epsilon=A^m_{2j,\epsilon}\cup A^m_{+,_\epsilon}\cup A^m_{2j-1,\epsilon}\cup A^m_{-,\epsilon}$$ 
where
$$A^m_{2j,\epsilon}:=A^m_\epsilon\cap V_{2j,\epsilon},\quad A^m_{2j-1,\epsilon}:=A^m_\epsilon\cap V_{2j-1,\epsilon}$$
$$A^m_{+,\epsilon}:=A^m_\epsilon\cap (U^+\setminus V_{2j,\epsilon}),\quad A^m_{-,\epsilon}:=A^m_\epsilon\cap (U^-\setminus V_{2j-1,\epsilon}).$$
In order to describe more precisely the mutual position of these four parts in $\widebar W^0$, let us first recall that the periodic orbit $\widebar\alpha_j^m$:
\begin{itemize}
\item[--] intersects the torus $\widebar T_{2j-1}^m$ at a single point $\widebar q_{2j-1}^m\in \widebar \pi^m(D_{2j-1}^s)\cap \widebar \pi^m(D_{2j-1}^u)$,
\item[--] intersects the torus $\widebar T_{2j}^m$ at a single point $\widebar q_{2j}^m\in \widebar \pi^m(D_{2j}^s)\cap \widebar \pi^m(D_{2j}^u)$,
\item[--] does not intersect the torus $\widebar T_i^m$ for $i\neq 2j-1,2j$. 
\end{itemize}
Consider the points 
$$\widebar p_{2j-1}^m=\widebar h_\epsilon(\widebar q_{2j-1}^m)\quad\mbox{and}\quad\widebar p_{2j}^m=\widebar h_\epsilon(\widebar q_{2j}^m).$$ 
Notice that these two points do not depend on $\epsilon$ since the restriction $\widebar h_\epsilon$ to $\widebar\pi^m(\partial U)$ does not depend on this parameter (see Step 1). More precisely, recall that $\widebar h_\epsilon$ is induced by a diffeomorphism $h_\epsilon:U\to U$ such that
$$h_\epsilon=\mathrm{Id}\mbox{ on }T_{2j-1}^{out}\mbox{ and }h_\epsilon=\tau_{+1}\mbox{ on }T_{2j-1}^{in}.$$ 
Since $D_{2j-1}^u\subset T_{2j-1}^{out}$ and $D_{2j-1}^s\subset T_{2j-1}^{in}$, it follows that 
\begin{eqnarray*}
&&\widebar h_\epsilon(\widebar\pi^m(D_{2j-1}^u)) = \widebar\pi^0(h_1(D_{2j-1}^u))=\widebar\pi^0(D_{2j-1}^u)\\
&&\widebar h_\epsilon(\widebar\pi^m(D_{2j-1}^s))=\widebar\pi^0(h_1(D_{2j-1}^s))=\check\tau_{+1}(\widebar\pi^0(D_{2j-1}^s))
\end{eqnarray*}
 hence 
 $$\widebar p_{2j-1}^m\in \check\tau_{+1}(\widebar\pi^0(D_{2j-1}^s))\cap \widebar\pi^0(D_{2j-1}^u).$$
 And by a similar computation 
 $$\widebar p_{2j}^m\in \check\tau_{+1}(\widebar\pi^0(D_{2j}^s))\cap \widebar\pi^0(D_{2j}^u).$$
 We denote by  $\ell^s(\widebar p_{2j-1}^m)$ the  leaf of the foliation $\widebar f^{s,0}_{2j-1}$ containing $\widebar p_{2j-1}^m$, and  by  $\ell^s(\widebar p_{2j}^m)$ the  leaf of the foliation $\widebar f^{s,0}_{2j}$ containing $\widebar p_{2j}^m$. Let
 $$I^s(\widebar p_{2j-1}^m)=A^m_\epsilon\cap \widebar T_{2j-1}^0\mbox{ and }I^s(\widebar p_{2j}^m)=A^m_\epsilon\cap \widebar T_{2j}^0.$$ 
We claim that $I^s(\widebar p_{2j-1}^m)$ is contained in the leaf $\ell^s(\widebar p_{2j-1}^m)$. Indeed, since $\widebar h_\epsilon:U\to U$ is the identity on $U\setminus\cup_i V_{i,\epsilon}$,  the foliation $\widebar h_\epsilon(\widebar\cF^{s,m})$ coincides with the foliation $\widebar\cF^{s,0}$ outside the $\widebar\pi^0(\cup_i V_{i,\epsilon})$. In particular, $A^m_\epsilon=\widebar h_\epsilon(W^s_{loc}(\widebar\alpha_{j}^m))$ is contained in a leaf of the foliation $\widebar \cF^{s,0}$ outside $\widebar\pi^0(\cup_i V_i)$. It follows that $I^s(\widebar p_{2j-1}^m)$ is contained in a leaf of the foliation $\widebar f^{s,0}_{2j-1}$. This means that $I^s(\widebar p_{2j-1}^m)$ is a neighborhood of the point $\widebar p_{2j-1}^m$ in the leaf  $\ell^s(\widebar p_{2j-1}^m)$, as claimed. Then, by adjusting the width of the ribbon $A^m$ if necessary, we may assume that $I^s(\widebar p_{2j-1}^m)$ is a crossbar of the unstable strip $\widebar\pi^0(D_{2j-1}^u)$. By similar arguments,  $I^s(\widebar p_{2j}^m)$ is a neighborhood of the point $\widebar p_{2j}^m$ in the leaf  $\ell^s(\widebar p_{2j}^m)$ and we may assume that $I^s(\widebar p_{2j}^m)$ is a crossbar of the unstable strip $\widebar\pi^0(D_{2j}^u)$. Now we can describe the different parts of $A^m_\epsilon$ (see Figure \ref{f.decomposition-B}):
\begin{itemize}
\item[--] Due to the explicit form of the diffeomorphism $\widebar h_\epsilon$ (see Step 1), $A^m_{2j,\epsilon}$ is a straight ribbon joining $I^s(\widebar p_{2j}^m)$ to $(\widebar Y^0_\epsilon\circ \check\tau_{-1})(I^s(\widebar p_{2j}^m))$ in $\widebar\pi^0(V_{2j,\epsilon})$.
\item[--] The ribbon $A^m_{+,\epsilon}$ is joining $(\widebar Y_\epsilon^0\circ \check\tau_{-1})(I^s(\widebar p_{2j}^m))$ to $I^s(\widebar p_{2j-1}^m)$ inside the tunnel $E_j^+$. This ribbon is a neighborhood of the orbit segment joining $(\widebar Y_\epsilon^0\circ \check\tau_{-1})(\widebar p_{2j}^m)$ to  $\widebar p_{2j-1}^m$ in a leaf of  $\widebar\cF^{s,0}$.
\item[--] Then $A^m_{2j-1,\epsilon}$ is a straight ribbon joining $I^s(\widebar p_{2j-1}^m)$ to $(\widebar Y_\epsilon^0\circ \check\tau_{-1})(I^s(\widebar p_{2j-1}^m))$ in $V_{2j-1,\epsilon}$.
\item[--] Finally $A^m_{-,\epsilon}$ is joining $(\widebar Y_\epsilon^0\circ \check\tau_{-1})(I^s(\widebar p_{2j-1}^m))$ to $I^s(\widebar p_{2j}^m)$ inside the tunnel $E^-_j$. This ribbon is a neighborhood of the orbit segment joining $(\widebar Y_\epsilon^0\circ \check\tau_{-1})(\widebar p_{2j-1}^m)$ to $\widebar p_{2j}^m$ in a leaf of  $\widebar \cF^{s,0}$.  
\end{itemize}

\begin{figure}[ht]
\centerline{\includegraphics[width=13.5cm]{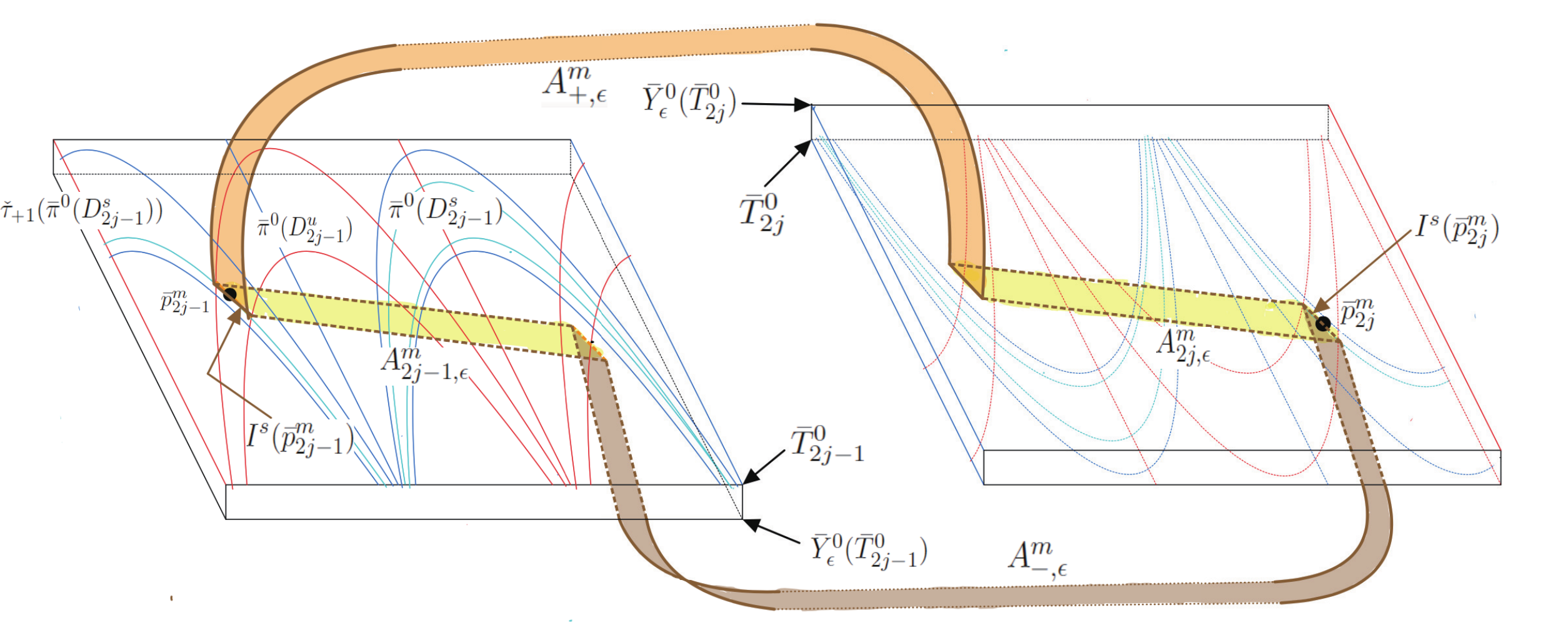}}
\caption{\label{f.decomposition-B}Decomposition of the closed ribbon $A^m_\epsilon=\widebar h_\epsilon(W^s_{loc}(\widebar\alpha_j^m))$.}
\end{figure}

\begin{remark}
\label{r.faithful-2}
In Steps 4, 5 and 6, we will describe some isotopies of the closed ribbons $A$ and $B$. These isotopies will be supported in small neighborhoods of the tori $\widebar T_{2j-1}^0$ and $\widebar T_{2j}^0$.  As a consequence, our figures below represent quite faithfully the isotopies we will consider (see Remark~\ref{r.faithful}). In principle, the isotopies could described by explicit formulas in the $(x,y,t)$ coordinate system used in Step 1, but we thought this would hardly shed any light on the proof. 
\end{remark}

\subsubsection{Step 4. Isotopy from $A^m_1=\widebar h_1(W^s_{loc}(\widebar\alpha_j^m))$ to $A^m=\lim_{\epsilon\to 0} \widebar h_\epsilon(W^s_{loc}(\widebar\alpha_j^m))$.}
The diffeomorphism $\widebar h_\epsilon$ has been defined for every $\epsilon\in (0,1]$. It does not converge towards a diffeomorphism as $\epsilon\to 0$. Nevertheless, using the explicit definition of $\widebar h_\epsilon$, one may easily observe that the closed ribbon $A^m_\epsilon=\widebar h_\epsilon(W^s_{loc}(\widebar\alpha_j^m))$ converges, as $\epsilon$ goes to $0$, to a closed ribbon $A^m$ which is isotopic to $A^m_1$. Moreover, the decomposition $A^m_\epsilon=A^m_{2j,\epsilon}\cup A^m_{+,\epsilon}\cup A^m_{2j-1,\epsilon}\cup A^m_{-,\epsilon}$ converges to a decomposition $A^m=A^m_{2j}\cup A^m_+\cup A^m_{2j-1}\cup A^m_-$ of the ribbon $A^m$. More precisely:
\begin{itemize}
\item[--] $A^m_{2j}$ is a straight ribbon joining the arc $I^s(\widebar p_{2j}^m)$ to the arc $\check\tau_{-1}(I^s(\widebar p_{2j}^m))$ in the torus $\widebar T_{2j}^0$;
\item[--] $A^m_+$ is joining the arc $\check\tau_{-1}(I^s(\widebar p_{2j}^m))$ to the arc $I^s_{\widebar p_{2j-1}^m}$ inside the tunnel $E_j^+$; 
it is  a neighborhood of the orbit segment joining $\check\tau_{-1}(\widebar p_{2j}^m)$ to $\widebar p_{2j-1}^m$ in a leaf of $\widebar\cF^{s,0}$; 
\item[--] $A^m_{2j-1}$ is a straight ribbon joining the arc $I^s(\widebar p_{2j-1}^m)$ to the  arc $\check\tau_{-1}(I^s(\widebar p_{2j-1}^m))$ in $\widebar T_{2j-1}^0$;
 \item[--] $A^m_-$ is joining the arc $\check\tau_{-1}(I^s(\widebar p_{2j-1}^m))$ to the arc $I^s_{\widebar p_{2j}^m}$ inside the tunnel $E_j^-$; it is a neighborhood of the orbit segment joining $\check\tau_{-1}(\widebar p_{2j-1}^m)$ to $\widebar p_{2j}^m$ in a leaf of $\widebar\cF^{s,0}$.
\end{itemize}
See figure~\ref{f.first-isotopy-B}.

\begin{figure}[ht]
\centerline{\includegraphics[width=12cm]{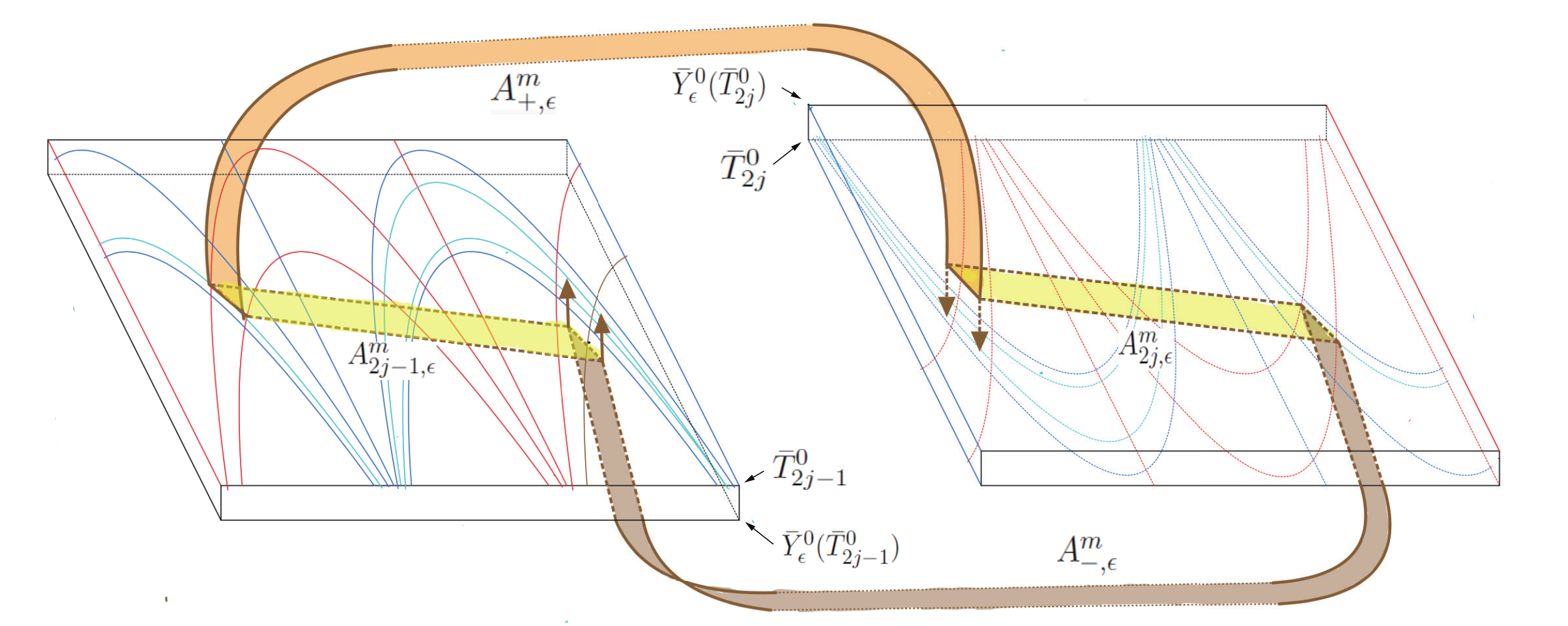}}
\vspace{0.5cm}
\centerline{\includegraphics[width=12cm]{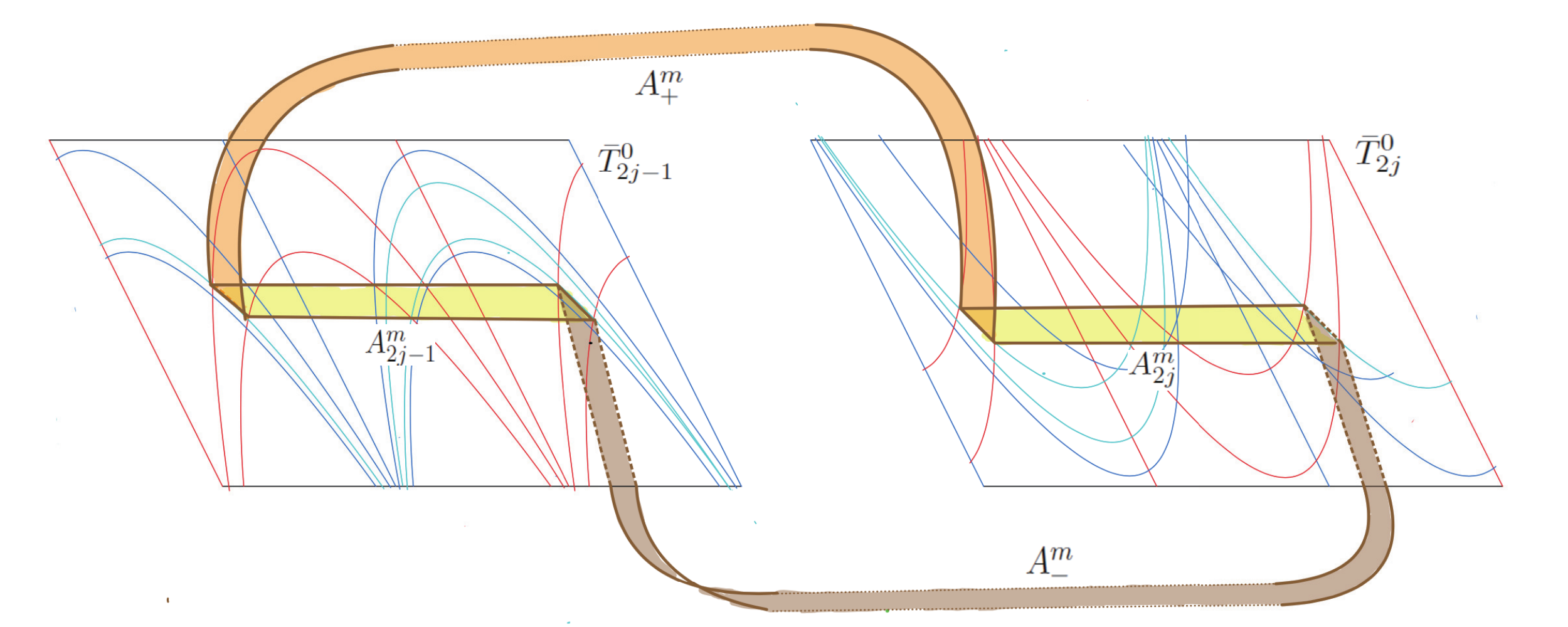}}
\caption{\label{f.first-isotopy-B}Isotopy from $A^m_1=\widebar h_1(W^s_{loc}(\widebar\alpha_j^m))$ to $A^m=\lim_{\epsilon\to 0} \widebar h_\epsilon(W^s_{loc}(\widebar\alpha_j^m))$.}
\end{figure}

\subsubsection{Step 5. Isotopy from $A^m=\lim_{\epsilon\to 0}\widebar h_\epsilon(W^s_{loc}(\widebar\alpha_j^m))$ to a ribbon $\widehat A^m$.}
In the preceding step, we have described a decomposition $A^m=A^m_{2j}\cup A^m_+\cup A^m_{2j-1}\cup A^m_-$ of the ribbon $A^m=\lim_{\epsilon\to 0}\widebar h_\epsilon(W^s_{loc}(\widebar\alpha_j^m))$. Now, we deform $A^m$ by isotopy as follows. 

The common end $\check\tau_{-1}(I^s(\widebar p_{2j}^m))$ of the ribbons $A^m_{2j}$ and $A^m_+$ is contained in the leaf $\ell^s(\check\tau_{-1}(\widebar p_{2j}^m))$ of the foliation $\widebar f^{s,0}_{2j}$. The leaf $\ell^s(\check\tau_{-1}(\widebar p_{2j}^m))$  goes from one end of the stable strip $\widebar\pi^0(D^s_{2j})$ to the other. We deform $A^m_{2j}$ and $A^m_+$ by ``dragging" their common end along the leaf $\ell^s(\check\tau_{-1}(\widebar p_{2j}^m))$, in order to bring it inside the same connected component of $\widebar\pi^0(D^s_{2j})\cap \widebar\pi^0(D^u_{2j})$ as the arc $I^s(\widebar p^0_{2j})$. See the upper part of Figure~\ref{f.isotopy-B}. This can be done by an isotopy supported in the collar neighborhood $V_{2j}^\epsilon$ of the torus $\widebar T_{2j}^0$, for an arbitrarily small $\epsilon$, which respects the  level sets  $\{t=\mbox{constant}\}$ of the $(x,y,t)$ coordinates system on $V_{2j}^\epsilon$ (see Step 1). Moreover, we can choose an isotopy during which $A^m_{2j}$ moves inside the torus $\widebar T_{2j}^0$, and $A^m_+$ moves inside the tunnel $E_j^+$ (recall that this tunnel is the forward orbit of the strip $\widebar\pi^0(D^s_{2j})$ in $\widebar\pi^0(U^+)$) and inside a leaf of the foliation $\widebar\cF^{s,0}$. At the end of the isotopy, we obtain two new ribbons $\widehat A^m_{2j}$ and $\widehat A^m_+$ depicted on the lower part of Figure~\ref{f.isotopy-B}. Adjusting the width of these ribbons if necessary, we may assume that their common end is a crossbar of the unstable strip $D^u_{2j}$. 

Notice that the ribbon $\widehat A^m_+$ is in good position inside the tunnel $E_j^+$. This is a consequence of three facts: first the initial ribbon $A^m_+$ is in good position since it is a neighborhood of an orbit segment in a leaf of $\widebar\cF^{s,0}$; second the ribbon remains in this leaf of  $\widebar\cF^{s,0}$ and in the tunnel $E_j^+$ all along the isotopy; third the isotopy is supported in the neighborhood $V_{2j}^\epsilon$ and respects the level sets $\{t=\mbox{constant}\}$ of the $(x,y,t)$ coordinates system and $\frac{\partial}{\partial t}$ corresponds to the vector field $(\widebar\pi^0)_*X=\widebar Y^0$.  

We deform similarly the ribbons $A^m_{2j-1}$ and $A^m_-$ in a small neighborhood of the torus $\widebar T_{2j-1}^0$, leading to new ribbons $\widehat A^m_{2j-1}$ and $\widehat A^m_-$ (see again Figure~\ref{f.isotopy-B}). By construction, the closed ribbons $A^m=A^m_{2j}\cup A^m_+\cup A^m_{2j-1}\cup A^m_-$  and $\widehat A^m=\widehat A^m_{2j}\cup \widehat A^m_+\cup \widehat A^m_{2j-1}\cup \widehat A^m_-$ are isotopic.

\begin{figure}[ht]
\centerline{\includegraphics[width=12cm]{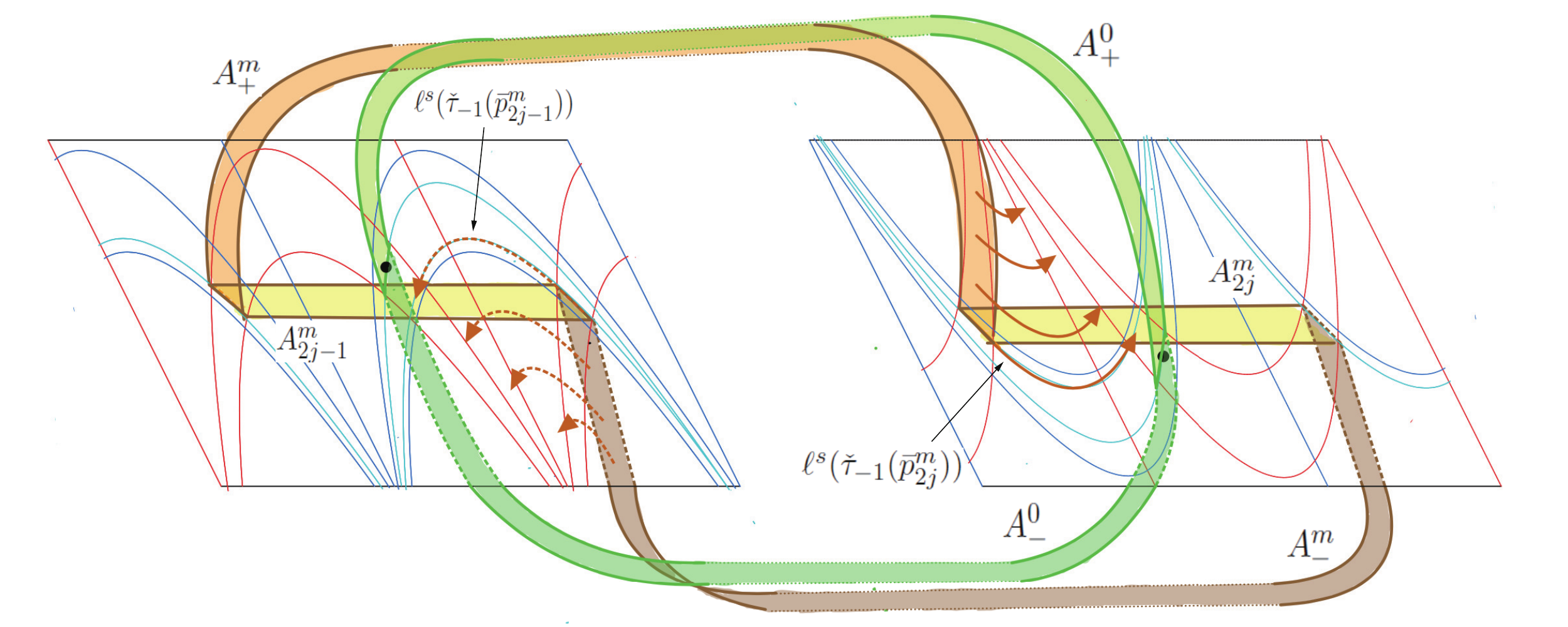}}
\vspace{0.5cm}
\centerline{\includegraphics[width=12cm]{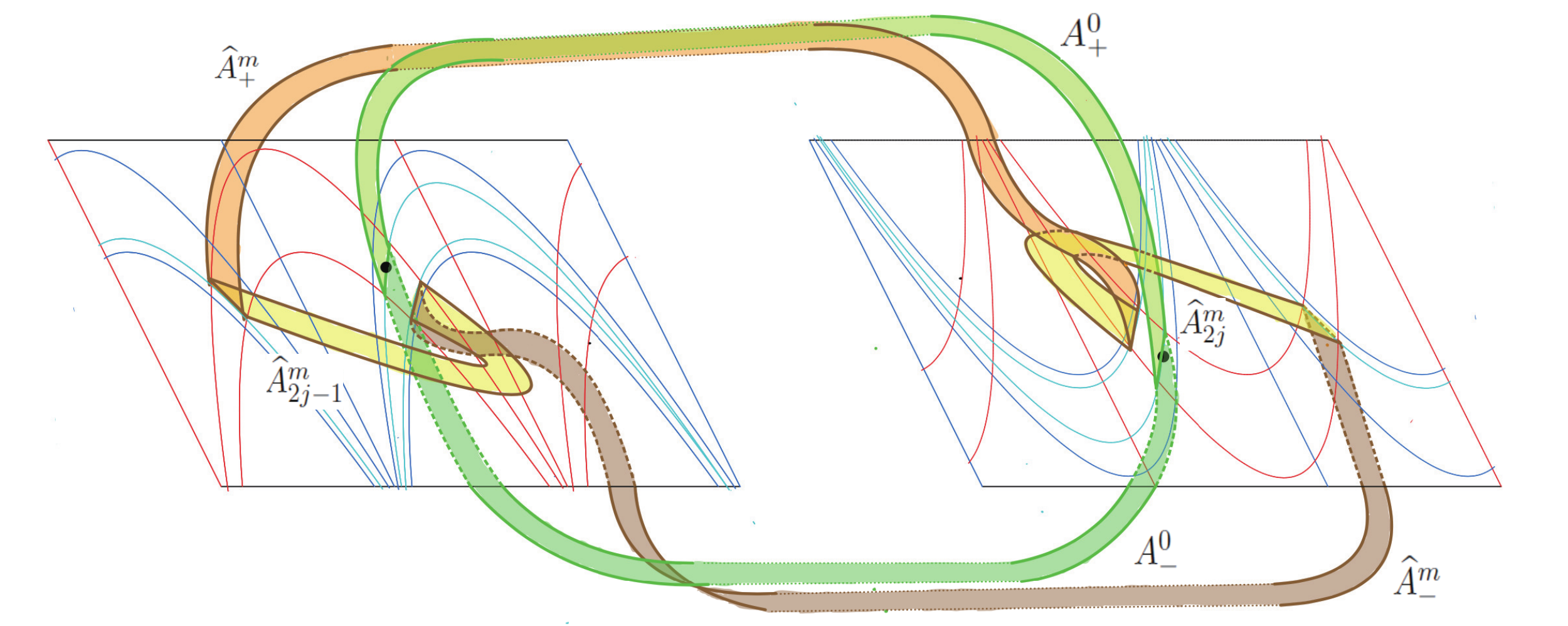}}
\caption{\label{f.isotopy-B}Isotopy from the closed ribbon $A^m=A^m_{2j}\cup A^m_+\cup A^m_{2j-1}\cup A^m_-$ to the closed ribbon $\widehat A^m=\widehat A^m_{2j}\cup \widehat A^m_+\cup \widehat A^m_{2j-1}\cup \widehat A^m_-$.}
\end{figure}

\subsubsection{Step 6. Isotopy from the closed ribbon $A^0=W^s_{loc}(\widebar\alpha_j^0)$ to a ribbon $\widehat A^0$.}
Now we will deform the closed ribbon $A^0=A^0_+\cup A^0_-$ into a closed ribbon $\widehat A^0=\widehat A^0_{2j}\cup \widehat A^0_+\cup \widehat A^0_{2j-1}\cup \widehat A^0_-$ satisfying the following properties:
\begin{itemize}
\item[--] $\widehat A^0_\pm$ is a ribbon in good position in the tunnel $\widebar\pi^0(E_j^\pm)$, with the same ends as the $\widehat A^m_\pm$;
\item[--] $\widehat A^0_{2j}$ is the section of the unstable strip $\widebar\pi^0(D^u_{2j})$ connecting the right end\footnote{We are considering ribbons with one end in the torus $\widebar T_{2j-1}^0$ and one end in the torus $\widebar T_{2j}^0$. For sake of simplicity, we call the former \emph{the left end} of the ribbon, and the later \emph{the the right end} of the ribbon (because, on our figures, $\widebar T_{2j-1}^0$ is on the left and $\widebar T_{2j}^0$ on the right).} of $\widehat A^m_-$ to the right end of  $\widehat A^m_+$ ;
\item[--] $\widehat A^0_{2j-1}$ is the section of the unstable strip $\widebar\pi^0(D^u_{2j-1})$ connecting the left end of $\widehat A^m_-$ to the left end of  $\widehat A^m_+$.
\end{itemize}
See the upper part of Figure~\ref{f.isotopy-A}. Let us describe more precisely the isotopy. Both the right end of $A^0_-$ and the right end of $\widehat A^m_-$ are stable crossbars of the unstable strip $\pi_{\widebar\varphi}^0(D^u_{2j})$. So we can ``drag" the right end of $A^0_-$ among cross bars of $D^u_{2j}$ to bring it up to the right end of $\widehat A^m_-$ (see Figure~\ref{f.isotopy-A}). This can be done by an isotopy supported in the intersection of the tunnel $E_j^-$ and a small collar neighborhood of $\widebar T^{2j}_0$ in $\widebar\pi^0(U^-)$ (recall that $E_j^+$ is the backward orbit of the unstable strip $\widebar\pi^0(D_{2j}^u)$ in $\widebar\pi^0(U^-)$). Moreover this can be done by an isotopy among ribbons that are in good position inside $E_j^-$(recall that, at the beginning of the isotopy, $A^0_-$ is a piece of local stable manifold, hence is in good position). Now, the right end of $A^0_+$ and the right end of $\widehat A^m_+$ are contained in the same connected component of $\widebar\pi^0(D_{2j}^u)\cap\widebar\pi^0(D_{2j}^s)$, and both of them are stable crossbars of the unstable strip $D_{2j}^u$. So we can drag the right end of $A^0_+$ inside  $\widebar\pi^0(D_{2j}^u)\cap \widebar\pi^0(D_{2j}^s)$, among crossbars of $\widebar\pi^0(D^u_{2j})$, to bring it up to the right end of $\widehat A^m_+$. This can be done by an isotopy among ribbons in good position inside the tunnel $\widebar\pi^0(E_j^+)$(recall that $\widebar\pi^0(E_j^+)$ is the forward orbit of the stable strip $\widebar\pi^0(D^s_{2j})$ in $\widebar\pi^0(U^+)$). At this stage, we have explained how to move the right ends of $A^0_-$ and $A^0_+$ up to the right ends of $\widehat A^m_-$ and $\widehat A^m_+$, by isotopies of $A^0_-$ and $A^0_+$. But such isotopies move the right ends of $A^0_-$ and $A^0_+$ appart from each other, and therefore create a ribbon connecting these two ends. Obviously, at the end of the isotopy, this ribbon is nothing but the segment of the unstable strip $D^u_{2j}$ connecting the right ends of $\widehat A^m_-$ and $\widehat A^m_+$; we call it $\widehat A^m_{2j}$. We perform a similar operation in a neighborhood the torus $\widebar T_{2j-1}^0$ to bring the lefts ends of $A^0_-$ and $A^0_+$ towards the ends of $\widehat A^m_-$ and $\widehat A^m_+$. At the end of the day, we obtain that the closed ribbon $A^0=W^s_{loc}(\widebar\alpha_j^0)$ is isotopic to a closed ribbon $\widehat A^0=\widehat A^0_{2j}\cup \widehat A^0_+\cup \widehat A^0_{2j-1}\cup \widehat A^0_-$ where $\widehat A^0_\pm$, $\widehat A^0_{2j-1}$ and $\widehat A^0_{2j}$ satisfy the properties announced above (see the lower part of Figure~\ref{f.isotopy-A}).

\begin{figure}[ht]
\centerline{\includegraphics[width=12cm]{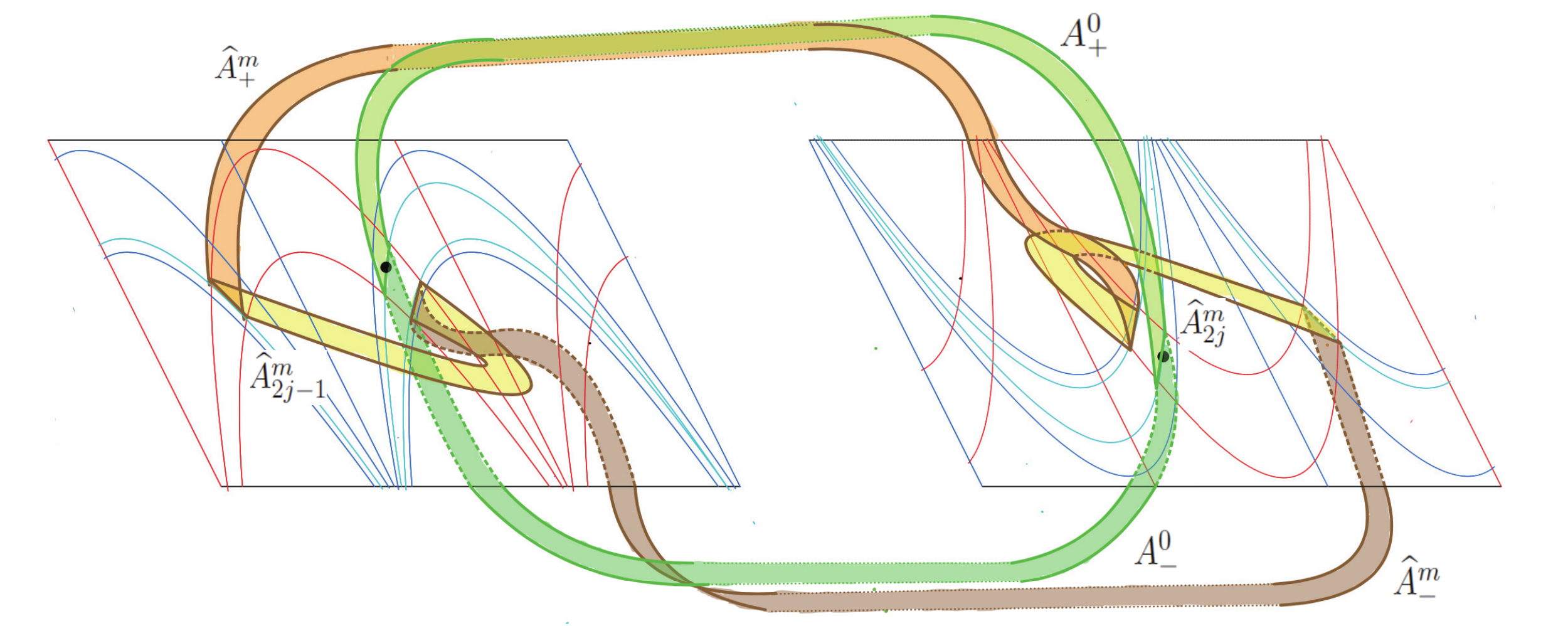}}
\vspace{0.5cm}
\centerline{\includegraphics[width=12cm]{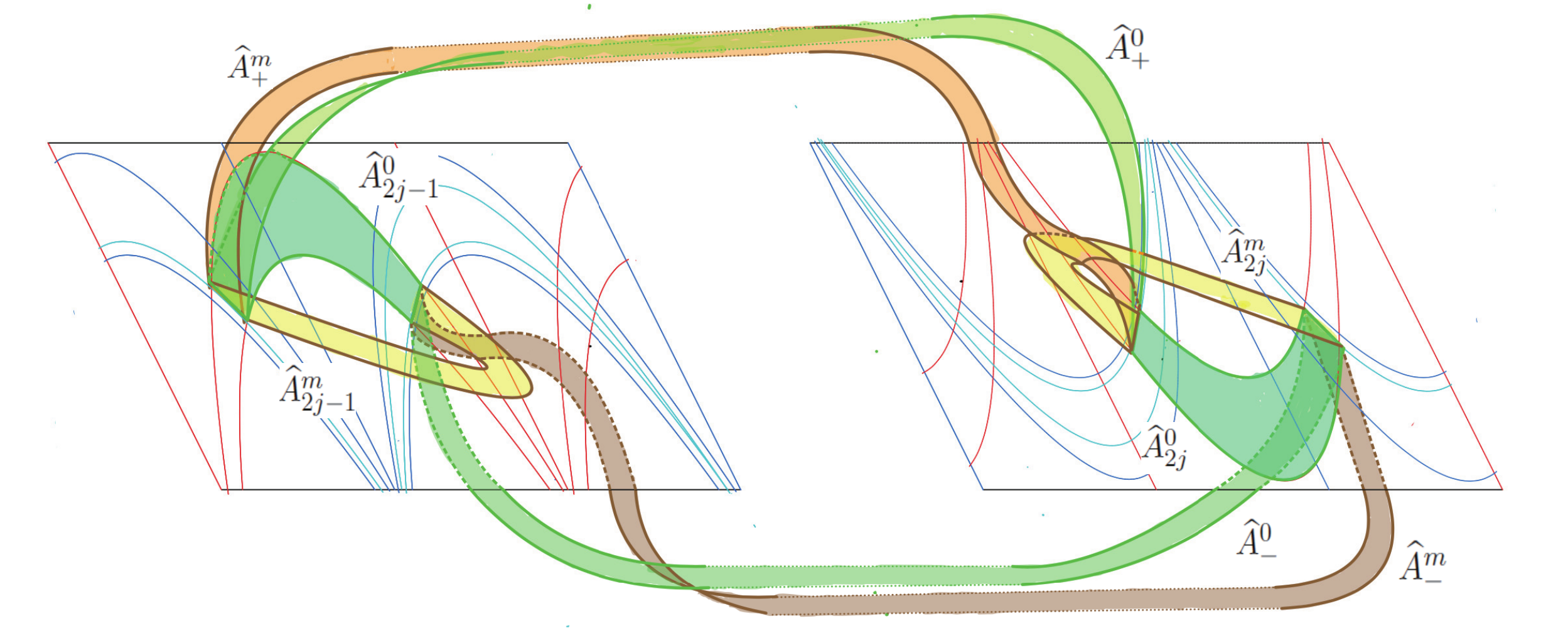}}
\caption{\label{f.isotopy-A}Isotopy between $A^0=A^0_+\cup A^0_-$ and $\widehat A^0=\widehat A^0_{2j}\cup \widehat A^0_+\cup \widehat A^0_{2j-1}\cup \widehat A^0_-$.}
\end{figure}

\subsubsection{Step 7. Isotopy between the closed ribbons $\widehat A^0$ and $\widehat A^m$.}
We will now prove that the ribbons $\widehat A^0=\widehat A^0_{2j}\cup \widehat A^0_+\cup \widehat A^0_{2j-1}\cup \widehat A^0_-$ and $\widehat A^m=\widehat A^m_{2j}\cup \widehat A^m_+\cup \widehat A^m_{2j-1}\cup \widehat A^m_-$ are isotopic. Recall that $\widehat A^0_+$ and $\widehat A^m_+$ are two ribbons with the same ends. Both $\widehat A^0_+$ and $\widehat A^m_+$ are contained in the tunnel $\widebar\pi^0(E_j^+)$, and are in good position inside this tunnel. By the observation at the end of our preliminaries, it follows that $\widehat A^0_+$ and $\widehat A^m_+$ are isotopic (with fixed ends). By a similar argument, the ribbons $\widehat A^0_-$ and $\widehat A^m_-$ are isotopic with fixed ends. 

So we are left to compare the closed ribbons $\widehat A^0=\widehat A^0_{2j}\cup \widehat A^0_+\cup \widehat A^0_{2j-1}\cup \widehat A^0_-$  and $\widehat A^m_{2j}\cup \widehat A^0_+\cup \widehat A^m_{2j-1}\cup \widehat A^0_-$. This is equivalent to understand what happens when we replace $\widehat A^0_{2j}$ by $\widehat A^m_{2j}$ and replace  $\widehat A^0_{2j-1}$ by $\widehat A^m_{2j-1}$ in the closed ribbon $\widehat A^0$. 

\begin{claim}
\label{c.core-homoopic-to-zero}
The core of the closed ribbon $\widehat A^0_{2j}\cup \widehat A^m_{2j}$ is homotopic to $0$ in the torus $T_{2j}^{in}$. 
\end{claim}

\begin{proof}
We need to recall that the construction of the orbits $\widebar\alpha_{j}^0$ and $\widebar\alpha_j^m$ included the choice of some connected components $\widebar R_{2j}$ and $\widebar L_{2j}$ of $\widebar\varphi^0(D^u_{2j})\cap D^s_{2j}$ and $\widebar\varphi^m(D^u_{2j})\cap D^s_{2j}$ respectively. Recall $\widebar q^0_{2j}\in \widebar\pi^0(\widebar R_{2j})$ and $\widebar q^m_{2j}\in \widebar\pi^m(\widebar L_{2j})$, which implies that $\widebar p^m_{2j}\in \check\tau_{-1}(\widebar\pi^0(\widebar L_{2j}))$. Also recall that the connected components $\widebar R_{2j}$ and $\widebar L_{2j}$ were chosen so that $\widebar L_{2j}=\vartheta(\widebar R_{2j})$ where $\vartheta$ is the symmetry with respect to the core of the annulus $A_{2j}^{1,s}$. 

Performing the isotopy of step 6 backwards, we see that the core of the closed ribbon $\widehat A^0_{2j}\cup \widehat A^m_{2j}$ is homotopic to the concatenation of three arcs: 
\begin{itemize}
\item[--] an arc $a$ joining $\check\tau_{-1}(\widebar p^m_{2j})\in\widebar\pi^0(\widebar L_{2j})$ to a point in $\widebar\pi^0(\widebar R_{2j})$ inside the strip $\widebar\pi^0(D^s_{2j})$,
\item[--] an arc $b$ joining the end of $a$ in $\widebar\pi^0(\widebar R_{2j})$ to $\widebar p^m_{2j}\in\check\tau_{+1}(\widebar\pi^0(\widebar L_{2j}))$ inside the strip $\widebar\pi^0(D^s_{2j})$,
\item[--] a straight horizontal\footnote{By ``horizontal", we mean along which the $y$-coordinate is constant.} segment $c$ joining  $\widebar p^m_{2j}\in\check\tau_{+1}(\widebar\pi^0(\widebar L_{2j}))$ to $\check\tau_{-1}(\widebar p^m_{2j})\in\widebar\pi^0(\widebar L_{2j})$.
\end{itemize}  
See Figure~\ref{f.core-homotopic-to-0}. The axial symmetry $\vartheta$ maps $\widebar L_{2j}$ to $\widebar R_{2j}$ and leaves the strip $D_{2j}^s$ invariant.  Observe that $\tau_{+1}\circ\vartheta$ is the symmetry with respect to the axis of the annulus $\tau_{+1/2}(A^{s,1}_{2j})=\widebar\varphi^0(A^{u,1}_{2j})$. Hence $\tau_{+1}\circ\vartheta$ maps $\widebar R_{2j}$ to $\tau^{-1}(\widebar L_{2j})$ and leaves the strip $\widebar\varphi^0(D_{2j}^u)$ invariant. As a consequence, the arcs $a$ and $b$ can be chosen so that $a$ is globally invariant under the axial symmetry $\check\vartheta:=\widebar\pi^0\circ\vartheta\circ(\widebar\pi^0_{T^{in}_{2j}})^{-1}$ and  $b$ is globally invariant under the axial symmetry $\check \tau_{+1}\circ\check\vartheta$. Now observe that, if an arc is invariant under an axial symmetry with vertical axis, then the total variation of the $y$-coordinate along this arc is equal to $0$. Hence the total variation of $y$-coordinate along the concatenaion of $a$ and $b$ is equal to $0$. This means that the concatenation of $a$ and $b$ is homotopic (with fixed ends) to the horizontal segment $c$. Equivalently,  the core of the closed ribbon $\widehat A^0_{2j}\cup \widehat A^m_{2j}$ is homotopic to $0$.
\end{proof}

\begin{figure}[ht]
\centerline{\includegraphics[width=6.5cm]{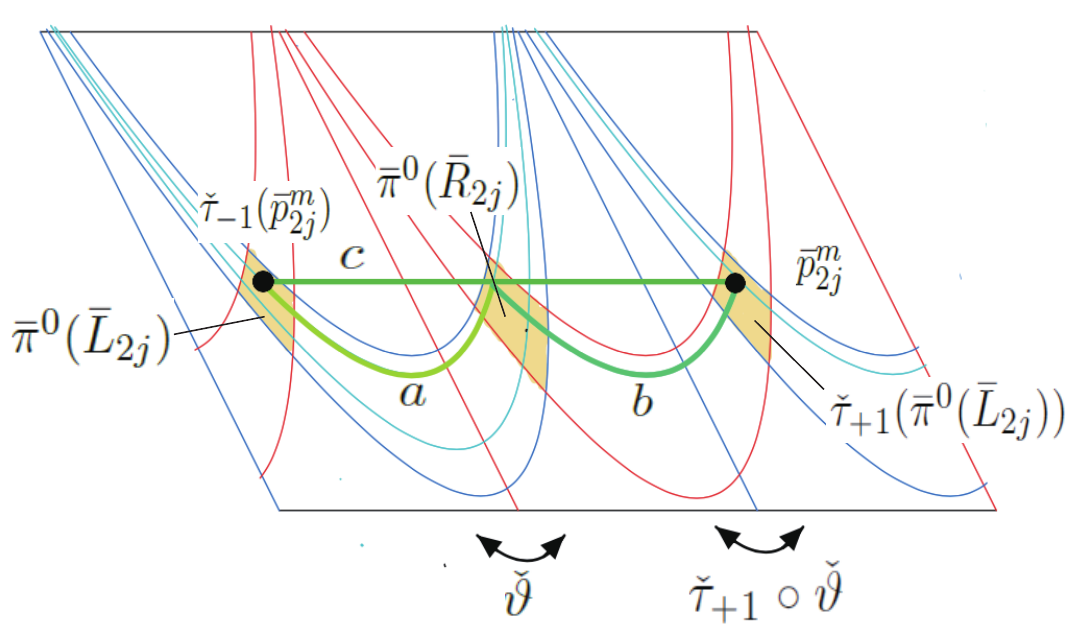}}
\caption{\label{f.core-homotopic-to-0}The arcs $a$, $b$ and $c$.}
\end{figure}

Of course, we may prove similarly that the core of the closed ribbon $\widehat A^0_{2j-1}\cup \widehat A^m_{2j-1}$ is homotopic to $0$ in the torus $T_{2j-1}^{in}$. It follows that the cores of the closed ribbons $\widehat A^0=\widehat A^0_{2j}\cup \widehat A^0_+\cup \widehat A^0_{2j-1}\cup \widehat A^0_-$ and $\widehat A^m_{2j}\cup \widehat A^0_+\cup \widehat A^m_{2j-1}\cup \widehat A^0_-$ are homotopic. As a consequence, these two closed ribbons differ by a certain number of half-twists. We have to show that number is equal to $0$. 

For this purpose, let us recall that our figures depict faithfully the situation close to $\widebar T_{2j-1}^0$ and $\widebar T_{2j}^0$, since we have a coordinates system where the flow and the foliations are very simple and fully explicit (see Remarks~\ref{r.faithful} and~\ref{r.faithful-2}). Observing carefully Figure~\ref{f.half-twists}, we see replacing $\widehat A^0_{2j-1}$ by $\widehat A^m_{2j-1}$ amounts to a left-handed half-twist, whereas replacing $\widehat A^0_{2j}$ by $\widehat A^m_{2j}$ amounts to a right-handed half-twist. These two half-twists cancel each other out, hence $\widehat A^0=\widehat A^0_{2j}\cup \widehat A^0_+\cup \widehat A^0_{2j-1}\cup \widehat A^0_-$ is isotopic to $\widehat A^m_{2j}\cup \widehat A^0_+\cup \widehat A^m_{2j-1}\cup \widehat A^0_-$. 

We conclude that the ribbons $\widehat A^0=\widehat A^0_{2j}\cup \widehat A^0_+\cup \widehat A^0_{2j-1}\cup \widehat A^0_-$ and $\widehat A^m=\widehat A^m_{2j}\cup \widehat A^m_+\cup \widehat A^m_{2j-1}\cup \widehat A^m_-$ are isotopic.

\begin{figure}[ht]
\centerline{\includegraphics[width=12cm]{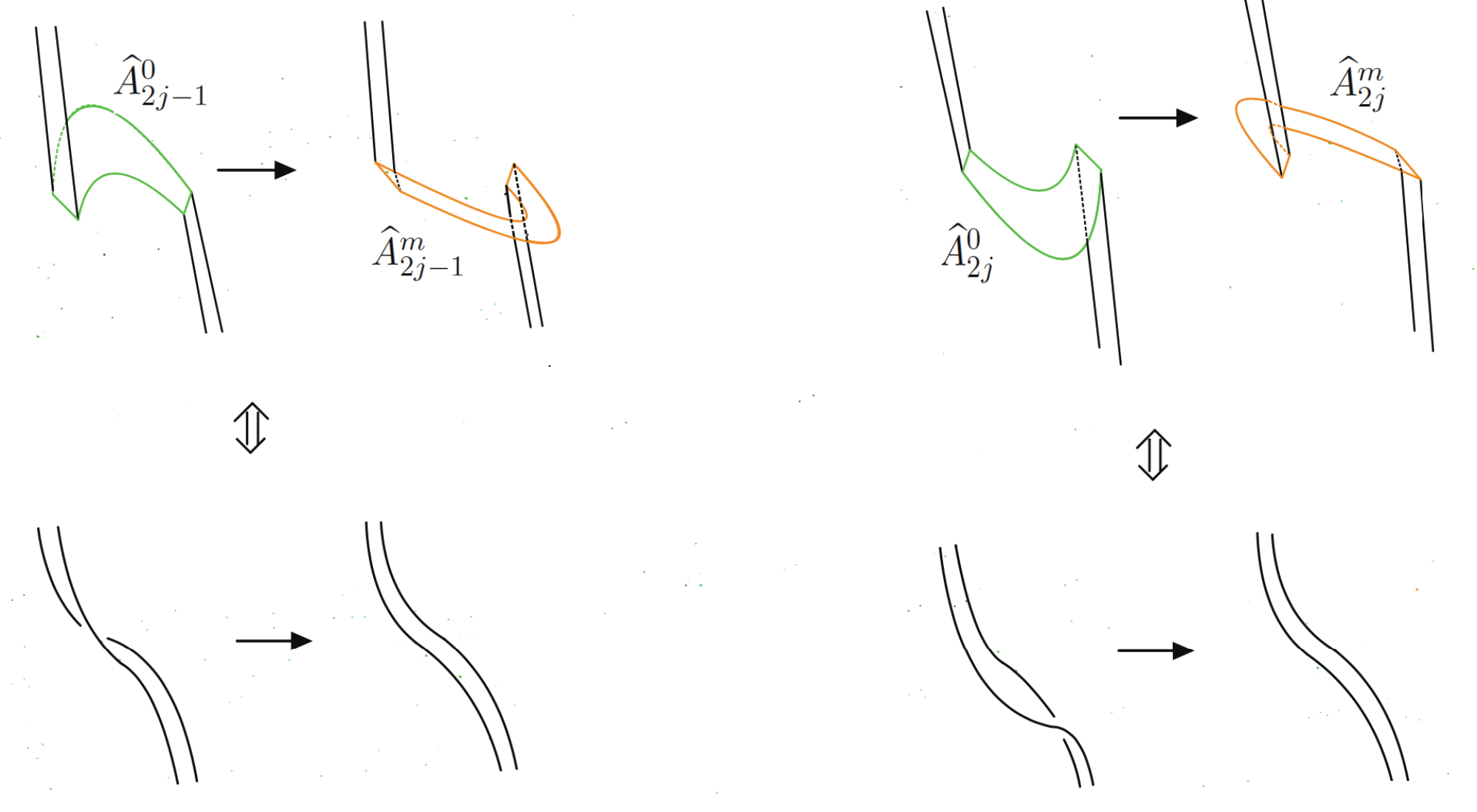}}
\caption{\label{f.half-twists}Replacing $\widehat A^0_{2j-1}$ by $\widehat A^m_{2j-1}$ (resp. $\widehat A^0_{2j}$ by $\widehat A^m_{2j}$) amounts to a left-handed (resp. right-handed) half-twist. }
\end{figure}

\subsubsection{End of the proof of Proposition~\ref{p.h1}.}
Putting together steps 4 to 7, we conclude that the closed ribbons  $A^0=W^s_{loc}(\widebar\alpha_j^0)$ and $A^m=\widebar h_1(W^s_{loc}(\widebar\alpha_j^m))$ are isotopic for $j\leq 2m$. Observe that the isotopy we have described is supported in the union of the tunnel $E_j^+$, the tunnel $E_j^-$, a small tubular neighborhood of the torus $\widebar T_{2j-1}^0$ and a small tubular neighborhood of the torus $\widebar T_{2j}^0$. In particular, the isotopies corresponding to the different integers $j\in\{1,\dots,2m\}$ have disjoint supports. Also recall that the closed ribbons  $W^s_{loc}(\widebar\alpha_j^0)$ and $\widebar h_1(W^s_{loc}(\widebar\alpha_j^m))$ are equal for $j\geq 2m$. We conclude that the ribbon links $W^s_{loc}(\widebar\alpha_0^0)\cup\dots\cup W^s_{loc}(\widebar\alpha_{2n}^0)$ and $\widebar h_1(W^s_{loc}(\widebar\alpha_0^m))\cup\dots\cup\widebar h_1(W^s_{loc}(\widebar\alpha_{2n}^m))$ are isotopic and Proposition~\ref{p.h1} is proved. This concludes the proof of Theorem~\ref{t.con}.

%% file: Hyperbolic-manifold.tex
In the previous section, for $m=0,\dots,2n$, we have constructed a closed toroidal $3$-manifold $W^m$, a transitive Anosov flow $Y^m_t$ on $W^m$, and $2n$ periodic orbits $\alpha_1^m,\dots,\alpha_{2n}^m$ of the Anosov flow $Y^m_t$. Using Dehn-Fried surgeries, we will now turn $W^m$ into a hyperbolic manifold, and turn $Y^m_t$ into a non-$\RR$-covered Anosov flow on this hyperbolic manifold. 

 Let $V(\alpha_j^m)$ be a small tubular neighborhood of the periodic orbit $\alpha_j^m$. We define an oriented meridian $\mu_j^m$ and an oriented longitude $\lambda_j^m$ of the torus 
$\partial V(\alpha_j^m)$ as follows. The tubular neighbourhood $V(\alpha_j^m)$ can be deformed so that so that $V(\alpha_j^m)\cap T_{2j}^m$ is a single small disc $d_{2j}^m$ 
and 
$W_{loc}^s(\alpha_j^m)\cap \partial V(\alpha_j^m)$ is made of two parallel closed curves, each of which is homotopic to the orbit $\alpha_j^m$ in $W^s_{loc}(\alpha_j^m)$. The orientation of $W^m$ induces an orientation on the transverse disc $d_{2j}^m$ (for which the basis $(\vec{e}_1,\vec{e}_2)$ defined in subsection~\ref{ss.orientation} is a direct basis). We set $\mu_j^m$ to be the oriented boundary of the small disc $d_{2j}^m$. We set $\lambda_j^m$ to be one of the two connected components of 
$W_{loc}^s(\alpha_j^m)\cap \partial V(\alpha_j^m)$, equipped with the orientation induced by the dynamical of the orbit $\alpha_j^m$. See Figure~\ref{f.ortWm}. One may easily check that the homology classes of $\mu_j^m$ and $\lambda_j^m$ in $H_1(\partial V(\alpha_j^m))$ do not depend on the choices in the above definition. Also observe that the homology class of $\mu_j^m$ would be unchanged if one  replaces the torus $T_{2j}^m$ by the torus $T_{2j-1}^m$ in the definition.

 \begin{figure}[htp]
\begin{center}
  \includegraphics[totalheight=3.8cm]{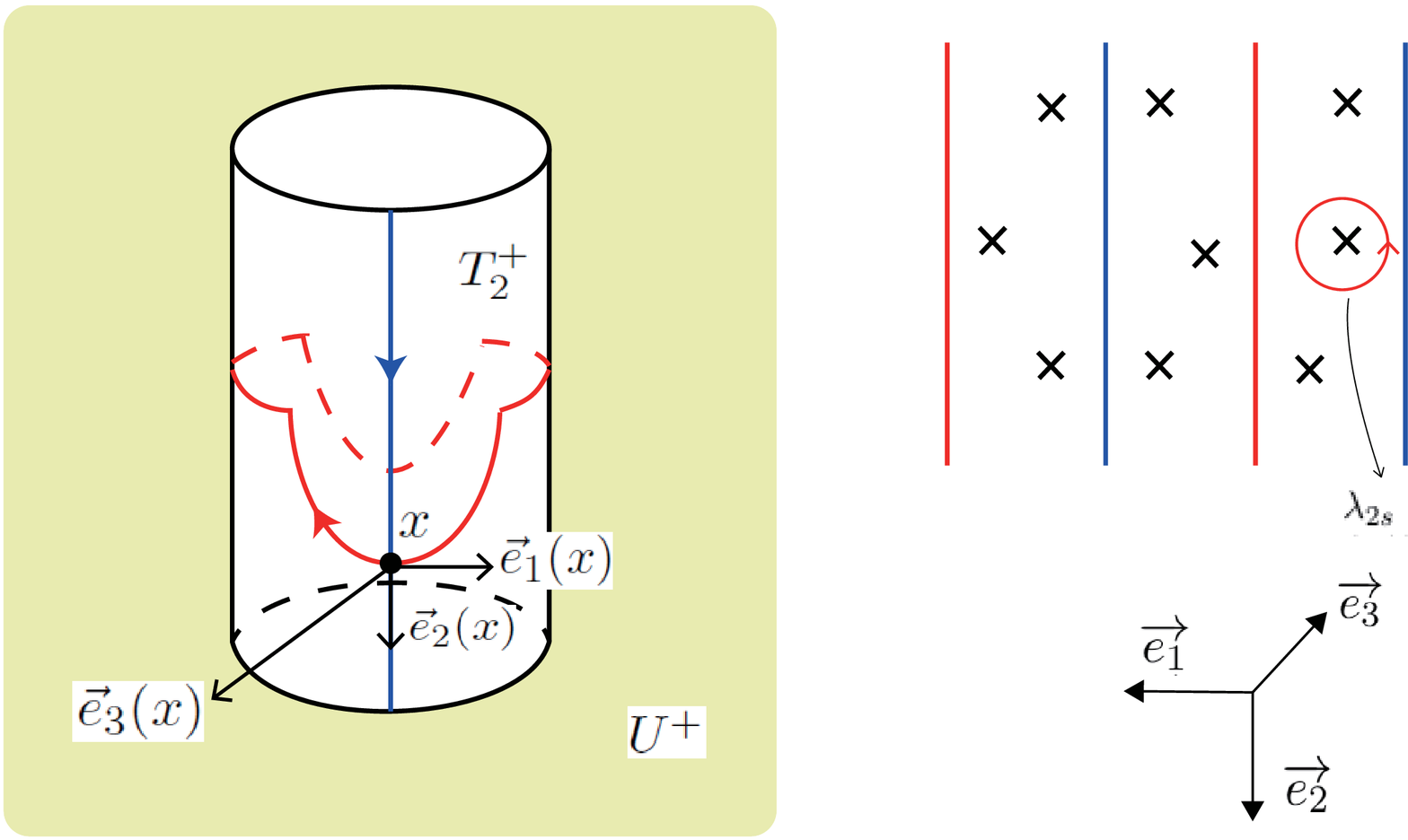}\\
  \caption{The oriented meridian $\mu_j^m$ and the oriented longitude $\lambda_j^m$}\label{f.ortWm}
\end{center}
\end{figure}

Performing a Dehn-Fried surgery on the orbit $\alpha^m_j$ consists in blowing-up and blowing-down $\alpha^m_j$ in order to ``change the meridian circle" (we refer to~\cite{Sha} for a precise definition). Given an integer $k\in\ZZ$, we call a \emph{$k$-Dehn-Fried surgery} on the orbit $\alpha^m_j$ a Dehn-surgery so that $\mu_{j}^m+ k \lambda_{j}^m$ will be a meridian circle after surgery.

\begin{definition}
Fix an integer $k\in\ZZ$. Starting from the manifold $W^m$ and the Anosov flow $Y^m_t$ constructed in section~\ref{s.con}, we perform a $k$-Dehn-Fried surgery on each of the $2n$ orbits $\alpha_1^m,\dots,\alpha_{2n}^m$. We will denote by $M^m$ the resulting $3$-manifold and by $Z^m_t$ the resulting Anosov flow\footnote{More precisely, the flow resulting of such Dehn-surgeries is a topological Anosov flow, well-defined up to orbital equivalence, and the work of Shannon (\cite{Sha}) ensures that this flow is actually orbitally equivalent to a true Anosov flow.}. 
\end{definition}

\begin{remark}
We omit to report the dependence in $k$ of the manifold $M^m$ and the Anosov flow $Z^m_t$ since this integer is fixed once forever.
\end{remark}

\begin{proposition}\label{p.homeomorphic-manifolds}
 For any $m_1,m_2\in\{0,1,\dots, 2n\}$, the manifolds $M^{m_1}$ and $M^{m_2}$ are homeomorphic. 
 \end{proposition}

%

\begin{proof}
 Let $m_1,m_2\in\{0,1,\dots, 2n\}$. For $i=1,2$, the manifold $M^{m_i}$ is obtained (up to homeomorphism) by considering the manifold $W^{m_i}$, by removing small open tubular neighbourhood $V(\alpha^{m_i}_j)$ of the periodic orbit $\alpha^{m_i}_j$ for $j=1,\dots,2n$, by considering a circle fibration of the two-torus $\partial V(\alpha^{m_i}_j)$ whose fibers is homologous to $\mu_{j}^{m_i}+ k \lambda_{j}^{m_i}$ for $j=1,\dots,2n$, and by squeezing each fiber of this fibration to a single point.  According to item~(3) of Theorem~\ref{t.con}, there exists a homeomorphism $H:W^{m_1}\to W^{m_2}$ which maps the local stable manifold $W_{loc}^s(\alpha^{m_1}_j)$ of the orbit $\alpha^{m_1}_j$ to the  local stable manifold $W_{loc}^s(\alpha^{m_2}_j)$ of the orbit $\alpha^{m_2}_j$ for every $j\in\{1,\dots,2n\}$. It follows that $H$ maps a tubular neighborhood $V(\alpha^{m_1}_j)$ of $\alpha^{m_1}_j$ to a tubular neighborhood $V(\alpha^{m_2}_j)$ of $\alpha^{m_2}_j$, and maps a circle homologous to $\mu_{j}^{m_1}+ k \lambda_{j}^{m_1}$ to a circle homologous to $\mu_{j}^{m_2}+ k\lambda_{j}^{m_2}$. Hence $H$ induces a homeomorphism from $M^{m_1}$ to $M^{m_2}$. 
\end{proof}

\begin{proposition}\label{p.hyperbolic}
If $|k|$ was chosen large enough, then the manifold $M^m$ is hyperbolic.
\end{proposition}

In order to prove this proposition, it is convenient to introduce the following notion: a link $\Gamma$ in a closed orientable $3$-manifold $W$ is said to be \emph{full-filling} if it intersects every incompressible torus in $W$ up to isotopy. The following lemma is an analogy of a similar result for geodesic flows proved in Appendix B of \cite{FH}: 

\begin{lemma}\label{l.hypknot}
Let $Y_t$ be an Anosov flow on a closed orientable $3$-manifold~$W$, and $\Gamma$ be a link made of finitely many periodic orbits of $Y_t$. Assume that  $\Gamma$ is full-filling, and that every periodic orbit $\gamma \in \Gamma$ intersects a two-torus $T_{\gamma}\subset M$ which is transverse to $Y_t$. Then $W_\Gamma$ is a hyperbolic manifold, where $W_\Gamma$ is the path closure of $W\setminus\Gamma$.
\end{lemma}

 Let $\Gamma$ be the union of finitely many periodic orbits  of an Anosov flow $Y_t$ on a $3$-manifold $W$. Only the condition that
 $\Gamma$ is full-filling can not guarantee that $\Gamma$ is a hyperbolic link in $W$. For instance, if there are two periodic orbits $\gamma_1, \gamma_2 \in \Gamma$ so that they bound an embedded annulus \footnote{This type of examples can be found in \cite{BarFe2} and \cite{BBY}.}, then we can easily construct a
torus $T$ in the exterior of $\Gamma$ (in $W$) that bounds a compact $3$-manifold with trivial pants-fibration structure.
Therefore, in this case $W\setminus\Gamma$ is not hyperbolic. But under the condition of 
 Lemma \ref{l.hypknot}, there is an important property (Proposition \ref{p.homotopyclass}) about the free homotopy classes of the periodic orbits to ensure that the above phenomena does not happen.
This property is actually an application of the topology theory of $3$-dimensional  Anosov flows built by Barbot
 (\cite{Ba}, \cite{Ba2}, \cite{Ba3}, etc) and Fenley (\cite{Fen1}, \cite{Fen2}, etc). We start with the following lemma, which is a direct consequence of Lemma 3.32 of \cite{Bart} that relies on the theory established by Barbot and Fenley.
\begin{lemma}\label{l.loz}
Let $\alpha$ and $\beta$ be two periodic orbits of an Anosov flow $Y_t$ on a closed $3$-manifold $W$ such that $\alpha$ and $\beta$ are freely homotopic.  Then there  exists
a periodic orbit $\gamma$ of $Y_t$ so that $\alpha^i$ and $\gamma^j$ are freely homotopic for some
 $i\in\{1,2\}$ and $j\in \{-2,-1\}$. 
\end{lemma}

The following proposition will be very useful to prove Lemma \ref{l.hypknot}.
\begin{proposition}\label{p.homotopyclass}
Let  $\alpha$ and $\beta$ be two periodic orbits of an Anosov flow $Y_t$ on a closed $3$-manifold $W$ with transverse torus such that at least one of $\alpha$ and $\beta$
transversely intersects  a transverse torus $T$ of $Y_t$. Then, for every two nonzero integers $p$ and $q$, $\alpha^p$ and $\beta^q$ are not freely homotopic. In particular, $\alpha$ and $\beta$ are not freely homotopic.
\end{proposition}
\begin{proof}
We assume by contradiction that  there exist two nonzero integers $p$ and $q$ such that $\alpha^p$ and $\beta^q$ are freely homotopic.
Without loss of generality, suppose $\alpha\cap T\neq\emptyset$.
Let $p_0$ (resp. $q_0$) be the (geometrical) intersection numbers of $\alpha$ 
(resp.   $\beta$) and
$T$. By the assumption  that $\alpha^p$ and $\beta^q$ are freely homotopic and the fact that $T$ is transverse to $Y_t$ everywhere, then $p_0 q_0>0$ and $pp_0= q q_0$ (named $m$). Then, without loss of generality, we can  assume that each of $p_0$, $q_0$, $p$ and $q$ is larger than $0$.

Now we build a $m$-cyclic cover   $W_m$  of $W$ by cutting $T$ and gluing $m$ cutting pieces together. Let $\Wi{Y}_t$ be the lifted Anosov flow on $W_m$ and  $\Wi{T}$ be a lift
of $T$ that is transverse to  $\Wi Y_t$.
Under this cover, $\alpha$ is lifted to $p_0$ periodic orbits $\Wi\alpha_1, \dots, \Wi\alpha_{p_0}$ and $\beta$ is lifted to $q_0$ periodic orbits
$\Wi\beta_1, \dots, \Wi\beta_{q_0}$. We can further assume that $\Wi\alpha_1$ positively intersects to $\Wi T$.
Since $\alpha^p$ is freely homotopic to $\beta^q$, then by homotopy lifting property, $\Wi\alpha_1$ is freely homotopic to  $\Wi\beta_k$ for some $k\in\{1,\dots,q_0\}$. By Lemma \ref{l.loz}, there exists a periodic orbit $\Wi \gamma$ of $\Wi Y_t$ and some
 $i\in\{1,2\}$ and $j\in \{-2,-1\}$ such that $(\Wi\alpha_1)^i$ and $(\Wi \gamma)^j$  are freely homotopic.
 Recall that $\Wi\alpha_1$ positively intersects to $\Wi T$ that is a transverse torus of $\Wi Y_t$, then $\Wi \gamma$  also positively intersects to $\Wi T$, which conflicts to the fact that $(\Wi\alpha_1)^i$ and $(\Wi \gamma)^j$  ($i\in\{1,2\}, j\in \{-2,-1\}$) are freely homotopic. 
Then the conclusion of the proposition follows.
\end{proof}

Now we can  finish the proof of  Lemma \ref{l.hypknot} by using  some  techniques similar to Appendix B of \cite{FH}.

\begin{proof} [Proof of Lemma \ref{l.hypknot}]
First, we show that $W_\Gamma$ is irreducible. Suppose that $S$ is an embedded 2-sphere in $W_\Gamma$. Then $S\subset W$, since $W$ is irreducible, then $S$ bounds a $3$-ball $B_1$ in $W$. Either $\Gamma$ is disjoint with $B_1$, or there exists a sub-periodic orbits set $\Gamma_0\subset \Gamma$ so that $\Gamma_0\subset B_1$. If $\Gamma$ is disjoint with $B_1$, then $S$ bounds the $3$-ball $B_1 \subset W_\Gamma$. If $\Gamma_0$ is in $B_1$, then every periodic orbit $\gamma\subset \Gamma_0$ is contractable in $W$, this is impossible since $\gamma$ is a periodic orbit of an Anosov flow.

Now we show that $W_\Gamma$ is atoroidal.
 Let $T$ be an incompressible torus in $W_\Gamma$. Note that $\Gamma$ geometrically
intersects  every incompressible torus in $W$, therefore $T$ should be compressible in $W$.
Since $T$ is compressible in $W$ and $W$ is an irreducible $3$-manifold,
either $T$ is in a $3$-ball $B$ of $W$ or $T$ bounds a solid torus $V$ in $W$ which contains  $\Gamma_0$, that is a sub-periodic orbits set of $\Gamma$. 

Now we prove that the first case can not happen.
 Let $\Wi{W}$ be the universal cover of $W$ and $\pi: \Wi{W} \to W$ be the covering map.
Assume by contradiction that $T$ is embedded into a $3$-ball $B$ of $W$, $T$ lifts to a connected component $\Wi{T}$ in $\Wi{W}$ so that $\pi |_{\Wi{T}}: \Wi{T} \to T$ is a homeomorphism map. Since $T$ is incompressible in $W_\Gamma$, $T$ is incompressible in $W\setminus\Gamma$ and $\Wi{T}$ is incompressible
in  $\Wi{W} \setminus \Wi{\Gamma}$ that is a covering space of $W\setminus\Gamma$. Here $\Wi{\Gamma}$
is the  preimage of $\Gamma$ under $\pi$. This is impossible due to the following reasons.  Note that $\Wi W $ is a principal $\mathbb{R}$-bundle over $O$ by the orbits of $\Wi Y_t$ (\cite{Fen1}), where $O$ is the orbit space of $\Wi Y_t$ which is homeomorphic
to $\mathbb{R}^2$ and $\Wi{Y_t}$ is the lifting flow of $Y_t$. Therefore, $\Wi{W} \setminus \Wi{\Gamma}$
is homeomorphic to $(\mathbb{R}^2 - G) \times \RR$ for some infinite discrete set $G \subset \mathbb{R}^2$. So $\pi_1 (\Wi{W} \setminus \Wi{\Gamma})$ is free, which
does not contain  any $\mathbb{Z}\oplus \mathbb{Z}$ subgroup. But  $\Wi{T}$ is incompressible
in  $\Wi{W} \setminus \Wi{\Gamma}$, then $\pi_1 (\Wi{W} \setminus \Wi{\Gamma})$ should contain a
$\mathbb{Z}\oplus \mathbb{Z}$ subgroup. We get a contradiction, and thus $W_\Gamma$ must be atoroidal.

In the second case, let $V$ be an embedded solid torus in $W$ so that $\Gamma_0= \Gamma \cap V$ is a sub-link of  $\Gamma$ that is in the interior of $V$. We claim that 
$\Gamma_0$ only contains a unique periodic orbit $\gamma_1$  that is a core of $V$. Due to this claim, we have that $T$ is parallel to a boundary component of $W_\Gamma$, then
$W_\Gamma$ is atoroidal.

 Now we prove the claim.
Assume that
there exist two periodic orbits $\gamma_1$ and $\gamma_2$ in $\Gamma_0$. Since $\pi_1 (V)\cong \ZZ$ which can be generated by a core of $V$, and every nonzero power of $\gamma_1$ or $\gamma_2$ is homotopy nonvanishing,
then there exist two nonzero integers $p_1$ and $p_2$ such that $\gamma_1^{p_1}$
and $\gamma_2^{p_2}$ are freely homotopic. But there is a transverse torus $T_{\gamma_1}$ which intersects  $\gamma_1$. This contradicts the conclusion of  Proposition \ref{p.homotopyclass},
 and so from now on  we can assume that $\Gamma_0=\{\gamma_1\}$.
$\gamma_1$  is certainly freely homotopic to $c(V)^z$ for some $z\in \ZZ\setminus \{0\}$, where $c(V)$ is an oriented core of $V$.
This means that $\hbox{Int}(\gamma_1, T_{\gamma_1}) =z\cdot \hbox{Int}(c(V), T_{\gamma_1})$ where
$\hbox{Int}(c, T_{\gamma_1})$ is the algebraic intersection number of a  closed curve $c$ and $T_{\gamma_1}$ in $W$. 

For simplicity, let $\gamma=\gamma_1$, $k=\hbox{Int}(c(V), T_{\gamma_1})$ and $p=zk$, and without loss of generality assume 
that both of $z$ and  $k$ are positive. Similar to the proof of Proposition \ref{p.homotopyclass},
we build a $p$-cyclic cover $W_p$ of $W$ by cutting $W$ along $T_{\gamma}$ and gluing $p$ cutting pieces together. Let $\Wi Y_t$ be the lifting Anosov flow on $W_p$ and $\Wi T_{\gamma}$ be a lift
of $T_{\gamma}$ that is transverse $\Wi Y_t$. Under this cover, $\gamma$ is lifted to $p$ periodic 
orbits $\Wi \gamma^1, \dots, \Wi \gamma^p$. By construction, we can assume 
$\Wi \gamma^1, \dots, \Wi \gamma^z$ are in a lifting connected component $\Wi V$ of $V$.
Certainly $\Wi V$ is also homeomorphic to a solid torus, therefore similar to the above discussion,
every two of $\Wi \gamma^1, \dots, \Wi \gamma^z$ are freely homotopic up to finite power.
By Proposition \ref{p.homotopyclass} once more, we have that $z=1$ and 
therefore $\gamma$ and $c(V)$ are freely homotopic in $V$.
Now we are left to prove that they are isotopic. Due to  \cite{Fen1} and  following an easy computation of fundamental group, $\Wi{\gamma}$ is not knotted in $\Wi W$ and thus is not knotted in $\Wi V$. Therefore, $\gamma$ is not knotted in $V$, and then $c(V)$ and $\gamma$ are isotopic in $V$. The claim is proved.

Next, we will prove that each boundary torus $T$ of $W_\Gamma$ is incompressible in $W_\Gamma$.
We assume by contradiction that there exists a boundary torus $T$ of $W_\Gamma$ is compressible in $W_\Gamma$.
By Dehn Lemma, there exists an essential simple closed curve $c$ in $T$ so that $c$ bounds an embedded disk $D_0$ in $W_\Gamma$. Then there is a two-handle $U$ parameterized by $D\times [-1,1]$ in $W_\Gamma$ so that,  $D_0 =D\times {0}$, $\partial D \times [-1,1] \subset T$. Then $U(\gamma)\cup U$ is homeomorphic to $W_0-B$, where $B$ is an open 3-ball and $W_0$ is a closed 3-manifold with Heegaard genus  no larger than $1$.
Therefore, $W_0$ is homeomorphic to one of $\SS^3$, a lens space and $\SS^1 \times \SS^2$.
The first case can not happen since it implies that some periodic orbit $\gamma_i\in \Gamma$ $(i=1,\dots, p)$ is in the interior of a 3-ball.
Both of the left two cases also can not happen due to the following reasons. It is easy to observe that  $W_0$ is  a prime decomposition element of $W$. But notice that $W$ is irreducible, then $W$ is prime and therefore $W$ is homeomorphic to $W_0$. This means that $W$ is homeomorphic to either a lens space or $\SS^1 \times \SS^2$. But both  cases can not happen since as an underling manifold of an Anosov flow, $W$ must be irreducible and $\Wi W$ (the universal cover of $W$) is homeomorphic to $\RR^3$.

Now, we prove that  $W_\Gamma$ is anannular, i.e. every incompressible annulus with boundary neatly intersecting  $\partial W_\Gamma$ is parallel to  $\partial W_\Gamma$ relative to its boundary. Let $A$ be such an embedded incompressible annulus with $\partial A =a_1 \cup a_2$, where $a_1$ and $a_2$ are two circles in $\partial W_\Gamma$. 
Then naturally there are two cases: $a_1$ and $a_2$ are in two tori $T_{\gamma_1}$ and $T_{\gamma_2}$ of $\partial W_\Gamma$,
and $a_1$ and $a_2$ are in the same torus $T_\gamma$ of $\partial W_\Gamma$. Here $\gamma_1,\gamma_2,\gamma \in\Gamma$ and $T_{\gamma_1}$, $T_{\gamma_2}$, $T_{\gamma}$ are the corresponding boundary tori of  $\partial W_\Gamma$.

In the frist case, let $T$ be the inside boundary  of a small tubular neighborhood of the two complex $T_{\gamma_1}\cup A\cup T_{\gamma_2}$. Obviously, $T$ is a torus and moreover $T$
bounds $U_0\cong\Sigma_0 \times S^1$ on one side of $W_\Gamma \setminus T$ where  $\Sigma_0$ is a pants. Then $T$ is not parallel to a boundary torus of $W_\Gamma$. Since $W_\Gamma$ is atoroidal, $T$ is compressible in $W_\Gamma$. Due to the fact that each boundary torus of $W_\Gamma$ is incompressible in $W_\Gamma$, and by a similar discussion as before, one can get that 
$T$ bounds a solid torus $V_T$ on the other side of $W_\Gamma \setminus T$. This is certainly not true since $\partial W_\Gamma$ contains at least $4$ boundary connected components.  Therefore, this case never happens.

In the second case, let $T^1$ and $T^2$ be the inside boundary  of a small tubular neighborhood $N$  of the two complex $T_{\gamma}\cup A$. Similar to the above discussion,
$T^i$ ($i=1,2$)  either bounds a solid torus in $W_\Gamma$ or is parallel to
$T_{\Gamma}$.   Since $\partial W_\Gamma$ contains at least $4$ boundary connected components and $\partial N = T^1 \cup T^2 \cup T_\Gamma$, one can concluse that $T^1$ is parallel to $T_\Gamma$, and $T^2$ bounds a solid torus $V_{T^2}$ in $T_\Gamma$.
Moreover, $T^1$ and $T_\Gamma$ bound the thickened torus $U$ that is the union of
$N$ and $V_{T^2}$. This implies that each $a_i$ ($i=1,2$) is isotopic to the core of $V_{T^2}$ in $U$, and therefore $A$ is  parallel to  $\partial W_\Gamma$ relative to its boundary. Then $W_\Gamma$ is anannular.

In summary, $W_\Gamma$ is an atoroidal, anannular  and irreducible $3$-manifold with a finite number of incompressible tori as the boundary, then $W_\Gamma$ is anannular, atoroidal and Haken.   Then by Thurston's Hyperbolization Theorem \cite{Thu}, $W_\Gamma$ is a hyperbolic $3$-manifold, i.e. the interior of $W_\Gamma$ admits a complete hyperbolic geometry with finite volume. The sufficiency part is proved, and  then the proof of the lemma is complete.
\end{proof}

Now we can finish the proof of Proposition \ref{p.hyperbolic}.
 \begin{proof}[Proof of Proposition \ref{p.hyperbolic}]
 Consider the manifold $W^m$, the Anosov flow $Y^m_t$, and the link $\Gamma\subset W^m$ consisting of the periodic orbits $\alpha_1^m,\dots,\alpha^m_{2n}$ of $Y^m_t$.  It follows from our construction that (see Theorem~\ref{t.fhplug} and~\ref{t.con}): 
 \begin{itemize}
 \item[--] $W^m$ is a toroidal $3$-manifold with two hyperbolic JSJ pieces $U^-,U^+$ separated by $4n$ JSJ tori $T_1,\dots,T_{4n}$, 
 \item[--] the JSJ tori $T_1,\dots,T_{4n}$ are transverse to the orbits of the flow $Y^m_t$,
 \item[--] the periodic orbit $\alpha^m_j$ intersects each of the JSJ tori $T_{2j-1}$ and $T_{2j}$. 
 \end{itemize}
Hence, $\Gamma$ is full-filling. Further recall that every periodic orbit in $\Gamma$ intersects a transverse torus of $Y^m_t$, then by Lemma \ref{l.hypknot}, $W_{\Gamma}$ is hyperbolic.  Now recall that the manifold $M^m$ is obtained from $W^m$ by performing $k$-Dehn surgeries on the orbits $\alpha^m_1,\dots,\alpha^m_{2n}$. Hence, Thurston's hyperbolic Dehn-filling theorem implies that, when $|k|$ large enough, the manifold $M^m$ is  hyperbolic.
 \end{proof}
 
 In the sequel, we will assume that the integer $k$ was chosen in such a way that the manifolds $M^0,\dots,M^{2n}$ are hyperbolic.

\begin{proposition}\label{p.non-R-covered}
 For every $m=0,\dots, 2n$, the Anosov flow $Z_t^m$ is non-$\RR$-covered. 
\end{proposition} 

The proof of this result relies on the following lemma:

\begin{lemma}\label{l.nonR}
Let $Z_t$ be an Anosov flow on a closed three manifold $M$, and denote by $\cF^s$ the weak stable foliation of $Z_t$. Let $S$ be a surface embedded in $M$, transverse to the orbits of $Z_t$, and consider the one-dimensional foliation $f^s:=\cF^s\cap S$. Assume that there exists a $f^s$-invariant annulus $A\subset S$ so that $f^s_{|A}$ is a Reeb annulus, and so that the two boundary components of $A$ are included in two different leaves of $\cF^s$. Then the Anosov flow $Z_t$ is non-$\RR$-covered.
\end{lemma}

\begin{proof}
Let $\Wi M$ be the universal cover of $M$, and denote by $\Wi \cF^s$ the lift foliation of $\cF^s$ in $\Wi M$. Let $\Wi S$ be a connected component of the lift of $S$, and denote by $\Wi f^s$ the lift of $f^s$ in $\Wi S$. Let $\Wi A$ be a connected component of the lift of $A$ in $\Wi S$. Denote by $c_1,c_2$ be the boundary leaves of the Reeb annulus $A$. Since $A$ is a Reeb annulus, the honolomy of the leaf containing $c_i$ ($i=1,2$) in $\cF^s$ is nontrivial. It follows that $\Wi A$ is homeomorphic to $[0,1]\times \RR$. Denote by $\Wi c_1$ and $\Wi c_2$ the boundary components of $\Wi A$ (so that $\Wi c_1$ is a lift of $c_1$ and $\Wi c_2$ is a lift of $c_2$). For $i=1,2$, denote by $L_i$ the leaf of $\cF^s$ containing $c_i$, and by $\Wi L_i$ the lift of $L_i$ containing $\Wi c_i$. Recall that $L_1\neq L_2$ by assumption. We pick a leaf $L$ of $\cF^s$ which does not contain any periodic orbit, so that $L\cap A \neq \emptyset$. Observe that $L\neq L_1,L_2$ since $L_1$ and $L_2$ both contain periodic orbits. Hence $L_1,L_2,L$ are three pairwise different leaves. Let $\Wi L$ be a  connected component of the lift of $L$ so that $\Wi L \cap \Wi A \neq\emptyset$, and $\Wi c$ be a connected component of $\Wi L \cap\Wi A$.

We claim that $\Wi c$ (respectively $\Wi c_1$ and $\Wi c_2$) is the unique leaf of $\Wi f^s$ contained in $\Wi L$ (respectively in $\Wi l_1$ and  $\Wi l_2$). We will only treat the case of $\Wi c$, the arguments being similar for $\Wi c_1$ and $\Wi c_2$. Assume by contradiction that there exists a leaf $\Wi c '$ of $\Wi f^s$ such that $\Wi c', \Wi c\subset\Wi L$. On the one hand, since $\Wi f^s$ is a Reeb foliation on the band $\Wi A$, one can find an arc $a$ in $\Wi A$, transverse to  $\Wi f^s$, and so that the two ends $p'$ and $p$ of $a$ are in $\Wi c'$ and $\Wi c$ respectively. On the other hand, since $\Wi c'$ and $\Wi c$ are in the same leaf $\Wi L$ of  $\Wi \cF^s$, one can find an arc $b$ joining $p$ to $p'$ in $\Wi L$. Using standard techniques, the loop obtained by concatenating $a$ and $b$ can be perturbed to loop transverse
to $\Wi \cF^s$ (note that $\Wi \cF^s$ is transversally orientable since $\Wi M$ is simply connected). By Novikov theorem, such a transverse loop cannot be homotopic to a point, which contradict the fact that $\Wi M$ is simple connected. Therefore $\Wi c$ must be the unique leaf of $\Wi f^s$ as a subset of $\Wi l$, and the claim is proved.

Since $\Wi f^s$ is a Reeb foliation on $\Wi A$, each of the three leaves $\Wi c_1$, $\Wi c_2$ and $\Wi c$ does not separate the other two from each other in $\Wi A$. Using the claim proved in the previous paragraph, we deduce that each of the three leaves $\Wi L_1$, $\Wi L_2$ and $\Wi L$ does not separate the other two from each other in $\Wi M$. This shows that the Anosov flow $Z_t$ is non-$\RR$-covered.
\end{proof} 

\begin{proof}[Proof of Proposition \ref{p.non-R-covered}]
Fix $m\in\{0,\dots,2n\}$. According to our construction, the one-dimensional foliation $f^{m,s}_i$ induced by the stable foliation of the flow $Y^m_t$ on the two-torus $T_i$ is made of $2i+2$ cyclically ordered Reeb annuli. Moreover, the boundary leaves of these annuli are included in stable manifolds of pairwise different periodic orbits. The periodic orbits $\alpha^m_1,\dots,\alpha^m_{2n}$ intersect the torus $T_i$ at a single point $q_i^m$. It follows that most of the above-mentioned Reeb annuli are not destroyed by the Dehn-Fried surgeries. More precisely, the flows $Y^m_t$ and $Z^m_t$ are orbitally equivalent on the complement of $\alpha^m_1\cup\dots\cup\alpha^m_{2n}$. The image of $T_i\setminus\{q_i^m\}$ under the orbital equivalence is a punctured torus transverse to the orbits of the flow $Z^m_t$. The image of $f^{m,s}_i$ under the orbital equivalence is the one-dimensional  foliation induced by the stable foliation of $Z^m_t$ on the punctured torus. It contains $2i+1$ Reeb annuli. The boundary leaves of these annuli are included in stable manifolds of pairwise different periodic orbits. Hence, the flow $Z^m_t$ satisfies the hypothesis of Lemma \ref{l.nonR}. Hence $Z_t^m$ is a non-$\RR$-covered Anosov flow. 
 \end{proof}

%
%

According to Propositions~\ref{p.homeomorphic-manifolds}, \ref{p.hyperbolic} and ~\ref{p.non-R-covered}, assuming $|k|$ is large enough and up to some orbital equivalences, the flows $Z^0_t,\dots,Z^{2n}_t$ can be seen as $2n+1$ non-$\RR$-covered Anosov flows on the same hyperbolic $3$-manifold $M\cong M^0\cong\dots\cong M^{2n}$. This completes the construction part of the paper. What remains to be done is to prove that the Anosov flows $Z^1_t,\dots,Z^{2n-1}_t$ are pairwise orbitally inequivalent.

%% file: Lozenges.tex
From now on, we will distinguish the orbital equivalence classes of $\{Z_t^m\}$.
\subsection{Orbit space}\label{ss.Osp} 
In this subsection, we introduce some conceptions and facts about orbit space theory mainly built by
Fenley and Barbot (\cite{Fen1}, \cite{Fen2}, \cite{Ba}, etc). 
The details we refer to Section $2$ of \cite{BaFe}.

Let $\Xi_t$ be an Anosov flow on  a closed orientable 3-manifold $P$, $\Wi P$ be the universal cover of $P$ and $\Wi{\Xi}_t$ be the corresponding lift flow on $\Wi P$.
 Let $\cO$ be the quotient space of $\Wi P$ by collapsing every  orbit of $\Wi{\Xi}_t$ to a point, and  call $\cO$ the \emph{orbit space} of the Anosov flow $\Xi_t$.
The space $\cO$ with its induced topology is homeomorphic
to $\RR^2$.

For any point (or set) $x\in M$, we  denote by $W^s(x)$ and $W^u(x)$  the weak
stable and weak unstable leaf through $x$, and denote by $\Wi{W}^s(x)$ and  $\Wi{W}^u(x)$ 
the  corresponding lifts on $\Wi M$.

Let $\alpha$ and $\beta$ be  two points in $\cO$ (equivalently, two orbits of $\Wi{\Xi}_t$) and let  $A\subset \Wi W^s(\alpha), B\subset \Wi W^u(\alpha), C\subset \Wi W^u(\beta)$ and  $D\subset \Wi W^u(\beta)$ be four half leaves satisfying:
\begin{enumerate}
\item for every weak stable leaf $l^s$, $l^s \cap B \neq \emptyset$ if and only if $l^s \cap D \neq \emptyset$;
\item for every weak unstable leaf $l^u$, $l^u \cap A \neq \emptyset$ if and only if $l^u \cap C \neq \emptyset$;
\item the half-leaf $A$ does not intersect $D$ and $B$ does not intersect $C$. 
\end{enumerate}
A \emph{lozenge} $L$ with corners $\alpha$ and $\beta$ is the open subset of $\cO$ given by (see
Figure \ref{f.lozenge}):
$L = \{p \in \cO\mid \Wi W^s(p) \cap B \neq \emptyset, \Wi W^u (p) \cap A\neq \emptyset\}$.
The half-leaves $A, B, C$ and $D$ are called the \emph{edges}.

 \begin{figure}[htp]
\begin{center}
  \includegraphics[totalheight=2.8cm]{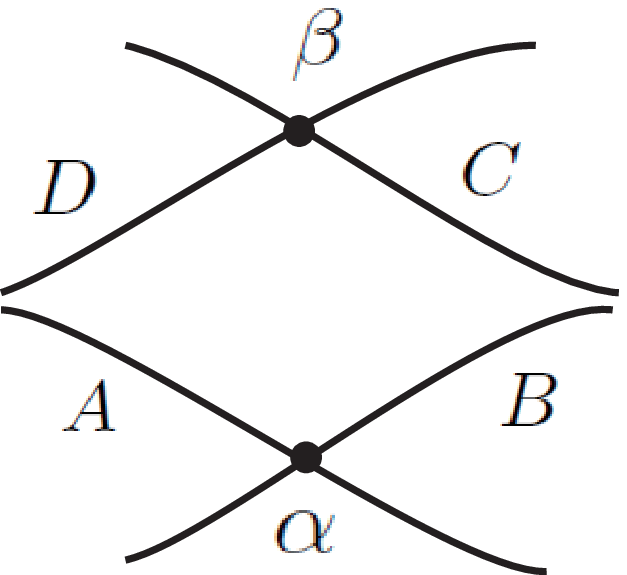}\\
  \caption{Lozenge}\label{f.lozenge}
\end{center}
\end{figure}

We say that two lozenges are \emph{adjacent} if they either share an edge (and a
corner), or they share a corner. In particular, if two lozenges are adjacent by sharing an edge, we say that
they are \emph{edge-adjacent}.

A \emph{chain of lozenges} $\cC$ is a finite or countable collection
of lozenges such that, for any two lozenges $L, L'\in \cC$, there exists a finite
sequence of lozenges $L_0, . . . ,L_n \in \cC$ such that $L=L_0, L'=L_n$ and for any
$i=0,1,\dots, n-1$, $L_i$ and $L_{i+1}$ are adjacent.  See Figure \ref{f.clozenge}.

 \begin{figure}[htp]
\begin{center}
  \includegraphics[totalheight=3.6cm]{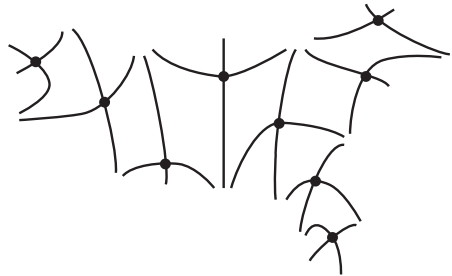}\\
  \caption{Chain of lozenges}\label{f.clozenge}
\end{center}
\end{figure}

Lozenges are important to the study of freely homotopic orbits. Because Barbot and 
Fenley  \cite{BarFe2} proved that if a periodic orbit is freely homotopic to another
periodic orbit or its inverse, then there exists two coherent lifts in
$\Wi P$ that are two corners of a chain of lozenges (see Section $2$ of \cite{BaFe}). Conversely,
if there exists two lifts $\Wi \alpha$ and $\Wi \beta$ in $\Wi P$  that are two corners
of a chain of lozenges, then, up to switching $\alpha$ and $\beta$, the orbit $\alpha$ is freely
homotopic to $\beta^{\pm1}$ or $\beta^{\pm2}$.
Let $\Wi \alpha$ be the lift of a periodic orbit, then  there exists
a uniquely defined maximal chain of lozenges which contains $\Wi \alpha$ as a corner.
We call this maximal chain $\Wi{FH}(\Wi \alpha)$ . By Fenley’s result,
the union of the corners in $\Wi{FH}(\Wi \alpha)$ contains (a coherent lift of) the complete
free homotopy class of $\alpha$, as well as all the orbits freely homotopic to the
inverse of $\alpha$.

Now we recall a fundamental relationship between fundamental Birkhorff annuli and Lozenges discovered by Barbot \cite{Ba}.
\begin{lemma}\label{l.photoB}
Let $\Xi_t$ be an Anosov flow 
with co-orientable stable foliation \footnote{If the stable foliation is not co-orientable, $L$ maybe is associated to a Klein bottle. See Barbot \cite{Ba}.}
on a closed orientable $3$-manifold $P$, and $L$ be a lozenge of 
the orbit space $\cO$ of $(P,\Xi_t)$ such that the two vertices of $L$ are associated to two periodic orbits
of $\Xi_t$. Then there is a fundamental Birkhoff annulus $A$ of $\Xi_t$ so that $A$ is associated to $L$ in the sense that  $L$ is the projection in $\cO$ of a lift $\Wi A \subset \Wi P$ of $A$.
\end{lemma}

\subsection{Chain of lozenges of $Y_t^m$}\label{ss.clytm} 
 It is obvious that  for every $\gamma_i^{j,\pm}\in \Gamma$, there are $(2i+2)$ periodic orbits
$\gamma_i^{1,\pm}, \dots, $ $\gamma_i^{2i+2,\pm}$ that are freely homotopic to
$\gamma_i^{j,\pm}$, and there are $(2i+2)$ periodic orbits
$\gamma_i^{1,\mp}, \dots, \gamma_i^{2i+2,\mp}$ that are freely homotopic to
$(\gamma_i^{j,\pm})^{-1}$. The definitions of the related notations see item (6') of Theorem \ref{t.fhplug}. We will show that the periodic orbits in $\Gamma$
are special in the following sense:

\begin{lemma}\label{l.Whomotopy}
 Let $\omega$ be a periodic orbit of $Y_t^m$ which is not contained in $\Gamma$, then there does not exist any other periodic orbit $\omega'$ of $Y_t^m$ such that either $\omega'$ or $(\omega')^{-1}$ is freely homotopic to $\omega$.
\end{lemma}
\begin{proof}
We assume by contradiction that there exists a periodic orbit
$\omega$  of $Y_t^m$ which is not contained in $\Gamma$  and  a  periodic orbit $\omega'$ of $Y_t^m$ such that either $\omega'$ or $(\omega')^{-1}$ is freely homotopic to $\omega$. 
Since there are $4n$ transverse tori $T_1, \dots, T_{4n}$, by Proposition \ref{p.homotopyclass}, both of $\omega$ and $\omega'$ are disjoint to $\cup_1^{4n}T_i$. Without loss of generality, suppose $\omega$ is a periodic orbit of $(U^+, X_t)$.

Let  $g: A=S^1 \times [0,1] \to W$ be a continuous map  such that:
\begin{enumerate}
\item$g(S^1 \times {0}) = \omega$ and $g(S^1 \times 1)=\omega'$;
\item $g(A)$ is a two complex in $W$ which is clean when close to $\omega$ and $\omega'$;
\item $g(A)$ and $T_i$ are in general position.
\end{enumerate}
 By the generality,  $g$ locally is a homeomorphism. Therefore, $g^{-1} (T_1 \cup \dots \cup T_{4n})$  is the union of finitely many pairwise disjoint circles $c_1, \dots, c_m$ in the interior of $A$. 

We further assume that $m>0$.
One can observe that each $c_j$ ($j\in \{1, \dots, m\}$) is homotopy non-vanishing. Otherwise, let $c_{j_0}$ be an inner most circle of $\{c_1, \dots, c_m\}$.
Then $g(c_{j_0})$ is a homotopy vanishing loop (maybe not embedded) in some $T_i$, which means that $g(c_{j_0})$ bounds an immered disk $h: D^2 \to T_i$.  Note that $g(c_{j_0})$ also bounds an immersed disk $g(D_{j_0})$ ($D_{j_0}$ is the disk bounded by $c_{j_0}$ in $A$) in either $U^+ \cap g(A)$ or $U^- \cap g(A)$.
Since $W$ admits an Anosov flow and therefore admits a taut foliation, then 
$W$ is irreducible, i.e. $\pi_2 (W)=0$. Further notice that $g(c_{j_0})= g(\partial D_{j_0}) = g(h (\partial D^2))$, then we can cancel $c_{j_0}$ for $g$ by rebuilding $g$ following the homotopy with a small pertubation.

From now on, we can suppose each $c_j$ is essential in  $A$.
Since $A$ is an annulus,  $S^1 \times {0}, c_1, \dots c_m, S^1 \times {1}$ are pairwise parallel and disjoint. Suppose  $S^1 \times {0}$ and $c_1$ are adjacent. Then without loss of generality, we can suppose $g( S^1 \times {0})=\omega$ and $g(c_1)\subset T_1$ are homotopic in $U^+$. But by item
7 of Theorem \ref{t.fhplug}, there does not exist any closed curve $c$ in $T_1$ such that $\omega$ and $c$ are freely homotopic in $U^+$. We get a contradiction.

The unique left case is that $m=0$, in this case $\omega$ is  freely homotopic to either  $\omega'$ or $(\omega')^{-1}$ in $U^+$, but it also can not happen due to item
7 in Theorem \ref{t.fhplug} once more. The proof of the lemma is complete.
\end{proof}

Due to Lemma \ref{l.Whomotopy}, we have the following consequence:

\begin{corollary}\label{c.BkaGa}
If $\alpha$ is a periodic orbit in the boundary of a fundamental Birkhoff annulus $A$ of $Y_t^m$,
then $\alpha\in \Gamma$. 
\end{corollary}

\begin{remark}\label{r.brper}
Branching periodic orbit is a very important concept in the study of Anosov flows. We refer the reader  \cite{Bart} for more details about branching periodic orbits.
It is easy to see that every periodic orbit of $\Gamma$ is a branching
  periodic orbit of the Anosov flow $Y_t^m$.  Moreover,  notice that if $\beta$ is a branching periodic orbit of $Y_t^m$, then there exists another periodic orbit $\beta'$ of $Y_t^m$ such that $\beta$ and $\beta'$ are freely homotopic. Therefore, due to Lemma
\ref{l.Whomotopy},
the set of the branching periodic orbits of $Y_t^m$ is exactly $\Gamma$.
\end{remark}

\begin{lemma}\label{l.Tip}
For every $T_i$ ($i=1,\dots,4n$), there exists a torus $T_i'$ of $Y_t^m$ associated to $T_i$ such that:
\begin{enumerate}
\item $T_i'$ is isotopic to $T_i$. 
\item $T_i'$ consists of $(4i+4)$ cyclically ordered fundamental Birkhoff annuli $A_1, \dots, A_{4i+4}$ with $(4i+4)$-cyclically ordered periodic orbits $\gamma_i^{j,+}, \gamma_i^{j,-}$ ($j\in \ZZ / (2i+2)\ZZ$).
\item For every two fundamental Birkhoff annuli in $T_i'$, $Y_t^m$ transversely intersects their interiors towards the same side of $T_i'$.
\item Every flowline of $Y_t^m$ intersecting with a fundamental Birkhoff annulus with a boundary periodic orbit $\gamma_i^{j,\pm}$ (take $\gamma_i^{j,+}$ if $i$ is even and  $\gamma_i^{j,-}$ if $i$ is odd) is negatively asymptotic to  $W^s(\gamma_i^{j,\pm})$ and every flowline of $Y_t^m$ intersecting with a fundamental Birkhoff annulus with a boundary periodic orbit $\gamma_i^{j,\mp}$ (take $\gamma_i^{j,-}$ if $i$ is even and $\gamma_i^{j,+}$ if $i$ is odd) is positively asymptotic to  $W^u(\gamma_i^{j,\mp})$. 
\end{enumerate}
\end{lemma}

 \begin{figure}[htp]
\begin{center}
  \includegraphics[totalheight=4.2cm]{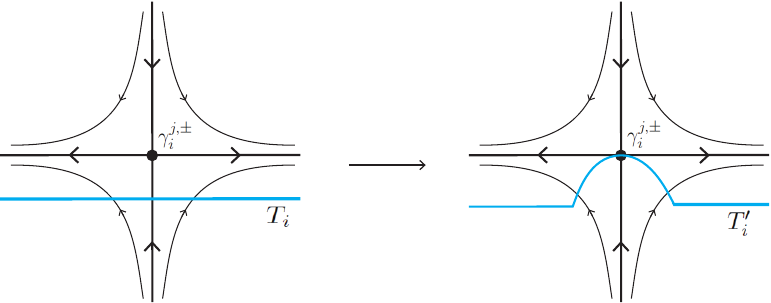}\\
  \caption{Isotopie $T_i$ to $T_i'$}\label{f.Tinew}
\end{center}
\end{figure}

\begin{proof}
For simplicity, we only prove the lemma in the case $i$ is even. The
case $i$ is odd can be similarly proved.

Take a small tubular neighborhood $U_j^s$ of $c_i^{s,j}$
and $U_j^u$ of $c_i^{u,j}$ in $T_i$, and then smoothly isotopies $T_i$
to $T_i'$ such that:
\begin{enumerate}
\item $T_i - \cup_j (U_j^s \cup U_j^u)$ is invariant under the isotopy;
\item $U_j^s$ is isotopic to $(U_j^s)'$ and $U_j^u$ is isotopic to $(U_j^u)'$.
\end{enumerate}
This is not difficult to do, since $c_i^{s,j}\subset W^s (\gamma_i^{j,+})$
and $c_i^{u,j}\subset W^u(\gamma_i^{j,-})$. See Figure \ref{f.Tinew} as an illustration.

Item 1, 2, 3 of the lemma can be routinely checked by the construction of $T_i'$.
Item 4 follows from the fact that each of $\gamma_i^{j,+}$ and $\gamma_i^{j,-}$
is a saddle periodic orbit.  
\end{proof}

\begin{lemma}\label{l.clzY}
For every  torus $T_i'$ ($i=1,\dots, 4n$), 
there exists a chain of lozenges $C$ in the orbit space $\cO_W$ of 
$Y_t^m$ such that:
\begin{enumerate}
\item $C$ consists of infinitely many lozenges $\{L_k \mid k\in \ZZ\}$ such that,
\begin{enumerate}
\item $L_k$ and $L_{k+1}$ are edge-adjacent along a stable separatrix if
$k$ is odd and an unstable separatrix if $k$ is even;
\item every three lozenges $L_{k_1}$, $L_{k_2}$ and $L_{k_3}$ of $C$ never share a common vertex. 
\end{enumerate}
\item $L_k$ corresponds to the orbits crossing $A_k$.
\item There exists an $\alpha \in \pi_1 (W)$ which can be represented by an oriented circle in $T_i'$ that transversely intersects $c_i^{1,j}$ once such that $\alpha \cdot L_k = L_{k+(4i+4)}$.
\end{enumerate}
See Figure \ref{f.infL} as an illustration. 
\end{lemma}

 \begin{figure}[htp]
\begin{center}
  \includegraphics[totalheight=3.6cm]{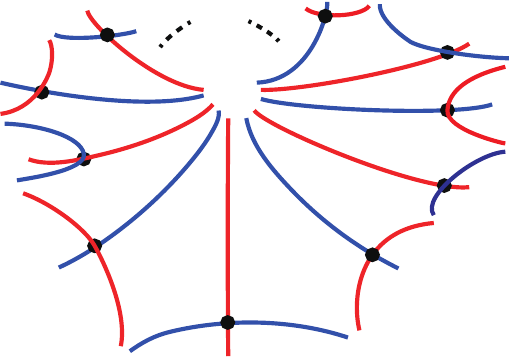}\\
  \caption{The chain of Lozenges $C$}\label{f.infL}
\end{center}
\end{figure}

\begin{proof}
Due to Lemma \ref{l.photoB} and the fact that $A_k\cap A_{k+1}$ is a periodic orbit of $Y_t^m$, we can choose a set of lozenges $\{L_k\}$ ($k=1,\dots, 4i+4$) in $\cO_W$ such  that $L_k$ corresponds to $A_k$ and $L_k$ and $L_{k+1}$ are adjacent. Due to item 4 of Lemma \ref{l.Tip}, $L_k$ and $L_{k+1}$ are edge-adjacent, and moreover up to changing the
order of the subscrips of $\{A_k\}$ if necessary, we can suppose that $L_k \cap L_{k+1}$ is a stable
separatrix if $k$ is odd and an unstable separatrix if $k$ is even. 
Notice that the two periodic orbits associated to $L_{k_1} \cap L_{k_1+1}$ and  $L_{k_2} \cap L_{k_2+1}$  ($k_1, k_2 \in \{1,\dots, 4i+1\}$ and $k_1\neq k_2$) are different.  One can routinely check that $L_1, \dots, L_k$ satisfies item 1 and item 2 of the lemma. 

Further define $L_{n(4i+4)+k}:= \alpha^n \cdot L_k$, then
$\{L_k \mid k\in \ZZ\}$ automatically satisfies item 2 and item 3 of the lemma.
Moreover, one can routinely check that $L_k$ ($k\in \ZZ$) corresponds to 
$A_k$ ($k\in \ZZ / (4i+4)\ZZ$) and $\{L_k \mid k\in \ZZ\}$ satisfies item 1 of the lemma.
\end{proof}

\begin{lemma}\label{r.clzY}
Every chain of lozenges in $\cO_W$ with vertices corresponding
to periodic orbits is the one descirbed in Lemma \ref{l.clzY}. 
\end{lemma}
\begin{proof}
Let $C'$ be a chain of lozenges in $\cO_W$ with vertices corresponding
to periodic orbits. Then  due to  Lemma \ref{l.photoB}, Lemma \ref{l.Whomotopy}
and Lemma \ref{l.clzY},
 we have that  $C'$ contains a chain of lozenges $C$  in Lemma \ref{l.clzY}. 
We assume by contradiction that $C' \neq C$,  then there exists a lozenge $L'$ in $C'\setminus C$ that share a commen vertex
$v_1$ with $C$. We further assume that $v_2$ is the other vertex of $L'$

By Lemma \ref{l.photoB} once more, we have a fundamental Birkhoff annulus $A'$
associated to $L'$ with two boundary periodic orbits $\beta_1$ and $\beta_2$ with respect to $v_1$ and $v_2$ correspondingly. Due to Lemma \ref{l.Whomotopy}, $\beta_1, \beta_2 \in \Gamma$.  Moreover, by using the similar technique with the proof of  Lemma \ref{l.Whomotopy}, we can get that both of $\beta_1$ and $\beta_2$ belong to the same torus $T_i'$ for some $i\in \{1,\dots,4n\}$. By Lemma \ref{l.Tip},  $T_i'$ is isotopic to $T_i$ and $T_1', \dots, T_{4n}'$ are pairwise disjoint. Then up to isotopy, we can  think  the manifolds obtained by cutting $W$ along $\cup T_i'$ are also $U^+$ and $U^-$.
Due to the fact that $U^+$/$U^-$ is a hyperbolic $3$-manifold and by using the similar technique with the proof of  Lemma \ref{l.Whomotopy} once more, we can get that $A'$ is homotopic to an annulus $A$ of $T_i'$ with boundary $\beta_1 \cup \beta_2$ relative to their commen boundary $\beta_1 \cup \beta_2$.

Since $\beta_1$ and $\beta_2$ are two boundary periodic orbits of a fundamental Birkhoff
annulus $A'$, then $\beta_1$ is freely homotopic to $(\beta_2)^{-1}$. In particular, $\beta_1 \neq \beta_2$, therefore without loss of generality, we can suppose that  $A$ is the union of $k$ ($1\leq k \leq 4n-1$) fundamental Birkhoff annuli $A_1, \dots, A_k$. Let $C^k$ be the sub-chain of lozenges of $C$ consisting of $k$ lozenges $L_1,  \dots, L_k$ associated to $A_1, \dots, A_k$ respectively such that $v_1$ is a corner of $L_1$. Let $v_3$ be the vertex of $L_k$ that is a lift of $\beta_2$.
We can conclude the above discussion as follows:
\begin{enumerate}
\item $L'$ is associated to $A'$ and $C^k$ is associated to $A$.
\item $A'$ and $A$ are homotopic relative to their commen boundary $\beta_1 \cup \beta_2.$ 
\item $v_1$ is a corner shared by $L'$ and $C^k$ which is a lift of $\beta_1$,
and $v_2$ and $v_3$ are two lifts of $\beta_2$.
\end{enumerate}
Then by Homotopy Lifting Theorem, we have: $v_2 =v_3$. 
According to the shape of $C$ (see Figure \ref{f.infL}), we know that
$L'$ must be in the same component  of the four components of $\cO_W \setminus (\Wi W^s(v_1) \cup \Wi W^u(v_1))$ with $C^k$. Then one can  immediately obtain that $L'=L_1$ from the definition of lozenge. Then we get a contradiction and therefore the proof of the lemma is complete.
\end{proof}

\subsection{Old and new lozenges of the flow $Z_t^m$}\label{ss.Onl}

Now we discuss the Anosov flow $Z_t^m$ ($|k|\gg 0$) on the hyperbolic $3$-manifold $M$. For simplicity, we replace $\alpha_j^m$ ($j=1,\dots,2n$) by  $\alpha_j$, and  use $\alpha_j$ to code the corresponding periodic orbit both before and after the DF surgery.
A lozenge $L$ in the orbit space $\cO_m$ of $Z_t^m$ is called an \emph{old lozenge} (resp. \emph{new lozenge}) if $L$ is disjoint with (resp. intersects)  the lifts of $\alpha_1, \dots, \alpha_{2n}$. We call it an old lozenge due to the following fact: a fundamental Birkhoff annulus $A$ of $Z_t^m$ corresponding to $L$ can be also ragarded as
 a fundamental Birkhoff annulus of $Y_t^m$.
The following two lemmas tell us that many combinations of new lozenges and old lozenges  are forbidden for  $Z_t^m$.

 \begin{figure}[htp]
\begin{center}
  \includegraphics[totalheight=4cm]{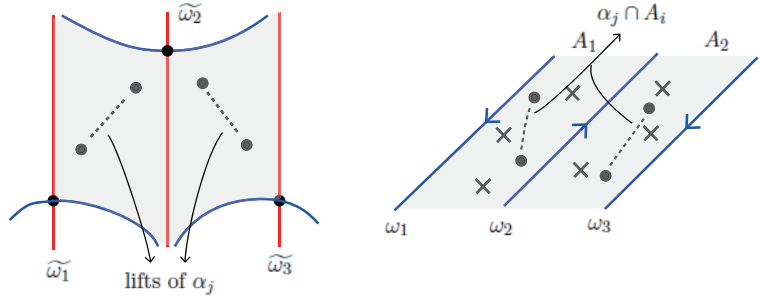}\\
  \caption{Two new adjacent lozenges}\label{f.2newadj}
\end{center}
\end{figure}

\begin{lemma}\label{l.noadjloz}
There do not exist any two new lozenges $L_1$ and $L_2$ of $Z_t^m$ so that,
\begin{enumerate}
  \item each vertex of $L_i$ ($i=1,2$) is a lift of a periodic orbit of $Z_t^m$;
  \item $L_1$ and $L_2$ share a common edge (see the left of Figure \ref{f.2newadj}). 
\end{enumerate}
\end{lemma}

\begin{proof}
Assume by contradiction that  there exist such two new lozenges $L_1$ (with vertices $\widetilde{\omega_1}$ and $\widetilde{\omega_2}$) and $L_2$ (with vertices $\widetilde{\omega_2}$ and $\widetilde{\omega_3}$). By Lemma \ref{l.photoB}, there exist two adjacent
fundamental Birkhoff annuli $A_1$ (with boundary periodic orbits $\omega_1$ and $\omega_2$)  and $A_2$ (with boundary periodic orbits $\omega_2$ and $\omega_3$) such that,
\begin{enumerate}
  \item $A_1$ and $A_2$  are associated to $L_1$ and $L_2$ as Lemma \ref{l.photoB} explains;
  \item the union of $A_1$ and $A_2$ can be isotopically pushed a little to an annulus $A$ which transversely intersects $Z_t^m$.
\end{enumerate}

Since the interior of $L_i$ contains some lifts of $K=\{\alpha_1, \dots, \alpha_{2n}\}$,
 $A_i$ ($i=1,2$) intersects  $K$ finitely many (nonzero) times. Here finiteness is because that we can suppose that the compact surface $A_i$ intersects each $\alpha_j$  generally. Then we can assume that 
$\alpha_j$ intersects  $A_i $ $s_i^j$  times such that $\Sigma_{j} s_i^j >0$. 

In $W$, let $V(\alpha_j)$ be a standard small solid torus neighborhood of $\alpha_j$, recall that  we can choose two  oriented circles $\mu_j,\lambda_j$
on $\partial V(\alpha_j)$, so that $\mu_j$ is a meridian circle and $\lambda_j$ is a longitute circle, i.e. a component of $W^s(\alpha_j) \cap \partial V(\alpha_j)$, and a circle parameterized by $\mu_j+k\lambda_j$ bounds a disk in $M$ \footnote{The meridian circle and the longitute circle were firstly introduced at the beginning of Section \ref{s.conh}. In Section \ref{s.conh}, we use $\mu_j^m$ and $\lambda_j^m$ to represent the   meridian circle and the longitute circle respectively.  But from now on, for the sake of simplicity, we  replace $\mu_j^m$ and $\lambda_j^m$ with $\mu_j$ and $\lambda_j$.}. 

Now we do some homology caculations, which will provide us a contradiction to finish the proof of the lemma.
Due to the annulus $A_1$ in $W$, we can get that in $H_1 (W)$,
$$[\omega_2]=-[\omega_1]-\Sigma_{j} s_1^j [\mu_j+k\lambda_j]= -[\omega_1]-k\Sigma_j s_1^j [\lambda_j]$$ since $[\mu_j]=0$ in $H_1 (W)$.
\footnote{In fact $\omega_i$ ($i=1,2$) maybe is some sugeried periodic orbit $\alpha_j$, but it is easy to observe that
as a boundary component of a Birkhoff annulus, this does not affect the establishment of the above equality. We will also use this observation  several times during the proofs below.}

Due to the annulus $A_2$ in $W$, we can get that in $H_1 (W)$,
$$[\omega_2]=-[\omega_3]+\Sigma_{j} s_2^j [\mu_j+k\lambda_j]= -[\omega_3]+k\Sigma_j s_2^j [\lambda_j].$$

 \begin{figure}[htp]
\begin{center}
  \includegraphics[totalheight=4.5cm]{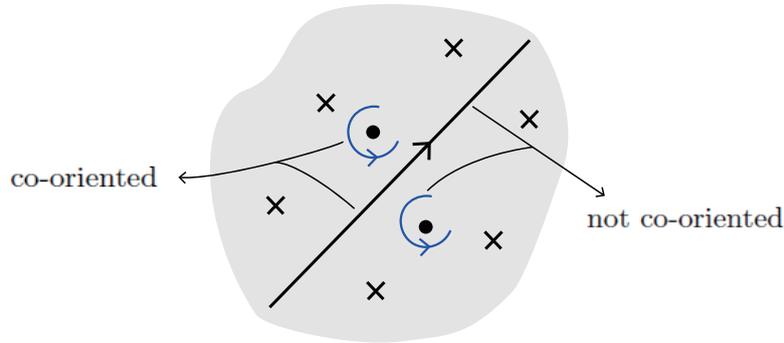}\\
  \caption{Co-orientation and non-co-orientation when $\alpha_j$ crosses $A_i$}\label{f.coandnco}
\end{center}
\end{figure}

The relationships between $A_1$ and $A_2$ can be found in Figure \ref{f.coandnco}. Also notice that
$\lambda_j$ is freely homotopic to $\alpha_j$ in $W^m$. 
We conclude that

$$[\omega_2]=-[\omega_1]-k\Sigma_j s_1^j [\alpha_j]=-[\omega_3]+k\Sigma_j s_2^j [\alpha_j].$$

Without loss of generality,  we assume that $k>0$. Recall that $s_i^j \geq 0$  and $\Sigma_j s_i^j >0$, then,
$$[\omega_2]= -([\omega_1]+k\Sigma_j s_1^j [\alpha_j]),$$ where $k>0$ and $\Sigma_j s_1^j >0$.
We suppose that $s_1^2 >0$, since $Int (\alpha_2, T_2)=1$ and $Int (\alpha_j, T_2)=0$ if $j\neq 2$, then
$$Int (\omega_2, T_2)= - Int (\omega_1, T_2)-k \Sigma_j s_1^j Int (\alpha_j, T_2)= - Int (\omega_1, T_2)-ks_1^2.$$

Further notice that in $(W, Y_t^m)$, $Int (\beta, T_2)\geq 0$ for every periodic orbit $\beta$,
therefore $Int (\omega_2, T_2)\geq 0$ and meanwhile  $Int(\omega_2, T_2)= - Int (\omega_1, T_2)-ks_1^2$. Then we get a contradiction, and the proof of the lemma is complete.
\end{proof}

 \begin{figure}[htp]
\begin{center}
  \includegraphics[totalheight=3.3cm]{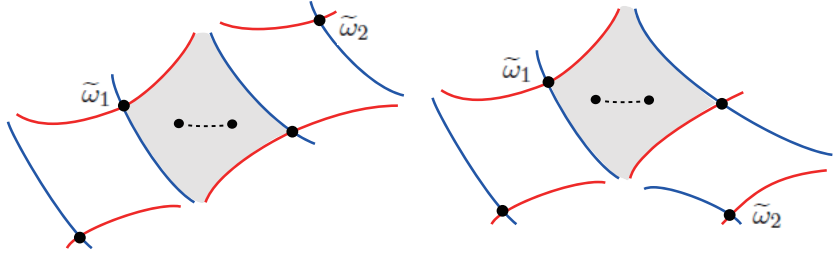}\\
  \caption{An  unrealizable case about two old lozenges}\label{f.lem4case}
\end{center}
\end{figure} 

Many parts in the proof of the following  lemma are similar to the proof of Lemma \ref{l.noadjloz}, so we will only give a brief proof.

\begin{lemma}\label{l.noonenew}
There do not exist two old lozenges $L_1$ and $L_2$ of $Z_t^m$ such that,
\begin{enumerate}
  \item each vertex of $L_i$ is a lift of a periodic orbit of $Z_t^m$;
  \item $L_1$ and $L_2$ are connected by a new lozenge $L_3$ as Figure \ref{f.lem4case} shows.
\end{enumerate}
\end{lemma}
\begin{proof}
For any case of Figure \ref{f.lem4case}, in $H_1(W)$, 
$$[\omega_1]=[\omega_2]\pm \Sigma_j s_j ([\mu_j]+k[\lambda_j])= [\omega_2]\pm k\cdot \Sigma_j s_j [\lambda_j].$$
There exists a $j_0$ such that $s_{j_0}\neq 0$, therefore
 $$Int (\omega_1, T_{2j_0})= Int (\omega_2, T_{2j_0})\pm k\cdot  \Sigma_j s_j Int(\lambda_j, T_{2j_0}).$$

Recall that  $$Int (\lambda_j, T_{2j})=Int (\alpha_j, T_{2j})=1 ~\mbox{and}~ Int(\lambda_i, T_{2j})= Int (\alpha_i, T_{2j})=0$$ for every $i\neq j$,
then $Int (\omega_1, T_{2j_0})= Int (\omega_2, T_{2j_0})\pm ks_{j_0}$.  Further notice that $Int (\omega_1, T_{2j_0})= Int (\omega_2, T_{2j_0})=0$, hence $ks_{j_0}=0$, which induces a contradiction. The proof of the lemma is complete.
\end{proof}

\subsection{Clusters of lozenges}\label{ss.feap}
\begin{definition}\label{d.eap}
We say that a subset $C_e$ of the orbit space $\cO$ of a $3$-dimensional Anosov flow is  a \emph{cluster of lozenges} (or, for simplicity, \emph{cluster}) if $C_e$ is the union of a set of  lozenges such that every lozenge of $C_e$ is edge-adjacent to another lozenge of $C_e$. 
See Figure  \ref{f.B} for two examples of clusters. 
We say a cluster  $C_e$ is  \emph{maximal} if $C_e$ can not be strictly contained in another  cluster, and at this time we call $C_e$ a \emph{maximal cluster of lozenges} (resp. for simplicity, \emph{maximal cluster}). 
\end{definition}

We say that a cluster  $C_e$ is a \emph{fan-type cluster} 
if $C_e$ is a cluster $\{L_k\mid k \in \ZZ \mbox{ or } k\in \{1,\dots, N\}\}$ such that:
\begin{enumerate}
\item for every $k$, $L_k$ and $L_{k+1}$  are edge-adjacent along either a stable separatrix or an unstable separatrix;
\item for every $k$, if  $L_k$ and $L_{k+1}$ are edge-adjacent along a stable (resp. an unstable) separatrix, then  $L_{k+1}$ and $L_{k+2}$  are edge-adjacent along an unstable
(resp. a stable) separatrix;
\item every three lozenges $L_{k_1}$, $L_{k_2}$ and $L_{k_3}$  never share a common vertex. 
\end{enumerate}
For instance, the right part of Figure  \ref{f.B} is a fan-type cluster with $7$ lozenges,  and the chain of lozenges $C$ in Lemma \ref{l.clzY} (see Figure \ref{f.infL}) is   a fan-type cluster with infinitely many lozenges.

 \begin{figure}[htp]
\begin{center}
  \includegraphics[totalheight=4.5cm]{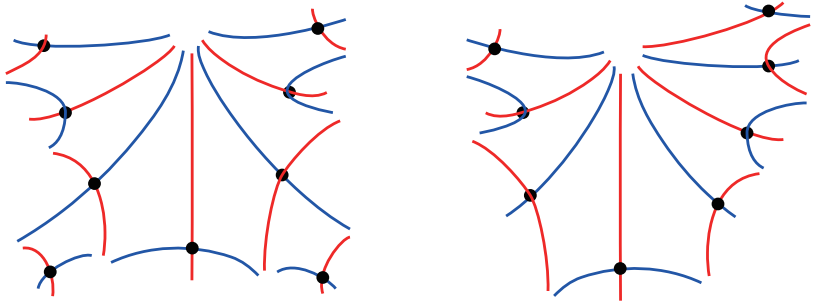}\\
  \caption{Left: a cluster;  Right: a fan-type cluster  with $7$ lozenges}\label{f.B}
\end{center}
\end{figure}

\begin{lemma}\label{l.oldfan}
For every $i\in \{1,\dots, 4n\}$, there exists a chain of lozenges $C$ in the orbit space $\cO_m$ of $Z_t^m$ and a fan-type cluster $C_i\subset C$ such that:
\begin{enumerate}
\item $C_i$ consists of $4i+3$ lozenges;
\item $C_i$ corresponds to the $(4i+3)$ ordered fundamental Birkhoff annuli in $T_i'$
that are disjoint with $\alpha_{[\frac{i+1}{2}]}$. Here $C_i$ is the projection in the orbit space one connected component
of the lift of  the $(4i+3)$ fundamental Birkhoff annuli.
\end{enumerate}
\end{lemma}

\begin{proof}
The periodic orbit  $\alpha_{[\frac{i+1}{2}]}$ intersects $T_i$ at one point disjoint with the compact leaves of $\cL_i^s$ and $\cL_i^u$, see Theorem~\ref{t.con}. Therefore, by Lemma \ref{l.Tip}, $\alpha_{[\frac{i+1}{2}]}$ intersects $T_i'$
at one point inside a fundamental Birkhoff annulus, say $A_{4i+4}$, of $T_i'$.
Define the union of the left $(4i+3)$ fundamental Birkhoff annuli by $\Sigma:=A_1 \cup\dots\cup A_{4i+3}$, that is a compact annulus.
We can also think that $\Sigma$ is in $M$. By using the same techniques with the 
proof of Lemma \ref{l.clzY}, there exists a fan-type cluster $C_i$
that satisfies the conditions in Lemma \ref{l.oldfan}. Let $C$ be the chain of lozenges
that is extended from $C_i$. Then $(C, C_i)$ satisfies the conditions of the lemma.
\end{proof}

 \begin{figure}[htp]
\begin{center}
  \includegraphics[totalheight=4.6cm]{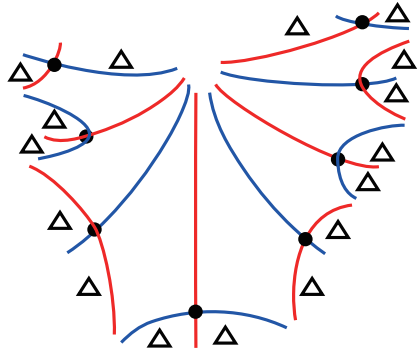}\\
  \caption{$8i+8$ possible places for a new Lozenge edge-adjacent to $C_i$ in the case $i=1$}\label{f.D}
\end{center}
\end{figure} 

\begin{remark}\label{r.oFEAP}
We call a fan-type cluster $C_i$ in Lemma \ref{l.oldfan} an \emph{old fan-type cluster}. Notice that due to its definition, an old fan-type cluster $C_i$ contains $4i+3$ lozenges for some $i=1,\dots,4n$.
\end{remark}

Generally, there are $8i+8$ possible places for a new Lozenge edge-adjacent to
an old fan-type cluster $C_i$, see Figure \ref{f.D}.  We will use this observation to 
 describe the maximal clusters of $(M, Z_t^m)$ as follows:

\begin{proposition}\label{p.EAP}
Let $C$ be a chain of lozenges of the orbit space $\cO_m$ of $(M, Z_t^m)$ which contains some maximal clusters, then each maximal cluster $C_i'$ of $C$ is extended from an old  fan-type cluster $C_i$  
by adding $s$ ($0\leq s\leq 8i+8$) new Lozenges edge-adjacent to $C_i$. Here $C_i$ consists of $4i+3$ lozenges for some $i=1,\dots,4n$. 
\end{proposition}
\begin{proof}
The proof consists of  the following two claims:
\begin{enumerate}
  \item[Claim 1] Let $C$ be a chain of lozenges in the orbit space $\cO_m$ of $(M, Z_t^m)$ so that  $C$ contains some maximal clusters, then every maximal cluster $C_i'$ of $C$ contains  an old fan-type cluster $C_i$ as a  subset.
  \item[Claim 2] Every maximal cluster $C_i'$ is extended from an old  fan-type  cluster $C_i$
by adding $s$ ($0\leq s\leq 8i+8$) new Lozenges edge-adjacent to $C_i$.
\end{enumerate}
It is obvious that the conclusions of  Proposition \ref{p.EAP} can follow from these two claims. From now on, we focus on proving them.

\begin{proof}[The proof of the Claim 1]
We assume by contradiction that there exists a maximal cluster $C_i'$ of some  chain of lozenges $C$ in the orbit space $\cO_m$ such that $C_i'$ does not contain any old fan-type  cluster as a  subset.
Let $L_1$ and $L_2$ be two edge-adjacent lozenges of $C_i'$. By Lemma \ref{l.photoB}, $L_1$ and $L_2$ are associated to two adjacent fundamental Birkhoff annuli $A_1$ and $A_2$ of $Z_t^m$. By  our assumption and the fact that a lift lozenge associated to every old fundamental Birkhoff annulus is contained in  a fan-type maximal cluster\footnote{This follows from Lemma \ref{l.Whomotopy}, Lemma \ref{r.clzY} and the definition of the orbit space $\cO_m$}, each of $A_1$ and  $A_2$ has to be a new fundamental Birkhoff annulus. Then $L_1$ and $L_2$ are two new edge-adjacent lozenges. This is impossible due to Lemma \ref{l.noadjloz}, and therefore Claim 1 is proved.
\end{proof}

\begin{proof}[The proof of the Claim 2]
By Claim 1, every maximal cluster $C_i'$ of a chain of lozenges $C$ containing some maximal clusters must contains  an old fan-type cluster. 

First by Lemma \ref{r.clzY}, we know that every old lozenge of $C_i'$ is contained in an
old fan-type cluster.
Now we prove that $C_i'$ can not contain two old fan-type clusters. Otherwise $C_i'$ contains two old fan-type clusters
$C_i$ and $C_j$. By Lemma \ref{l.noadjloz}, $C_i'$ must contains three lozenges as Figure \ref{f.lem4case} shows. But Lemma \ref{l.noonenew} tells us that this case never happens. Therefore,
 $C_i'$ does not contain two  old fan-type clusters.
 
 Then $C_i'$ can be obtained from extending an old fan-type cluster $C_i$ with some new lozenges. 
 Lemma \ref{l.noadjloz} tells us that a new lozenge $L_0$  must be in the $8i+8$ possible places  edge-adjacent to
$C_i$, as Figure \ref{f.D} shows.  Then Claim 2 can follow from a simple combinatorial discussion depending on these discussions.
\end{proof}
The proof of Proposition \ref{p.EAP} directly follows from Claim 1 and Claim 2. 
\end{proof}

\begin{remark}\label{r.nadd}
Indeed,  not every case    happens   for constructing $C_i'$ from $C_i$ by adding new lozenges  in Proposition \ref{p.EAP}, for instance,
Lemma \ref{l.noadjloz} tells us that it is impossible to add two new lozenges which are edge-adjacent.  
\end{remark}

%% file: Separatrix-adjacent-Birkhoff-annuli.tex
We say that two fundamental Birkhoff annuli $A$ and $B$ are \emph{$s$/$u$-separtrix-adjacent} if $A\cap B$ is a periodic orbit $\beta$ of $Z_t^m$ and $A\cup B$ cuts a small tubular neighborhood $U(\beta)$ of $\beta$ to two connected components
such that one component only contains a local stable/unstable separatrix of
$W_{loc}^s(\beta) \cup W_{loc}^u(\beta)$.

\begin{definition}\label{d.SAa}
Let $\Sigma= A_1 \cup \dots \cup A_k$ ($k\in \NN$) be an immersed annulus in $M$ consisting of $k$ fundamental Birkhoff annuli $A_1,\dots, A_k$ of $Z_t^m$ such that
$A_l \cap A_{l+1}$ ($l=1,\dots, k-1$) is a periodic orbit $\beta_l$ of $Z_t^m$.
We say that $\Sigma$ is a  \emph{$k$-components separatrix-adjacent  annulus}
(abbreviated as \emph{$k$-SA annulus}) if $\Sigma$ satisfies the following conditions:
\begin{enumerate}
\item For every $l$, $A_l$ and $A_{l+1}$ are $s$/$u$-separatrix-adjacent. 
\item If  $A_l$ and $A_{l+1}$ are $s$/$u$-separatrix-adjacent, then
$A_{l+1}$ and $A_{l+2}$ are $u$/$s$-separatrix-adjacent.  
\end{enumerate}
\end{definition}

\begin{example}
Recall that $T_i$ ($i=1,\dots, 4n$)  is a transverse torus of $Y_t^m$ and $T_i'$ is a corresponding torus with $4i+3$ fundamental Birkhoff annuli (Lemma \ref{l.Tip}), and 
the Anosov flow $Z_t^m$ on $M$ is obtained from $Y_t^m$ on $W$ by doing the
$k$-DF surgery at each $\alpha_j^m$ ($j=1,\dots, 2n$). After this surgery, one of the  $4i+4$  fundamental Birkhoff annuli of $T_i'$ is destroyed and the undestroyed  $4i+3$ Birkhoff annuli in $T_i'$ form a  $(4i+3)$-SA annulus of $Z_t^m$, named an \emph{old $(4i+3)$-components separatrix-adjacent  annulus} (abbreviated as an \emph{old $(4i+3)$-SA annulus}) of $Z_t^m$. Certainly, an old $(4i+3)$-SA annulus of $Z_t^m$ is disjoint with the surgeried periodic orbits.  We say that the old  $(4i+3)$-SA annulus $\Sigma =A_1 \cup A_2 \cup \dots \cup A_{4i+3}$ is associated to $T_i$ (in $W$).
\end{example}

By using the same techniques with the 
proof of Lemma \ref{l.clzY} and Lemma \ref{l.oldfan}, we can routinely check the following basic relationship between SA annulus and fan-type cluster:
\begin{lemma}\label{l.photo}
For every    $k$-SA annulus $\Sigma= A_1\cup A_2\cup \dots \cup A_k$ of $Z_t^m$, in the orbit space $\cO$ of $Z_t^m$, there is a fan-type cluster $C_e$ that is the union of $k$ lozenges $L_1, L_2, \dots, L_k$ such that $L_i$ corresponds to the orbits of a lift of $A_i$. And vice versa. We call $C_e$ a \emph{photo} of $\Sigma$, see Figure \ref{f.Bphoto}.
\end{lemma}

  \begin{figure}[htp]
\begin{center}
  \includegraphics[totalheight=4.2cm]{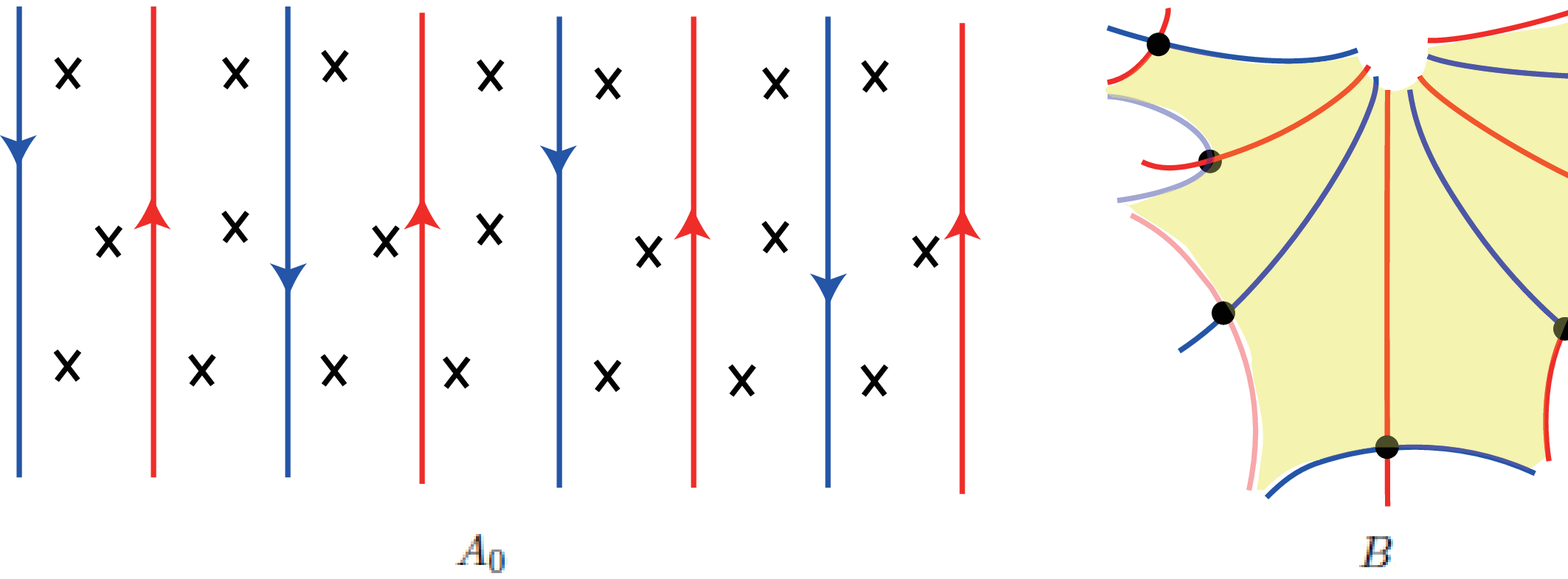}\\
  \caption{The photo of a $7$-SA annulus}\label{f.Bphoto}
\end{center}
\end{figure}

\begin{remark}
\begin{enumerate}
  \item The red (resp. blue) edge-adjacent separatrix in $C_e$ is associated to  an unstable (resp. a stable) separatrix.
  \item  The red  (resp. blue) periodic orbit in the interior of $\Sigma$ is  associated to  an unstable (resp. a stable) separatrix in the orbit space, and
the  colors of the  periodic orbits in $\Sigma$ are defined alternatively.  
\end{enumerate}
\end{remark}

Recall that $W$ is oriented with the orientation induced by the orientation on $U$ (see   Section \ref{ss.orientation} for the definition of the orientation on $U$).
The orientation of $W$ naturally induces an orientation on $M$. From now on, we fix this orientation on $M$. 
Let  $\Sigma= A_1\cup A_2\cup \dots \cup A_k$ be a $k$-SA annulus of $Z_t^m$,
and $q$ be a point in a boundary periodic orbit $\beta$ of $\Sigma$. We define a frame $(\vec{e}_h (q), \vec{e}_p (q), \vec{e}_{in} (q))$ at $q$ (see Figure \ref{f.TEBtyp}) such that:
\begin{enumerate}
    \item $\vec{e}_h (q)$ is vertical to $\beta$, tangent to $\Sigma$ and points outwards to $\Sigma$;
  \item $\vec{e}_p (q)$ defines the same orientation on $\beta$;
  \item $\vec{e}_{in} (q)$ is transverse to $\Sigma$ and is coherent to the orientations of the flowlines intersecting 
  the Birkhoff annulus adjacent to $\beta$.
\end{enumerate}

 \begin{figure}[htp]
\begin{center}
  \includegraphics[totalheight=5cm]{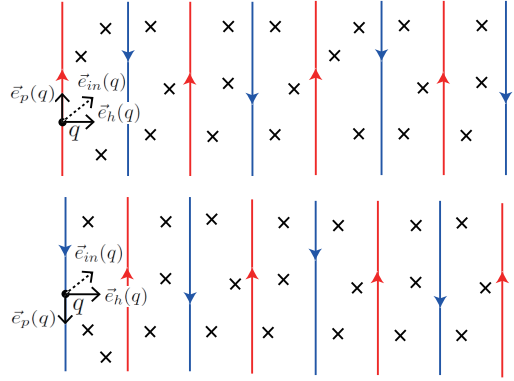}\\
  \caption{Up:  R-type $7$-SA annulus; Down: L-type $7$-SA annulus}\label{f.TEBtyp}
\end{center}
\end{figure}

When $k$ is odd, one can easily observe that

\begin{lemma}\label{l.frcoor}
Let  $\Sigma= A_1\cup A_2\cup \dots \cup A_k$ be a $k$-SA annulus of $Z_t^m$.
If $k$ is odd, for every two points 
 $q_1$ and $q_2$  in the two boundary components of $\Sigma$, 
the two frames $(\vec{e}_h (q_1), \vec{e}_p (q_1), \vec{e}_{in} (q_1))$  and $(\vec{e}_h (q_2), \vec{e}_p (q_2), \vec{e}_{in} (q_2))$ define the same orientation on $M$.
\end{lemma}

In particular, Lemma \ref{l.frcoor} permits us to define the following orientation on a $(4i+3)$-SA annulus.

\begin{definition}\label{d.RLtype}
Let $\Sigma$ be a $(4i+3)$-SA ($i=1,\dots, 4n$) annulus of $Z_t^m$ in $M$. Let $q\in \partial \Sigma$, then we say that $\Sigma$ is (see Figure \ref{f.TEBtyp})
\begin{enumerate}
  \item \emph{$R$-type} if $(\vec{e}_h (q), \vec{e}_p (q), \vec{e}_{in} (q))$  
defines an orientation on $M$ that is reverse to the orientation induced by  $(\vec{e}_1, \vec{e}_2, \vec{e}_3)$ in Section \ref{ss.orientation}. See Figure \ref{f.orientation} for the orientation on $U$ (which induces the orientation on $M$);
  \item \emph{$L$-type} if $( \vec{e}_h (q), \vec{e}_p (q),\vec{e}_{in} (q))$   defines the same orientation on $M$ induced by  $(\vec{e}_1, \vec{e}_2, \vec{e}_3)$.
\end{enumerate}
\end{definition}

\begin{example}\label{e.LRtype}
Let $\Sigma_i$ ($i=1,\dots, 4n$) be an old $(4i+3)$-SA annulus of $(M,Z_t^m)$ associated to $T_i$. 
Due to item 2 of Theorem \ref{t.con}, the periodic orbit $\alpha_j^m$ ($j=1,\dots, 2n$)
of $Y_t^m$ satisfies that:
\begin{enumerate}
\item $\alpha_j^m$ intersects the torus $T_{2j-1}$ at a point  in $\pi^m (A_{2j-1}^{s,1}) \cap \pi^m (A_{2j-1}^{u,1})$; 
\item $\alpha_j^m$ intersects the torus $T_{2j}$ at a point  in $\pi^m (A_{2j}^{s,1}) \cap \pi^m (A_{2j}^{u,1})$.
\end{enumerate}
See Figure \ref{f.1} as an illustration. Then 
due to the construction of $Z_t^m$ (Section \ref{s.conh}), we know that
\begin{enumerate}
\item  if $j\leq m$, $\Sigma_{2j-1}$ is L-type and $\Sigma_{2j}$ is R-type; 
\item   if $j>m$, $\Sigma_{2j-1}$ is R-type and $\Sigma_{2j}$ is L-type.
\end{enumerate}
\end{example}

We suppose that $h: M\to M$ is an orbit-preserving homeomorphism between $Z_t^{m_1}$ and $ Z_t^{m_2}$ where $m_1<m_2 \in \{0, 1,\dots, 2n\}$. The main result of this section is

\begin{proposition}\label{t.hoTEB}
Let $\Sigma =A_1 \cup A_2 \cup \dots \cup A_{4i+3}$  ($i=1,\dots,4n$) be an old 
$(4i+3)$-SA annulus of $Z_t^{m_1}$ associated to $T_i$, then we have:
\begin{enumerate}
\item If $h$ is orientation-preserving, then $h(\Sigma)$ is the unique old $(4i+3)$-SA annulus of $(M, Z_t^{m_2})$.
\item If $h$ is orientation-reversing and $i\in (m_1,  m_2]$, then $h(\Sigma)$ is the unique old $(4i+3)$-SA annulus of $Z_t^{m_2}$.
\item  If $h$ is orientation-reversing and  $i\in (0,m_1] \cup (m_2,2n]$, then $h(\Sigma)$ can be extended to a $(4i+4)$-SA annulus $\Sigma'$ of $ Z_t^{m_1}$, and moreover 
$h(\Sigma)$ is the union of an old $(4i+2)$-SA annulus of $Z_t^{m_2}$ associated to $T_i$
and a new fundamental Birkhoff annulus.
  \end{enumerate}
 \end{proposition}

To prove Proposition \ref{t.hoTEB}, we need more preparations. Firstly, we have 

 \begin{figure}[htp]
\begin{center}
  \includegraphics[totalheight=4.2cm]{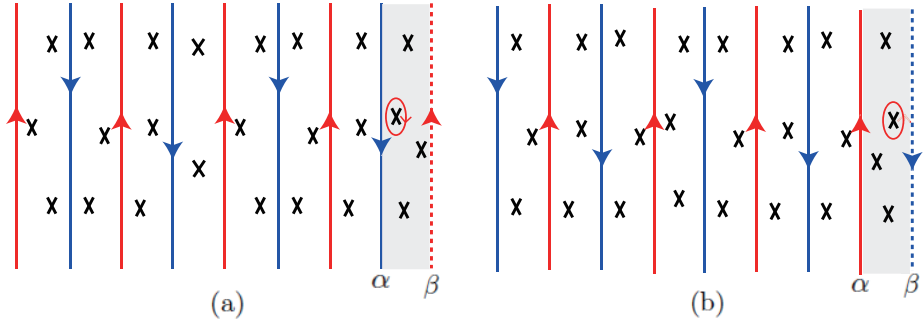}\\
  \caption{$(4i+4)$-SA annulus extended from an old $(4i+3)$-SA annulus}\label{f.ETEB}
\end{center}
\end{figure}

\begin{lemma}\label{l.no4EB}
\begin{enumerate}
  \item If $k>0$, there does not exist any $(4i+4)$-SA annulus $\Sigma'$ which is an extension of an $R$-type old $(4i+3)$-SA annulus $\Sigma$.
  \item If $k<0$, there does not exist any $(4i+4)$-SA annulus $\Sigma'$ which is an extension of an $L$-type old $(4i+3)$-SA annulus $\Sigma$.
\end{enumerate}
 \end{lemma}
 \begin{proof}
 We assume by contradiction the existences of $\Sigma$ and $\Sigma'$.
Suppose that $\alpha$ and $\beta$ bound the new Birkhoff annulus $A_N$ in $\Sigma'$ and
 $\alpha\subset \partial \Sigma$.   We assume that
$\alpha_j $ intersects  $A_N $ $s^j$  times,  then $\Sigma_{j} s^j >0$. 

 First we prove item $1$. See (a) of Figure \ref{f.ETEB},   in $H_1 (W)$, $$-[\beta]=[\alpha]+\Sigma_j s^j [\mu_j +k\lambda_j] = [\alpha]+k\Sigma_j s^j [\alpha_j].$$ Since $\Sigma_{j} s^j >0$, there exists some $T_{2j_0}$ such that $Int(\alpha_{j_0}, T_{2j_0})>0$ in $W$.
 Notice that $Int(\alpha, T_{2j_0})=0$, $Int(\alpha_j, T_{2j_0})\geq0$ and $Int(\beta, T_{2j_0})\geq0$. Then on one hand $Int(-\beta, T_{2j_0})\leq 0$, and on the other hand, $$Int(-\beta, T_{2j_0})= Int(\alpha, T_{2j_0})+ k \Sigma_j Int(\alpha_j, T_{2j})\geq k Int(\alpha_{j_0}, T_{2j_0}) >0$$ since $k >0$ and $Int(\alpha_{j_0}, T_{2j_0})>0$. Thus we get a contradiction, and item 1 is proved.

 The proof of item 2 is similar. In this case, see (b) of Figure \ref{f.ETEB}, following the orientations, $$Int(\beta, T_{2j_0})= Int(-\alpha, T_{2j_0})+ k \Sigma_j Int(\alpha_j, T_{2j_0})\leq k Int(\alpha_{j_0}, T_{2j_0}) <0.$$ This also induces a contradiction, and item 2 is proved. Therefore, the proof of the lemma is complete.
 \end{proof}


Depending on Propsotion \ref{p.EAP} (see Propsotion \ref{p.EAP} and Figure \ref{f.D}), there are exactly four possible maximal fan-type clusters $C_i^m$, $C_i^{u,m}$ , $C_i^{s,m}$ and $C_i^{us,m}$ in $\cO_m$ such that the number of lozenges no less than $3$, where
\begin{enumerate}
\item $C_i^m$ is an old fan-type cluster with $4i+3$ ($i=1,\dots,4n$) lozenges;
\item $C_i^{u,m}$ consists of $4i+4$ lozenges and is extended from  $C_i^m$ along an unstable 
separatrix in the boundary of $C_i^m$;
\item $C_i^{s,m}$ consists of $4i+4$ lozenges and is extended from  $C_i^m$ along a stable 
separatrix in the boundary of $C_i^m$;
\item $C_i^{us,m}$ consists of $4i+5$ lozenges and is extended from  $C_i^m$ along an unstable separatrix and a stable separatrix in the boundary of $C_i^m$. 
\end{enumerate}
 
See Figure \ref{f.D1D2} for $C_i^m$, $C_i^{u,m}$, $C_i^{s,m}$ and $C_i^{us,m}$.
  \begin{figure}[htp]
\begin{center}
  \includegraphics[totalheight=3.13cm]{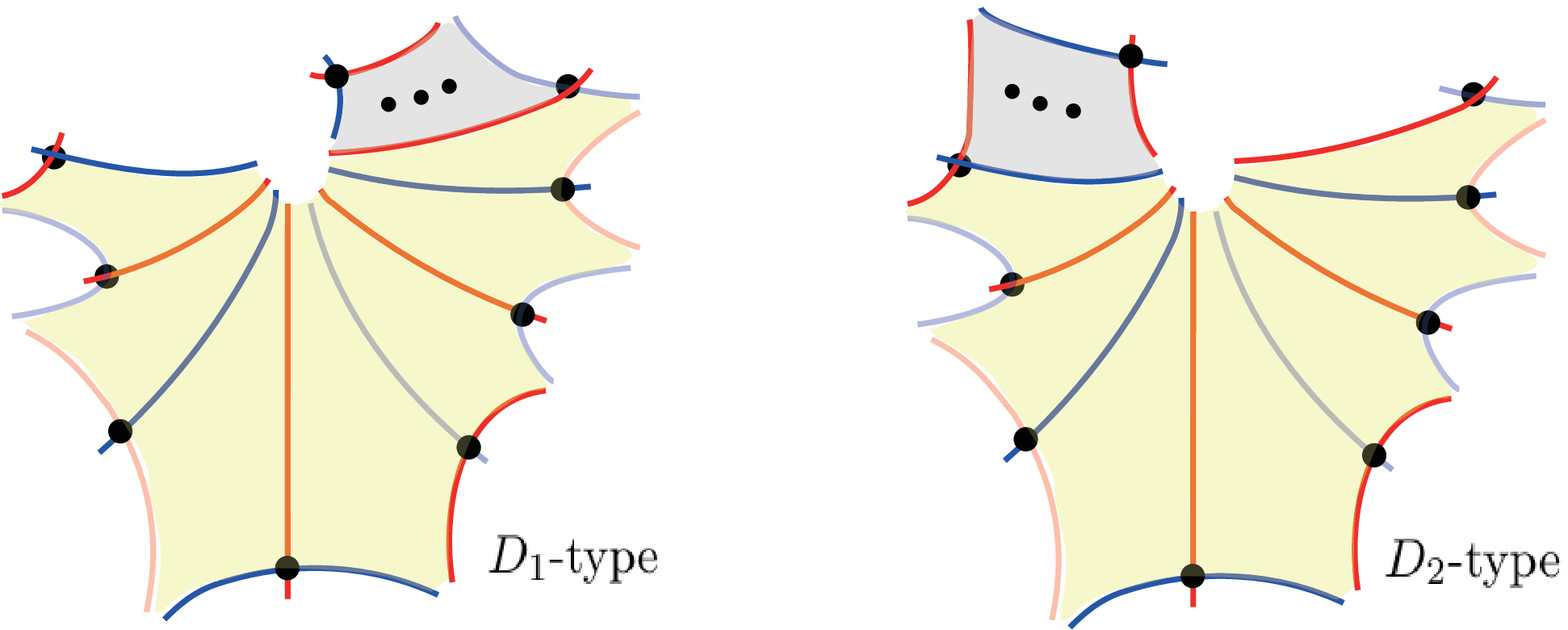}\\
  \caption{From left to right: $C_i^m$, $C_i^{u,m}$, $C_i^{s,m}$ and $C_i^{us,m}$}\label{f.D1D2}
\end{center}
\end{figure}

We further assume that $\cO_{m_1}$ and $\cO_{m_2}$ are the orbit spaces of $Z_t^{m_1}$ and $Z_t^{m_2}$ respectively, and $\overline{h}: \cO_{m_1} \to \cO_{m_2}$ is a homeomorphism commutative with $\pi_1 (M)$ which is induced by $h$. 
Let $C_e^{m_1}$ be a maximal cluster of $\cO_{m_1}$, then
$C_e^{m_2}=\overline{h}(C_e^{m_1})$  is also a maximal cluster of the orbit space $\cO_{m_2}$. Moreover, suppose $C_i^{m_1}$ and $C_j^{m_2}$ are the two old fan-type clusters in $C_e^{m_1}$ and $C_e^{m_2}$ respectively. 
The action of $\overline{h}$ on  an old fan-type cluster  can be well understood:
\begin{lemma}\label{l.hEAP}
Both $C_i^{m_1}$ and $C_j^{m_2}$ consist of $4i+3$ lozenges for the same $i\in \{1,\dots, 4n\}$, i.e. $C_j^{m_2}=C_i^{m_2}$. $\overline{h}$ should be one of the following two cases:
\begin{enumerate}
  \item $\overline{h}(C_i^{m_1})=C_i^{m_2}$.
  \item $C_i^{m_1}\subset C_i^{u,m_1}$ (or $C_i^{m_1}\subset C_i^{s,m_1}$) and $\overline{h} (C_i^{u,m_1})= C_i^{u,m_2}$ (resp. $\overline{h} (C_i^{s,m_1})= C_i^{s,m_2}$).
Moreover, $\overline{h}(C_i^{m_1})$ is the union of $4i+2$ old lozenges of $C_i^{m_2}$ and the new lozenge of $C_i^{u,m_2}$ (resp. $C_i^{s,m_2}$). 
\end{enumerate}
\end{lemma}
\begin{proof}
Suppose that $C_i^{m_1}$  consists of $4i+3$ lozenges for some $i\in \{1,\dots, 4n\}$ and $C_e^{m_1}$ is the maximal fan-type cluster in $\cO_{m_1}$  that contains $C_i^{m_1}$. Then due to the previous discussion,  the number of the lozenges in  $C_e^{m_1}$ belongs to $\{4i+3, 4i+4, 4i+5\}$. Observe that $C_e^{m_2}=\overline{h}(C_e^{m_1})$ must also be a maximal fan-type cluster in $\cO_{m_2}$. Due to these facts and the above descriptions of maximal fan-type clusters, 
 one can easily get that the old fan-type cluster $C_j^{m_2}$ in $C_e^{m_2}$ also consists of $4i+3$ lozenges.

Now we only need to prove that if $\overline{h}(C_i^{m_1})\neq C_i^{m_2}$, then $\overline{h}$ must satisfy the conclusion in item $2$ of the lemma. From now on, we assume that $\overline{h}(C_i^{m_1})\neq C_i^{m_2}$.

As an illustration of the discussion below, we refer the reader to see Figure \ref{f.D1D2}. Firstly, $C_e^{m_1} \neq C_i^{m_1}$ since otherwise $\overline{h}(C_i^{m_1})= C_i^{m_2}$ that contradicts our hypothesis. Now we prove that $C_e^{m_1} \neq C_i^{us,m_1}$. We assume by contradiction that  $C_e^{m_1} = C_i^{us,m_1}$. Due to the fact that there are only four possible fan-type clusters with the number of lozenges
no less than $3$ and $\overline{h}$ must preserve maximal   fan-type clusters, then in this case $\overline{h}(C_e^{m_1})= C_i^{us,m_2}$ and therefore 
$\overline{h}(C_i^{m_1})= C_i^{m_2}$ since $C_i^{m_j}$ ($j=1,2$) is in the middle of $C_e^{m_j}$. This also contradicts our hypothesis. Now we claim that 
\begin{enumerate}
  \item if $C_e^{m_1}= C_i^{u,m_1}$, then  $\overline{h}(C_e^{m_1})= C_i^{u,m_2}$;
\item if $C_e^{m_1}= C_i^{s,m_1}$, then  $\overline{h}(C_e^{m_1})= C_i^{s,m_2}$.
\end{enumerate}
Without loss of generality, we only need to prove the first case. We assume by contradiction  that 
$C_e^{m_1}= C_i^{u,m_1}$ and  $\overline{h}(C_e^{m_1})\neq C_i^{u,m_2}$. Then because of the same reason as before, 
$\overline{h}(C_e^{m_1})= C_i^{s,m_2}$. Further observe that both of the two adjacent-edges in $C_i^{u,m_1}$ adjacent to the two end lozenges of $C_i^{u,m_1}$ correspond to two unstable separatrices, while   both of the two adjacent-edges in $C_i^{s,m_2}$ adjacent to the two end lozenges of $C_i^{s,m_2}$ correspond to two stable separatrices. Then $\overline{h}(C_i^{u,m_1})\neq C_i^{s,m_2}$ since $\overline{h}$ must preserve  the types of  adjacent-edges. We get a contradiction. 

From now on, we only need to discuss the two cases of the claim. Futhermore without loss 
of generality, we only need to consider the first case of the claim: 
$C_e^{m_1}= C_i^{u,m_1}$ and  $\overline{h}(C_e^{m_1})= C_i^{u,m_2}$. Notice that:
\begin{enumerate}
\item $C_i^{m_j}\subset C_i^{u,m_j}$ ($j=1,2$) and each of $C_i^{m_j}$ and $C_i^{u,m_j}$ is a fan-type cluster;
\item $C_i^{m_j}$ consists of $4i+3$ lozenges and $C_i^{u,m_j}$ consists of $4i+4$ lozenges;
\item $\overline{h}(C_i^{m_1})\neq C_i^{m_2}$.
\end{enumerate}
Then one can easily observe that $\overline{h}(C_i^{m_1})$ is the union of 
$4i+2$ old lozenges of $C_i^{m_2}$ and the new lozenge in $C_i^{u,m_2}$. The proof of item 2 of the lemma is complete.
\end{proof}

\begin{proof}[Proof of Proposition \ref{t.hoTEB}]
By Lemma \ref{l.hEAP}, there are two cases for $\overline{h} (C_i^{m_1})$.
The first case is that $\overline{h} (C_i^{m_1})=C_i^{m_2}$. This means that a lift connected component of $h(\Sigma)$ corresponds to $C_i^{m_2}$ in $\cO_2$. Then
due to the natural correspondence between lozenge and fundamental Birkhoff annulus (see Lemma \ref{l.photoB} and Lemma \ref{l.photo}), we have that $h(\Sigma)$ is the (unique)
old $(4i+3)$-SA annulus of $Z_t^{m_2}$.

Now we prove item 1 of the proposition. Due to the the above discussion, we only need
to prove that if $h$ is orientation-preserving, then the second case of Lemma \ref{l.hEAP} for $\overline{h} (C_i^{m_1})$ does not happen. Due to Lemma \ref{l.hEAP}, without loss of generality, we can assume by contradiction
that  $C_i^{m_1}\subset C_i^{u,m_1}$, $\overline{h} (C_i^{u,m_1})= C_i^{u,m_2}$ and $\overline{h}(C_i^{m_1})$ is the union of $4i+2$ old lozenges of $C_i^{m_2}$ and the new lozenge of $C_i^{u,m_2}$. By Lemma \ref{l.photo},
 there exists a $(4i+4)$-SA annulus $\Sigma'$ of $Z_t^{m_1}$ that is an extension of the  old $(4i+3)$-SA annulus $\Sigma$ and  $h(\Sigma)$ is not an old $(4i+3)$-SA annulus of $ Z_t^{m_2}$. Notice that: 
\begin{enumerate}
\item Since $h$ is orientation-preserving, we have that if $\Sigma$ is R-type (or L-type), then $h(\Sigma)$ is also R-type (resp. L-type).
\item $h(\Sigma)\subset h(\Sigma')$ where $h(\Sigma')$ is a $(4i+4)$-SA annulus of $Z_t^{m_2}$ that is the union of the old $(4i+3)$-SA annulus $\Sigma_2$ and a new fundamental Birkhoff annulus.
\item Due to Lemma \ref{l.hEAP} and  Lemma \ref{l.photo}, $h(\Sigma)$ is the union of an old $(4i+2)$-SA annulus in $\Sigma_2$ and the new fundamental Birkhoff annulus.
\end{enumerate}
Then one can routinely check that $\Sigma_2$ is L-type (or R-type). But by  Lemma \ref{l.no4EB},  it can not happen that there exsit two $(4i+4)$-SA annuli of $Z_t^{m_1}$
and $Z_t^{m_2}$ respectively such that one is extended from an L-type old $(4i+3)$-SA annulus and the other is  extended from an R-type old $(4i+3)$-SA annulus. We get a contradiction, and therefore item 1 of the proposition is proved.

Now we turn to prove item $2$. In this case $i\in (m_1, m_2]$, by our construction  the old  $(4i+3)$-SA annulus $\Sigma$  of $Z_t^{m_1}$ associated to $T_i$ and the old $(4i+3)$-SA annulus $\Sigma_2$ of $Z_t^{m_2}$ associated to $T_i$ are  two different types (see Example \ref{e.LRtype}). Similar to the discussion before, further combined with  Lemma \ref{l.hEAP} and  Lemma \ref{l.photo}, we have that if item 2 of Lemma \ref{l.hEAP}  happens, then $\Sigma$ and $h(\Sigma)$  are in the same type. This contradicts the fact that $h$ is orientation-reversing. And then the only possibility is that item 1 of Lemma \ref{l.hEAP} happens, which means that $h(\Sigma)$ is the old $(4i+3)$-SA annulus of $Z_t^{m_2}$. Item 2 of the proposition is proved.

The proof of item 3 of the proposition is very similar to the proof of item 2, so we just give a brief proof here. 
Similarly, when $i\in (0,m_1] \cup (m_2,2n]$,  by our construction the old  $(4i+3)$-SA annulus $\Sigma$  of $Z_t^{m_1}$ associated to $T_i$ and the old $(4i+3)$-SA annulus $\Sigma_2$ of $Z_t^{m_2}$ associated to $T_i$ are  the same type. Then item 1 of Lemma \ref{l.hEAP} never happens since $h$ is orientation-reversing.  This implies that  this case corresponds to item 2 of Lemma \ref{l.hEAP}, which means that $\Sigma$ is contained in a $(4i+4)$-SA annulus $\Sigma'$ of $Z_t^{m_1}$.  Item $3$ is proved, and therefore the proof of the proposition is complete.
\end{proof}

%% file: Proof-main-theorem.tex
\begin{proof}[The proof of Theorem \ref{t.main}]
For every $n\in \NN$, for some $k$ with $|k|$ big enough, $Z_t^{1}, \dots, Z_t^{2n-1}$ are $2n-1$ ($2n-1\geq n$) Anosov flows on $M$. By Proposition \ref{p.hyperbolic}, $M$ is a hyperbolic $3$-manifold. Let $h: M\to M$ be an orbit-preserving homeomorphism between the two Anosov flows $Z_t^{m_1}$ and $ Z_t^{m_2}$ ($m_1<m_2 \in \{1,\dots, 2n-1\}$). 

If $h$ is orientation-preserving, by item $1$
of Proposition \ref{t.hoTEB}, up to isotopy along flowlines, $h$ maps an old $(4i+3)$-SA annulus  of  $Z_t^{m_1}$ associated to $T_i$ to the corresponding unique old   $(4i+3)$-SA annlus of  $Z_t^{m_2}$ associated to $T_i$. Since $h$ is orientation-preserving, $h$ preserves the types of the corresponding old $(4i+3)$-SA annuli. By our construction (see Example \ref{e.LRtype}), for every $i\in (m_1, m_2]$,  the two
$(4i+3)$-SA annuli of  $Z_t^{m_1}$ and $ Z_t^{m_2}$  are not co-oriented, i.e. one is L-type and the other is R-type. 
But $h$ is orientation-preserving, then
we get a contradiction. This means that there does not exist an orientation-prserving homeomorphism $h:M\to M$ which preserves the orbits of $Z_t^{m_1}$ and $ Z_t^{m_2}$.

Now we suppose that $h$ is orientation-reversing. Since $m_1<m_2\in \{1,\dots, 2n-1\}$, then by Example \ref{e.LRtype} once more, both of the two old $7$-SA annuli of $Z_t^{m_1}$ and $Z_t^{m_2}$ associated to $T_1$ are L-type, and both of the two old $(16n-1)$-SA annuli  of $Z_t^{m_1}$ and $Z_t^{m_2}$ associated to $T_{4n-1}$ are R-type. Let $\Sigma_1$ be the old $7$-SA annulus  of $Z_t^{m_1}$ associated to $T_1$  and $\Sigma_{4n-1}$  be the old $(16n-1)$-SA annulus  of $Z_t^{m_1}$ associated to $T_{4n-1}$. Certainly $\Sigma_1$ is L-type and $\Sigma_{4n-1}$ is R-type.   By item $3$ of  Proposition \ref{t.hoTEB}, $\Sigma_1$ can be extended to a $8$-SA annulus of 
$Z_t^{m_1}$ and $\Sigma_{4n-1}$ can be extended to a $16n$-SA annulus of 
$Z_t^{m_1}$. This is impossible due to Lemma \ref{l.no4EB} and the fact that $\Sigma_1$ is an R-type old $7$-SA annulus and $\Sigma_{4n-1}$ is an L-type old $(16n-1)$-SA annulus.

Then the  orbit-preserving homeomorphism  $h$ does not exist whenever
$h$  is orientation-preserving or not. 
This means that if $m_1 <m_2 \in \{1,\dots, 2n-1\}$, $Z_t^{m_1}$ and $Z_t^{m_2}$ are not orbitally equivalent.  The proof of Theorem \ref{t.main} is complete.
\end{proof}

%% file: ManyAnosovFlows.bbl
\begin{thebibliography}{MM}
\bibitem[Ano]{Ano}
Anosov, D. V.
\emph{Geodesic flows on closed Riemannian manifolds of negative curvature.}
Trudy Mat. Inst. Steklov.  90  (1967).
\bibitem[Ba]{Ba}
Thierry Barbot. \emph{Caract\'erisation des flots d'Anosov en dimension 3 par leurs feuilletages faibles.} Ergodic Theory Dynam. Systems 15 (1995), no. 2, 247-270.
\bibitem[Ba1]{Ba1}
Barbot, Thierry. \emph{Generalizations of the Bonatti-Langevin example of Anosov flow and their classification up to topological equivalence.} Comm. Anal. Geom. 6 (1998), no. 4, 749–798.
\bibitem[Ba2]{Ba2}
 Barbot, Thierry. \emph{Mise en position optimale de tores par rapport à un flot d'Anosov.} (French) [Optimal positioning of tori with respect to an Anosov flow] Comment. Math. Helv. 70 (1995), no. 1, 113–160.
\bibitem[Ba3]{Ba3}
Barbot, Thierry.    \emph{Flots d'Anosov sur les variétés graphées au sens de Waldhausen.} (French) [Anosov flows on graph manifolds in the sense of Waldhausen] Ann. Inst. Fourier (Grenoble) 46 (1996), no. 5, 1451–1517. 
\bibitem[Bart]{Bart} Thomas Barthelme. \emph{School on contemporary dynamical systems: Anosov
flows in dimension $3$.} 2017.
\bibitem[BaFe]{BaFe} Thomas Barthelm\'e;  Sergio Fenley.  \emph{Counting periodic orbits of Anosov flows in free homotopy classes.} Comment. Math. Helv. 92 (2017), no. 4, 641-714.
 \bibitem[BarFe1]{BarFe1} Thierry Barbot; Sergio Fenley. \emph{Classification and rigidity of totally periodic pseudo-Anosov flows in graph manifolds.} Ergodic Theory Dynam. Systems 35 (2015), no. 6, 1681-1722.
 \bibitem[BarFe2]{BarFe2}Thierry Barbot; Sergio Fenley. \emph{Free Seifert pieces of pseudo-Anosov flows.} Geom. Topol. 25 (2021), no. 3, 1331-1440.
\bibitem[BBY]{BBY}  Christian, Bonatti; Francois, Beguin and Bin, Yu.  \emph{Building Anosov flows on 3-manifolds}  Geom. Topol.  21  (2017),  no. 3, 1837-1930.
\bibitem[BI]{BI}    Christian, Bonatti; Ioannis, Iakovoglou. \emph{Anosov flows on 3-manifolds: the surgeries and the foliations}. Arxiv: 2007.11518v1
 \bibitem[BM]{BM}   Bowden, J; Mann, K. \emph{$C_0$ stability of boundary actions and inequivalent
Anosov flows.} Ann. Sci. Éc. Norm. Supér. (4) 55 (2022), no. 4, 1003–1046.
\bibitem[Fen1]{Fen1} 
Sergio, Fenley.  \emph{Anosov flows in 3-manifolds.}  Ann. of Math. (2) 139 (1994), no. 1, 79-115.
 \bibitem[Fen2]{Fen2}  Sergio, Fenley.  \emph{The structure of branching in Anosov flows of 3-manifolds.}
 Comment. Math. Helv. 92 (2017), no. 4, 641-714.
\bibitem[Fen3]{Fen3}  Fenley  Sergio.
\emph{non-$\RR$-covered Anosov flows in hyperbolic 3-manifolds are quasigeodesic.}
arXiv:2210.09238
\bibitem[Fen4]{Fen4} Fenley, S\'ergio R. \emph{Quasigeodesic Anosov flows and homotopic properties of  flow lines.}  J. Differential Geom. 41 (1995), no. 2, 479–514.
\bibitem[FH]{FH} Foulon, Patrick; Hasselblatt, Boris.  \emph{Contact Anosov flows on hyperbolic 3-manifolds.}  Geom. Topol. 17 (2013), no. 2, 1225–1252.
\bibitem[JG]{JG}Boju, Jiang; Jianhan, Guo.
\emph{Fixed points of surface diffeomorphisms.}  Pacific J. Math. 160 (1993), no. 1, 67–89.
\bibitem[Kir]{Kir}Kirby Bob. \emph{Problems in low dimensional manifold theory.} In Algebraic and geometric topology (Proc. Sympos. Pure Math., Stanford Univ., Stanford, Calif., 1976), Part 2, Proc.
Sympos. Pure Math., XXXII, pages 273–312. Amer. Math. Soc., Providence, R.I., 1978
 \bibitem[Mo]{Mo} Mosher, Lee. \emph{Dynamical systems and the homology norm of a  3 -manifold. I. Efficient intersection of surfaces and flows} Duke Math. J. 65 (1992), no. 3, 449–500.
\bibitem[MS]{MS} Masur, H;  Smillie, J. \emph{Quadratic differentials with prescribed singularities and pseudo-Anosov diffeomorphisms.} Comment. Math. Helv. 68 (1993) 289–307
 \bibitem[Sha]{Sha} Shannon, M. \emph{Dehn surgeries and smooth structures on 3-dimensional transitive
Anosov flows.} Thesis (Ph.D.)-University of Burgundy. 2020. 189 PP.
\bibitem[Thu]{Thu} Thurston, Williams. \emph{Three-dimensional manifolds, Kleinian groups and hyperbolic geometry}.
Bull. Amer. Math. Soc. 6 (1982) 357-381

\end{thebibliography}
